\documentclass[a4paper, 12pt]{article}
\usepackage[utf8]{inputenc}
\usepackage[T2A]{fontenc}
\usepackage{amsmath,amssymb,amsthm, amscd}
\usepackage{mathtools}
\usepackage{authblk}
\usepackage{graphicx}
\usepackage{cancel}
\graphicspath{{pictures/}}
\graphicspath{{./images/}}

\usepackage{longtable}
\usepackage{array}
\usepackage[margin=1in]{geometry}
\usepackage{appendix}
\usepackage{amsmath,amssymb}

\usepackage[backend=biber, style=numeric, sorting=none, doi=true, url=true, isbn=true, eprint=true]{biblatex}

\newcolumntype{C}[1]{>{\centering\arraybackslash}p{#1}}
\newcolumntype{L}[1]{>{\raggedright\arraybackslash}p{#1}}

\DeclareGraphicsExtensions{.pdf,.png,.jpg}
\usepackage[russian,english]{babel}
\usepackage{csquotes}
\usepackage[unicode, pdftex]{hyperref}
\hypersetup{colorlinks=true, linkcolor=blue, filecolor=magenta, urlcolor=blue}
\newtheorem{Th}{Theorem}[section]
\newtheorem{stat}{Statement}[section]
\newtheorem{prop}{Proposition}[section]
\newtheorem{lem}{Lemma}[section]
\newtheorem{rem}{Remark}[section]

\newtheorem{constr}{Construction}
\newtheorem{corollary}{Corollary}[section]
\newtheorem{quest}{Question}[section]

\begin{document}

\title{The Number of Polytope Trees and Beyond}

\author[1]{Ali Aliev\thanks{ alievali0278@gmail.com}}
\author[1,2]{Natalia Amburg\thanks{amburg@mccme.ru}}

\affil[1]{HSE University, Moscow, Russia}
\affil[2]{NRC “Kurchatov Institute”, Moscow, Russia}

\date{}
\maketitle




\begin{abstract}
We study the symmetry-weighted enumeration of complexes formed by gluing copies of a regular polytope. For polytope trees and complexes obtained by attaching such trees to a fixed initial complex, we derive cut-and-join recurrences, differential equations for generating functions, explicit counting formulas, and asymptotic estimates. We then extend the method to complexes with cycles in their dual graphs, enumerating them by the numbers of polytopes and free facets. We derive recurrence relations and partial differential equations for the corresponding bivariate generating functions. In this setting, we obtain recurrences for connected configurations and explicit coefficient formulas for possibly disconnected configurations, both with and without a distinguished initial complex.
\end{abstract}


\noindent\textbf{2020 Mathematics Subject Classification.}
\par\smallskip
\noindent\textit{Primary:} 05A15.
\par
\noindent\textit{Secondary:} 05A19, 05C30, 57Q15, 20B25, 05E45.

\tableofcontents

\section{Introduction}

Enumerating non‑isomorphic simplicial complexes and triangulations --- where objects are weighted inversely to the size of their automorphism groups, as in Burnside's Lemma and the Pólya Enumeration Theorem --- is a fundamental problem in combinatorial geometry and topology. While exact results exist for specific cases --- such as triangulations of orientable surfaces \cite{GS06, GJ08, BGR08} --- the general higher-dimensional problem remains exceptionally difficult, although partial results have been obtained in \cite{CP21, BZ11, R13}.

Beyond pure mathematics, this problem has applications in physics. For example, in string theory and quantum gravity, integrating over the space of metrics can sometimes be approximated by summing over a dense set of piecewise-flat triangulations \cite{Ambjorn09}.

In this work, we count the number of polytope trees: complexes formed from \(n\) polytopes, where the dual graph of the complex is a tree. The idea of constructing such objects inductively is inspired by random graph theory, where vertices are added sequentially and edges are generated according to a specific probability distribution (see \cite{ER60, Ra11}). In this paper, however, we set aside the probabilistic aspects and focus on the combinatorial side.

In random graph models, the initial graph can be arbitrary; analogously, we may start from any fixed initial complex and attach polytope trees to its facets. This yields what we call \textit{inductive complexes} (or \textit{inductive triangulations}). The attachment process does not change the topology of the underlying manifold. Using our exact enumerative results for polytope trees, we compute the number of non-isomorphic inductive complexes. 

In a piecewise-Euclidean realization, attaching a regular tetrahedron along a boundary triangle replaces that triangle by the other three faces, changing the boundary metric locally while preserving the topology of the surface. This gives a combinatorially controlled family of local metric modifications, providing a natural point of contact with dynamical-triangulation models of two-dimensional quantum gravity and string theory \cite{Ambjorn09}.

Moreover, we expand our theory by allowing cycles in the dual graph of complexes.

In this paper, we use the standard definitions of triangulations and simplicial complexes \cite{LZ04, Zi95}; any additional constructions will be introduced as needed.

The remainder of this paper is organized as follows: Section~\ref{sec:construction} introduces the construction. In Section~\ref{sec:R_surfaces}, we enumerate tetrahedron trees, and in Section~\ref{sec:inductive_triangulations}, we find the number of inductive triangulations. Section~\ref{sec:general_inductive} generalizes the construction to arbitrary polyhedra. Section~\ref{sec:cycles} extends the framework to complexes whose dual graphs contain cycles. Section~\ref{sec:segment} treats the one-dimensional case and provides several illustrative examples. Finally, several appendices contain additional computations and auxiliary details.

\section{Construction of Inductive Triangulations} \label{sec:construction}

We introduce the following construction of inductive triangulations.

\begin{constr}
\textbf{Base:} At the initial step (step \(0\)), we have a fixed surface triangulation \(\mathcal{T}_0\) (two examples of sphere triangulations are shown in Figure~\ref{fig:sphere_triangulations})

\begin{figure}[!htbp]
    \centering
    \includegraphics[width=0.36\textwidth]{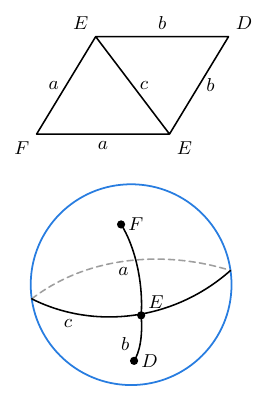}
    \hfill
    \includegraphics[width=0.36\textwidth]{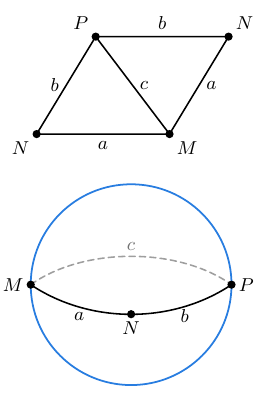}
    \caption{Two triangulations of the sphere from two triangles.}
    \label{fig:sphere_triangulations}
\end{figure}

\textbf{Inductive step:} At step \(n\), we have a triangulation \(\mathfrak{T}_n\). The triangulation \(\mathfrak{T}_{n+1}\) is constructed inductively as follows: we select a triangle from \(\mathfrak{T}_n\) and subdivide it into three triangles by drawing segments from each of its vertices to the centre of the triangle (the process is illustrated in Figure~\ref{fig:tetrahedron_tree}).
\end{constr}

For a fixed initial triangulation \(T_0\), this construction defines, at each step \(n\), the family of all triangulations obtainable from \(T_0\) by \(n\) subdivisions.

\begin{figure}[!htbp]
\centering
\includegraphics[width=0.66\textwidth]{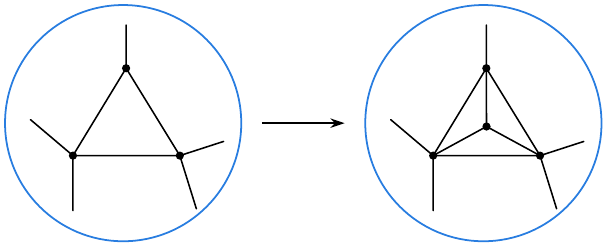}
\caption{The inductive step.}
\label{fig:tetrahedron_tree}
\end{figure}

Now, let us consider a trivial triangulation of the sphere consisting of two glued triangles, as shown on the right-hand side of Figure~\ref{fig:sphere_triangulations} (step \(0\)). We take one of these triangles and draw segments from its vertices to its centroid. Geometrically, subdividing a face in this manner is equivalent to gluing a tetrahedron to the free triangle in \(\mathbb{R}^3\) (see Figure~\ref{fig:tetrahedron2}). At step \(1\) we obtain a tetrahedron (see Figure~\ref{fig:tetrahedron1}). Inductively repeating this process, at step \(n\), the resulting triangulation is isomorphic to the complex formed by gluing \(n\) tetrahedra together along their faces. This construction is also considered in \cite{BZ11}; we refer to the resulting complex as a simplex tree.

Note that tetrahedra can be attached to any initial triangulation, as shown, for example, on the left‑hand side of Figure~\ref{fig:sphere_triangulations}. Moreover, we may consider an initial triangulation of an arbitrary manifold. We will call such triangulations \textit{inductive}. The triangulations obtained in this way are closely related to Apollonian networks and stacked polytopes; see \cite{BL22,HRS82}. We arrived at the constructions above independently, before becoming aware of these connections with earlier work. In the following sections, an explicit formula will be derived and more general constructions will be developed.
\begin{figure}[!htbp]
\centering
\includegraphics[width=0.83\textwidth]{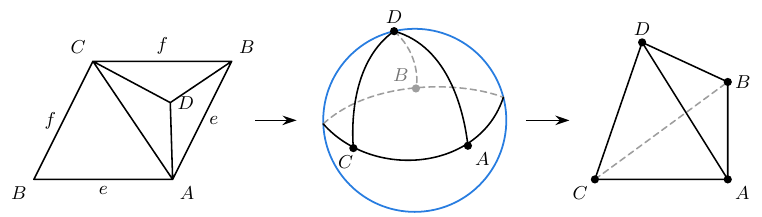}
\caption{Base of the induction for tetrahedron gluings.}
\label{fig:tetrahedron1}
\end{figure}

\begin{figure}[!htbp]
\centering
\includegraphics[width=0.58\textwidth]{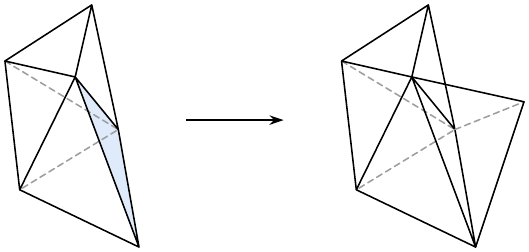}
\caption{Inductive step for tetrahedron gluings.}
\label{fig:tetrahedron2}
\end{figure}

\section{Tetrahedron Tree} \label{sec:R_surfaces}

Our primary goal is to derive an explicit formula for the number of non-isomorphic tetrahedron trees composed of \(n\) tetrahedra. Before doing so, let us explain the counting principle underlying our approach.

Let \(\mathcal{X}\) be a finite collection of objects equipped with an equivalence relation \(\sim\), where two objects are considered equivalent if they are isomorphic. Throughout this paper, unless explicitly stated otherwise, we consider the symmetry-weighted cardinality of a collection of isomorphism classes
\[
|\mathcal{X}/{\sim}|
:=
\sum_{[K]\in \mathcal{X}/{\sim}}
\frac{1}{|\operatorname{Aut}(K)|}.
\]
This weighted count naturally accounts for the varying sizes of automorphism groups and is the standard approach in enumerative combinatorics for such objects (see, e.g., \cite{BaezDolan2000, LZ04}). 

This formulation is closely related to the standard technique of introducing auxiliary structures --- such as rootings, labelings, or numberings --- which is commonly used to bypass automorphism issues (see, e.g., \cite{GJ08}) and obtain the symmetry-weighted enumeration of unmarked objects indirectly through the enumeration of suitably marked ones.

Throughout this paper, we will label the external hyperfaces of codimension 1. Let \(R_n\) denote the number of non-isomorphic labelings of the external faces of a tetrahedron tree that can be obtained at step \(n\). Gluing \(n\) tetrahedra into tree \(\Phi_{n}\) creates exactly \(n-1\) internal triangles, hence
\[
f_{2}^{\text{int}}(\Phi_{n})=n-1,\qquad f_{2}^{\text{ext}}(\Phi_{n})=4n-2(n-1)=2n+2.
\]
Let also \(\widehat{R}_n := \frac{R_n}{(2n+2)!}\).

\begin{rem}
To one unlabeled triangulation there correspond \( \frac{(2n+2)!}{\mid Aut \mid}\) labeled triangulations.
\end{rem}

\begin{stat} \label{stat:R_surfaces_recurrence}
\( R_1 = 2 \), and \( R_n \) satisfies the following recurrence relation  
\begin{equation}
(n-1)R_n = \frac{1}{2} \sum_{k=1}^{n-1} 3 \binom{2n+2}{2k+1} R_k R_{n-k}.  
\end{equation}
\end{stat}
\begin{proof}

The proof is based on a cut‑and‑join technique, similar to the recurrence relations developed in \cite{K92, AKS17}.
Fix a split \(n=k+(n-k)\) with \(1\le k\le n-1\). Every tree \(\Phi_{n}\) can be decomposed into two subtrees \(\Phi_{k}\) and \(\Phi_{n-k}\) glued together along one external triangle. There are exactly three ways (the order of the triangle's rotation group) to perform the gluing --- depending on the choice of the triangle's side along which the merging is made.  
The subtrees have \(2k+2\) and \(2(n-k)+2\) external triangles, respectively, so we need to distribute \(2k+2-1\) numbers on the figure \(\Phi_{k}\) (and accordingly \(2(n-k)+2-1\) numbers on the figure \(\Phi_{n-k}\)) (we subtract one because one label on the face along which the gluing occurs is no longer counted in \(\Phi_{n}\)). This gives us
\(\binom{2n+2}{2k+1}\) ways to do that.  
The symmetry \(k\leftrightarrow n-k\) counts each decomposition twice; so we need to divide by 2.  
Finally, every decomposition fixes a concrete seam (hyperface where cutting occurs), whereas the left-hand side of the recurrence is simply \(R_{n}\) --- the number of trees without a distinguished seam. To account for all possible seams we multiply by their total number: \(n-1\). The resulting recurrence is
\begin{equation} \label{eq:R_surfaces_recurrence}
(n-1)\,R_{n}=\frac{1}{2}\sum_{k=1}^{n-1}3\binom{2n+2}{2k+1}R_{k}R_{n-k}.  
\end{equation}
Figure~\ref{fig:R_recurrence_relation} illustrates the cut-and-join decomposition underlying the recurrence \eqref{eq:R_surfaces_recurrence}.

\begin{figure}[!htbp]
\centering
\includegraphics[scale=0.4]{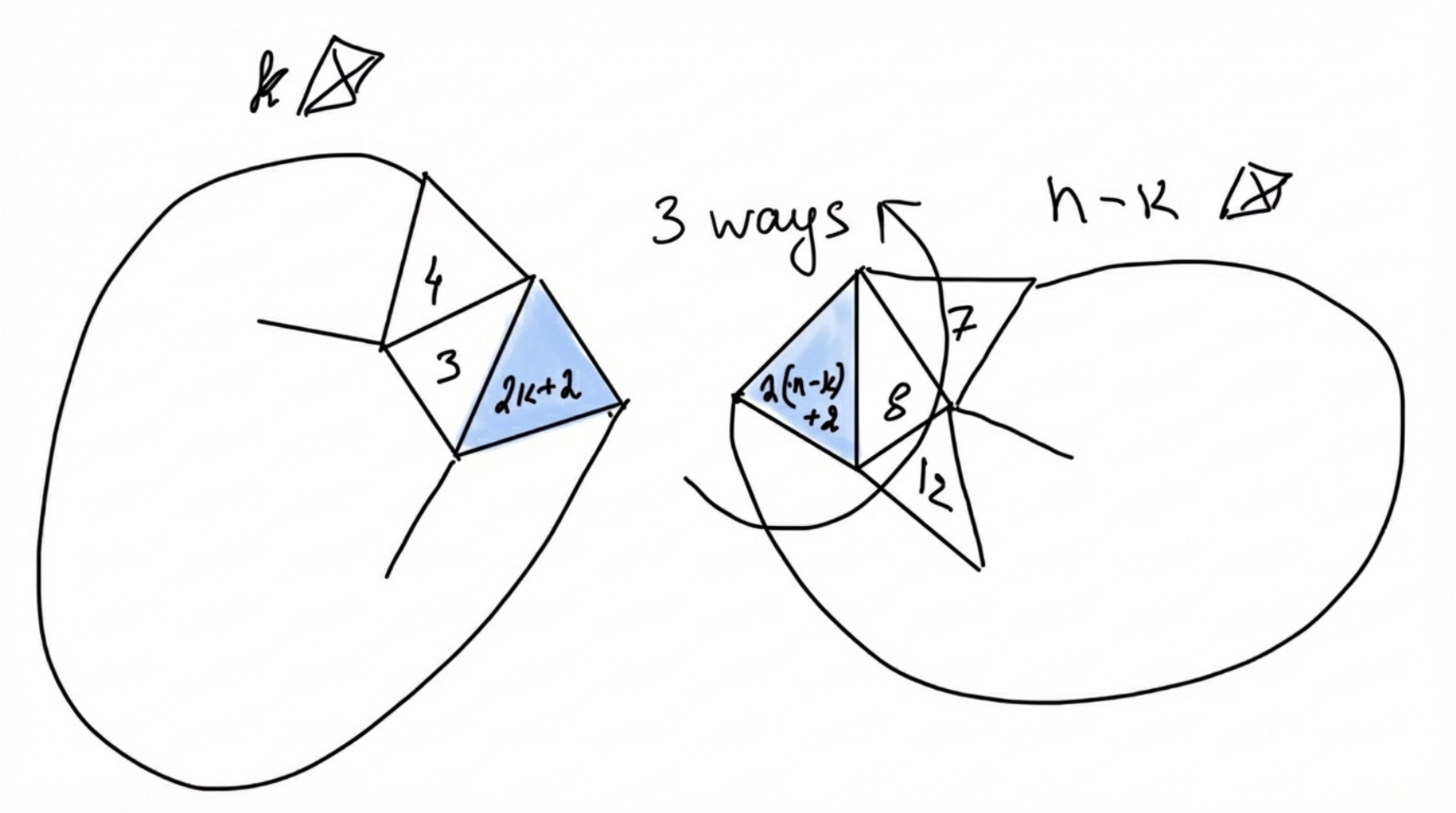}
\caption{Visual explanation for recurrence.}

\label{fig:R_recurrence_relation}

\end{figure}

It remains to find the initial condition: a tetrahedron has two numberings: one face is fixed and labeled \(1\), the remaining faces can be numbered in two ways: \(2,3,4\) and \(3,2,4\) clockwise (see Figure~\ref{fig:R_1}). Thus, \(R_{1} = 2\).

\begin{figure}[!htbp]
\centering
\includegraphics[width=0.52\textwidth]{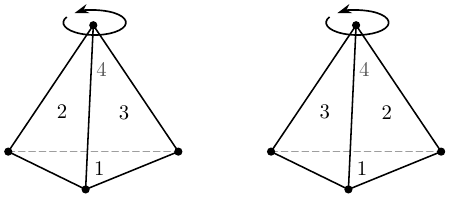}

\caption{Initial condition for tetrahedron tree numberings.}

\label{fig:R_1}

\end{figure}

\end{proof}

\begin{stat} \label{stat:R_surfaces_differential_equation}
Consider the exponential generating function \(\mathcal{R}_3(t) = \sum_{n=1}^{\infty} \dfrac{R_{n} \, t^{2n+2}}{(2n+2)!}\).
Then this function satisfies the following differential equation
\begin{equation} 
(\mathcal{R}_3')^2 = \frac{2}{3}\left(\frac{t\mathcal{R}_3'}{2}-2\mathcal{R}_3\right).
\end{equation}
\end{stat}

\begin{proof}
Let us rewrite the derived recurrence relation \eqref{eq:R_surfaces_recurrence} in terms of generating functions. Let \(\mathcal{R}_3(t) = \sum_{n=1}^{\infty} \dfrac{R_{n} \, t^{2n+2}}{(2n+2)!}\) and \(\mathcal{R}_3'(t) = \sum_{n=1}^{\infty} \dfrac{R_{n} \, t^{2n+1}}{(2n+1)!}\).
Squaring the derivative yields
\begin{equation*}
\begin{split}
(\mathcal{R}_3'(t))^2 
&= \sum_{n=2}^{\infty} \sum_{k=1}^{n-1} 
\frac{R_{k} \, t^{2k+1}}{(2k+1)!} \cdot \frac{R_{n-k} \, t^{2(n-k)+1}}{(2(n-k)+1)!} 
= \sum_{n=2}^{\infty} \sum_{k=1}^{n-1} 
\frac{\binom{2n+2}{2k+1} \, t^{2n+2}}{(2n+2)!} \, R_{k} R_{n-k}.
\end{split}
\end{equation*}
Using the recurrence \eqref{eq:R_surfaces_recurrence}, this simplifies to
\begin{align*}
(\mathcal{R}_3'(t))^2 
    &= \sum_{n=2}^{\infty} \frac{2(n-1) R_{n} \, t^{2n+2}}{3(2n+2)!} = \frac{2}{3} \sum_{n=1}^{\infty} \frac{(n-1) R_{n} \, t^{2n+2}}{(2n+2)!}\\
    &= \frac{2}{3} \left[ \sum_{n=1}^{\infty} \frac{(n+1) R_{n} \, t^{2n+2}}{(2n+2)!} 
                        - 2\sum_{n=1}^{\infty} \frac{R_{n} \, t^{2n+2}}{(2n+2)!} \right] \\
    &= \frac{2}{3} \left[ \frac{t\mathcal{R}_3'(t)}{2} - 2\mathcal{R}_3(t) \right],
\end{align*}
which establishes the claimed differential equation. This equation is a Lagrange (or Clairaut) equation \cite[Equation 2.1.5.5]{PZ03}. Let us also rewrite it in another form
\begin{equation} \label{eq:surfaces_differential_equation}
\mathcal{R}_3(t) = \frac{1}{4} \, t \, \mathcal{R}_3'(t) - \frac{3}{4} \, (\mathcal{R}_3'(t))^2.
\end{equation}
\end{proof}

\begin{Th} \label{thm:R_surfaces_solution}
Let 
\[
y = \sqrt[3]{ -\frac{t}{2} + \sqrt{ \frac{t^2}{4} - \frac{1}{27} } } + \sqrt[3]{ -\frac{t}{2} - \sqrt{ \frac{t^2}{4} - \frac{1}{27}}},
\] 
then the generating function \(\mathcal{R}_3(t)\) is
\begin{equation}
\mathcal{R}_3  = \frac{y^4}{12} - \frac{y^6}{6}.
\end{equation}
\end{Th}
\begin{proof}
Differentiate equation \eqref{eq:surfaces_differential_equation}
\begin{gather*}
\mathcal{R}_3' = \frac{1}{4} \mathcal{R}_3' + \frac{1}{4} t \mathcal{R}_3'' - \frac{3}{4} \cdot 2 \mathcal{R}_3' \mathcal{R}_3'', \\
\frac{3}{4} \mathcal{R}_3' = \mathcal{R}_3'' \left( \frac{1}{4} t - \frac{3}{2} \mathcal{R}_3' \right).
\end{gather*}
Introduce the parameter \(p = \mathcal{R}_3'\) 
\begin{gather*}
\frac{\mathrm{d}p}{\mathrm{d}t} \left(\frac{1}{4}t - \frac{3}{2} p\right) = \frac{3}{4} p. \\  
\frac{\mathrm{d}t}{\mathrm{d} p} = \frac{\frac{1}{4}t - \frac{3}{2} p}{\frac{3}{4} p} = \frac{4 (\frac{1}{4}t - \frac{3}{2} p)}{3p} =  \frac{t-6p}{3p} =  \frac{t}{3p} - 2.
\end{gather*}
Particular solutions are clearly not suitable (the check is omitted). We solve the resulting differential equation for \(t\) as a function of \(p\) using the method of variation of constants. The solution is omitted; the answer is given directly
\begin{gather*}
t = C_1 p^{\frac{1}{3}} - 3p.
\end{gather*}
Perform the reverse substitution
\begin{gather*}
t = C_1 (\mathcal{R}_3')^{\frac{1}{3}} - 3\mathcal{R}_3'. \\
C_1 (\mathcal{R}_3')^{\frac{1}{3}} = t + 3\mathcal{R}_3'.
\end{gather*}
\begin{equation} \label{eq:cubic_equation_tetrahedron}
C_1^3 \mathcal{R}_3' = (t + 3\mathcal{R}_3')^3. 
\end{equation}
This is a cubic equation in \(\mathcal{R}_3'\). We will solve it using Cardano's formula, but first we find the value of \(C_1^3\) and bring it to a convenient form.

Our generating function has coefficients for degrees \(\leq 3\) equal to 0, and the coefficient for degree \(4\) is \(\frac{R_{1}}{4!}\). This yields initial conditions
\[
\mathcal{R}_3(0) = \mathcal{R}_3'(0) = \mathcal{R}_3''(0) = \mathcal{R}_3'''(0) = 0, \quad \mathcal{R}_3^{(4)}(0) = R_1 = 2.
\]
Expand the brackets, differentiate \eqref{eq:cubic_equation_tetrahedron} 3 times, and substitute \(0\)
\begin{gather*}
C_1^3 \mathcal{R}_3' = t^3 + 9t^2 \mathcal{R}_3' + 27t (\mathcal{R}_3')^2 + 27(\mathcal{R}_3')^3. \\
C_1^3 \mathcal{R}_3^{(4)}(0) - \left(t^3 + 9t^2 \mathcal{R}_3' + 27t (\mathcal{R}_3')^2 + 27(\mathcal{R}_3')^3\right)''' \Big|_{t=0} = 0. \\
2C_1^3 - 6 - 0 - 0 - 0 = 0. \\
C_1^3 = 3.
\end{gather*}
Differentiating the expression \( t^3 + 9t^2 \mathcal{R}_3' + 27t (\mathcal{R}_3')^2 + 27(\mathcal{R}_3')^3 \) with respect to \( t \) and then evaluating at \( t=0 \) yields \(6\), because all terms except \(t^3\) vanish (either due to multiplication by \(t^i\), or multiplication by \(\mathcal{R}^{(j)}_3,\) where \(j = 1,2,3\)). Now make the substitution
\begin{equation} \label{eq:cubic_y_substitution}
y = t+3\mathcal{R}_3', \quad \mathcal{R}_3' = \frac{y-t}{3}. 
\end{equation}
Then
\begin{equation} \label{eq:y_cubic_equation}
y^3 = y-t.
\end{equation}

To solve this equation, we apply Cardano's formula. The equation has the form \(y^3 + p y + q = 0\), where \(p = -1\), \(q = t\). The discriminant is
\[
D = \left( \frac{t}{2} \right)^2 + \left( \frac{-1}{3} \right)^3 = \frac{t^2}{4} - \frac{1}{27}.
\]
Introduce
\[
\alpha = \sqrt[3]{ -\frac{t}{2} + \sqrt{D}}, \quad \beta = \sqrt[3]{ -\frac{t}{2} - \sqrt{D}}.
\]
Then the solutions are
\[
y_1 = \alpha + \beta = \sqrt[3]{ -\frac{t}{2} + \sqrt{ \frac{t^2}{4} - \frac{1}{27} } } + \sqrt[3]{ -\frac{t}{2} - \sqrt{ \frac{t^2}{4} - \frac{1}{27}}}; \quad y_{2,3} = -\frac{\alpha + \beta}{2} \pm i \frac{\alpha - \beta}{2} \sqrt{3}.
\]
Since \(y(0) = 0\), the roots \(y_{2,3}(0) = -\frac{\alpha(0) + \beta(0)}{2} \pm i \frac{\alpha(0) - \beta(0)}{2} \sqrt{3} \neq 0\), meaning they are not suitable.
Note that
\[
\sqrt{ \frac{t^2}{4} - \frac{1}{27} } = \sqrt{ \frac{27t^2 - 4}{108} } = \frac{1}{6\sqrt{3}} \sqrt{27t^2 - 4},
\]
so \( y \) takes the form
\begin{gather*}
y = \sqrt[3]{ -\frac{t}{2} + \frac{1}{6\sqrt{3}} \sqrt{27t^2 - 4} } + \sqrt[3]{ -\frac{t}{2} - \frac{1}{6\sqrt{3}} \sqrt{27t^2 - 4} }. \\
y = \sqrt[3]{ \frac{1}{6\sqrt{3}} \left( -3\sqrt{3}t + \sqrt{27t^2 - 4} \right) } + \sqrt[3]{ \frac{1}{6\sqrt{3}} \left( -3\sqrt{3}t - \sqrt{27t^2 - 4} \right) }. \\
y = \frac{1}{\sqrt[3]{2} \cdot \sqrt{3}} \left( \sqrt[3]{ -3\sqrt{3}t + \sqrt{27t^2 - 4} } + \sqrt[3]{ -3\sqrt{3}t - \sqrt{27t^2 - 4} } \right).
\end{gather*}
Hence
\begin{equation} \label{eq:R_3_explicit_formula}
\mathcal{R}'_3 = -\frac{t}{3} + \frac{ \sqrt[3]{ -3\sqrt{3}t + \sqrt{27t^2 - 4} } + \sqrt[3]{ -3\sqrt{3}t - \sqrt{27t^2 - 4} } }{3 \sqrt[3]{2} \cdot \sqrt{3} }.    
\end{equation}
Now, let us also express \(\mathcal{R}_3\) in terms of \(y\). Recall that
\begin{gather*}
t = y - y^3 \implies \mathrm{d}t = (1 - 3y^2)  \mathrm{d}y. \\
\mathrm{d}\mathcal{R}_3 = \mathcal{R}_3'  \mathrm{d}t = \frac{y - t}{3}  \mathrm{d}t = \frac{y^3}{3}(1 - 3y^2)  \mathrm{d}y. \\
\mathcal{R}_3 = \int \left( \frac{y^3}{3} - y^5 \right)  \mathrm{d}y = \frac{1}{3} \cdot \frac{y^4}{4} - \frac{y^6}{6} + C = \frac{y^4}{12} - \frac{y^6}{6} + C.
\end{gather*}
Let us also find \(C\) using initial conditions
\[
0 = \mathcal{R}_3(0) = \frac{y^4(0)}{12} - \frac{y^6(0)}{6} + C \implies C = 0.
\]
\end{proof}

\begin{corollary}
The number of non-isomorphic numberings of triangulations \(R_n\) is equal to \(\mathcal{R}_3^{(2n+2)}(0)\), where \(\mathcal{R}_3(t)\) is the generating function from statement \ref{stat:R_surfaces_differential_equation}.
\end{corollary}

To find an explicit formula for \(R_n\), we will apply the Lagrange inversion theorem.

\begin{Th}[3.6.6 \& 3.6.7 in \cite{AM83}, Theorem~A.2 in \cite{fs2009}] \label{thm:lagrange_bürmann}
Let $\phi(w) = \sum_{k\ge 0} \phi_k w^k$ be a formal power series with $\phi_0 \neq 0$,
and let $y(z)$ be the unique solution of the equation
\[
y = z\,\phi(y).
\]
Then for any $n\ge 1$ and any formal power series $H(w)$,
the coefficients of $H(y(z))$ are given by the Lagrange-Bürmann formula:
\[
[z^{\,n}]\,H(y(z)) = \frac{1}{n}\,[w^{\,n-1}] \left( H'(w)\,\phi(w)^n \right).
\]   
\end{Th}
In particular, we need well-known applications of the Lagrange inversion formula.
\begin{corollary}[Section 3.3 in \cite{Gessel2016}] \label{cor:lagrange_application1}
Let \(d\) be an integer with \(d\ge 2\). For the equation given by
\begin{equation} \label{eq:widetilde_y}
\widetilde{y} = 1 + s\,\widetilde{y}^{\,d}, \quad \widetilde{y}(0)=1,
\end{equation}
there exists a series expansion
\begin{equation} \label{eq:widetilde_y_series}
\widetilde{y}(s) = \sum_{n=0}^{\infty} \frac{1}{(d-1)n+1}\binom{dn}{n}\,s^n.  
\end{equation} 

\end{corollary}

\begin{corollary} \label{cor:lagrange_application2}
The unique formal power series \(y(t)\) with \(y(0)=0\) satisfying \(y^d - y + t = 0\) is given by
\begin{equation} \label{eq:y_series}
y(t) = \sum_{n=0}^{\infty} \frac{1}{(d-1)n+1} \binom{dn}{n}\, t^{(d-1)n+1}.   
\end{equation}
\end{corollary}

\begin{proof}
To obtain the expansion for \(y(t)\), we substitute \(y(t)=t\,u(t)\) into \(y^d - y + t = 0\). Since \(y(0)=0\) and differentiating the equation \(y^d - y + t = 0\) yields \(y'(0)=1\), the series \(u(t)=y(t)/t\) is well-defined and satisfies \(u(0)=1\). Substituting gives  
\[
(tu)^d - tu + t = 0 \quad\Longrightarrow\quad t\left(t^{d-1}u^d - u + 1\right)=0.
\]  
Notice that for \(t \neq 0\) we obtain \(t^{d-1}u^d - u + 1 = 0\). Furthermore, since \(u(0)=1\), substituting \(t=0\) gives \(t^{d-1}u^d(0) - u(0) + 1 = 0 - 1 + 1 = 0\), so the equation holds at \(t=0\) as well, hence
\[
t^{d-1}u^d - u + 1 = 0 \quad\Longrightarrow\quad u = 1 + t^{d-1}u^d.
\]  
Thus \(u(t)\) satisfies the same functional equation as \(\widetilde{y}(s)\) in \eqref{eq:widetilde_y}, with \(s = t^{d-1}\). Consequently, by the known series expansion \eqref{eq:widetilde_y_series},
\[
u(t) = \widetilde{y}\left(t^{d-1}\right) = \sum_{n=0}^{\infty} \frac{1}{(d-1)n+1}\binom{dn}{n} \left(t^{d-1}\right)^n.
\]  
Multiplying by \(t\) yields  
\[
y(t)=t\,u(t)=\sum_{n=0}^{\infty} \frac{1}{(d-1)n+1}\binom{dn}{n}\, t^{(d-1)n+1},
\]  
which is precisely \eqref{eq:y_series}.
\end{proof}

\begin{rem}
For the tetrahedron case, one can also find the explicit formula by first making the substitution \( t = \sqrt{\frac{4}{27}} \sin u\) and then deriving recurrence relations for the derivatives at zero.
\end{rem}

\begin{Th} \label{thm:R_surfaces_formula}
 Let \(R_n\) denote the number of non-isomorphic labelings of the external faces of a tetrahedron tree that can be obtained at step \(n\). Then, 
\(R_{n} = \dfrac{(3n)!}{3 \cdot n!}.\)
\end{Th}

\begin{proof}
We want to use the results from the statements \ref{stat:R_surfaces_recurrence}, \ref{stat:R_surfaces_differential_equation}. We apply Corollary~\ref{cor:lagrange_application2} for the case \(d = 3\) to equation \eqref{eq:y_cubic_equation}
\[
y(t) = \sum_{m=0}^{\infty} \frac{t^{2m+1}}{2m+1} \binom{3m}{m}.
\]
Thus, using \eqref{eq:cubic_y_substitution}, we got
\[
\mathcal{R}_3'= \frac{y - t}{3} = \frac{1}{3} \sum_{m=1}^{\infty} \frac{t^{2m+1}}{2m+1} \binom{3m}{2m}.
\]
Integrating term by term and using the initial condition \(\mathcal{R}_3(0) = \mathcal{R}'_3(0) = 0\) yields
\[
\mathcal{R}_3(t) = \frac{1}{3} \sum_{m=1}^{\infty} \frac{t^{2m+2}}{(2m+1)(2m+2)} \frac{(3m)!}{m! \cdot (2m)!} = \sum\limits_{n=1}^{\infty} \dfrac{R_{n} t^{2n+2}}{(2n+2)!}.
\]
Comparing coefficients of equal powers we obtain \(R_{n} = \dfrac{(3n)!}{3 \cdot n!}\).
\end{proof}

\begin{rem} \label{rem:surfaces_without_automorphism}
\[
R_{n} \sim \frac{1}{3} \left(\frac{3n}{e}\right)^{3n} \sqrt{2 \pi \cdot 3n} \left(\left(\frac{n}{e}\right)^{n} \sqrt{2 \pi n}\right)^{-1} 
= \frac{1}{\sqrt{3}} \, \frac{27^{\,n} \, n^{2n}}{e^{2n}}.
\]
Moreover, after removing labelings (numberings), we obtain the following asymptotic estimate
\begin{align*}
\frac{R_{n}}{(2n+2)!} 
\sim \frac{1}{\sqrt{3}} \, \frac{27^{\,n} \, n^{2n}}{e^{2n}} \cdot 
\left( (2n+2)(2n+1) \left(\frac{2n}{e}\right)^{2n} \sqrt{4\pi n} \right)^{-1}\\
= \frac{ 27^{\,n}}{4 \cdot 4^{n}  \, \sqrt{3} \, (n+1)(2n+1) \sqrt{\pi n}}.
\end{align*}
\end{rem}
This asymptotic behavior implies that  
\[
\frac{6n R_{n}}{(2n+2)! \, F(n,0)} \xrightarrow[n \to \infty]{} 0,
\]
where \(F(n,g)\) is taken from \cite{GJ08, BGR08}; that is, tetrahedron trees constitute a vanishingly small proportion among all triangulations.

\begin{rem}
Substituting the explicit formula from theorem \ref{thm:R_surfaces_formula} into the recurrence relation \eqref{eq:R_surfaces_recurrence}, we obtain the following identity
\begin{equation*}
\frac{(3n)!}{3 \, n!}
= \frac{3}{2(n-1)} \sum_{k=1}^{n-1} \frac{(3k)!}{3 \, k!} \cdot \frac{(3(n-k))!}{3 \, (n-k)!} \binom{2n+2}{2k+1}
= \frac{1}{6(n-1)} \sum_{k=1}^{n-1} \frac{(3n)!}{n!} \cdot \frac{\binom{n}{k}}{\binom{3n}{3k}} \binom{2n+2}{2k+1}.
\end{equation*}
Thus, we obtained an interesting relation for binomial coefficients
\begin{equation} \label{eq:R_surfaces_combinatorial}
2(n-1) = \sum_{k=1}^{n-1} \frac{\binom{n}{k} \, \binom{2n+2}{2k+1}}{\binom{3n}{3k}}.
\end{equation}
\end{rem}

\begin{quest}
     \(R_n\) is a discrete process. Can it be extended to a continuous one, for example as \(\frac{\Gamma(3n)}{\Gamma(n)}\) in the case of labeled triangulations, and as \(\frac{\Gamma(3n)}{\Gamma(n)\Gamma(2n+3)}\) in the unlabeled case?
\end{quest}

\begin{quest}
    The formula for \(R_n\) turned out to be quite simple. Is there a straightforward combinatorial explanation for it?
\end{quest}

\section{The Number of Inductive Triangulations} \label{sec:inductive_triangulations}

Now we want to count non-isomorphic labelings for any fixed initial triangulation. Starting from an initial triangulation of any manifold, we generate further simplicial complexes by gluing tetrahedral trees. These trees are exactly the "building blocks" whose count we computed in the previous section. Suppose we have an initial triangulation consisting of \(2n_0\) triangles and \( {}_{2n_0} T_{0} \in \mathbb{N}_{>0}\) non-isomorphic labelings. We begin with an initial triangulation and then add tetrahedron at each step. We introduce the notation \( {}_{h} T_{n} \), where the lower left subscript \(h\) indexes the initial fixed triangulation and the lower right subscript \(n\) corresponds to the step (i.e., the number of added tetrahedra).

\begin{stat} \label{stat:T_surfaces_recurrence}
The following recurrence relation holds
\begin{equation}
n \cdot {}_{2n_0} T_{n} = 3 \sum\limits_{k=1}^{n} {}_{2n_0} T_{n-k} R_{k} \binom{2(n_0+n)}{2k+1}
= 3 \sum\limits_{k=1}^{n} {}_{2n_0} T_{n-k}  \frac{(3k)!}{3 \cdot k!} \binom{2(n_0+n)}{2k+1}.
\end{equation}
\end{stat}

\begin{proof}
Let \( {}_{2n_0} \Phi_0 \) be an initial simplicial complex with a boundary triangulation \( {}_{2n_0} \Phi_0^{\partial} \), where \( 2n_0 \) denotes the number of boundary triangles. Now, consider gluing \( n \) tetrahedra to this complex. We denote:
\begin{itemize}
    \item \( {}_{2n_0} \Phi_n \) as the resulting full complex after the \( n \) attachments.
    \item \( {}_{2n_0} \Phi_n^{\partial} \) as the induced triangulation on the boundary of \( {}_{2n_0} \Phi_n \).
\end{itemize}
The idea of proving this recurrence is similar to proof of the statement \ref{stat:R_surfaces_recurrence}: we again make cuts at the \(n\) attachment faces of tetrahedra in the resulting triangulation.

Fix a split \(n=k+(n-k)\) with \(1\le k\le n\). Every complex \({}_{2n_0} \Phi_n\) can be decomposed into \({}_{2n_0} \Phi_{n-k}\) and subtree \(\Phi_{k}\) glued together along one external triangle. There are exactly three ways (the order of the triangle's rotation group) to perform the gluing --- depending on the choice of the triangle's side along which the merging is made.  
The subtree have \(2k+2\) external triangles, respectively, so we need to distribute \(2k+2-1\) numbers on the figure \(\Phi_{k}\) (and accordingly \(2(n-k)+2n_0-1\) numbers on the figure \({}_{2n_0} \Phi_{n-k}\)) (we subtract one because one label on the face along which the gluing occurs is no longer counted). This gives us
\(\binom{2(n_0+n)}{2k+1}\) ways to do that. Finally, every decomposition fixes a concrete seam, whereas the left-hand side of the recurrence is simply \({}_{2n_0} T_{n}\) --- the number of chains without a distinguished seam. To account for all possible seams we multiply by their total number: \(n\). Hence, the desired recurrence relation is
\begin{equation} \label{eq:T_surfaces_recurrence}
n \cdot {}_{2n_0} T_{n} = 3 \sum\limits_{k=1}^{n} {}_{2n_0} T_{n-k} R_{k} \binom{2(n_0+n)}{2k+1}.   
\end{equation}

A visual illustration of one inductive step is given in Figure~\ref{fig:T_induction_step}.

\end{proof}

\begin{rem}
Note that in this case there is no symmetry between \(k\leftrightarrow n-k\); so we don't need to divide by 2. Moreover, the construction requires \(n\) gluings, not \(n-1\).
\end{rem}

\begin{rem} \label{rem:R_T_connection_surfaces_from_construction}
Observe also that we can start with a single tetrahedron as the initial configuration and then build inductive triangulation.  In this setting, we can compare \({}_{2n_0} T_{n-1}\) with \(R_n\): \({}_{2n_0} T_{n-1}\) is \(n\) times larger than \(R_n\), because by symmetry any tetrahedron in \(R_n\) can be chosen as the starting one.  
\end{rem}
  
\begin{figure}[!htbp]
\centering
\includegraphics[scale=0.5]{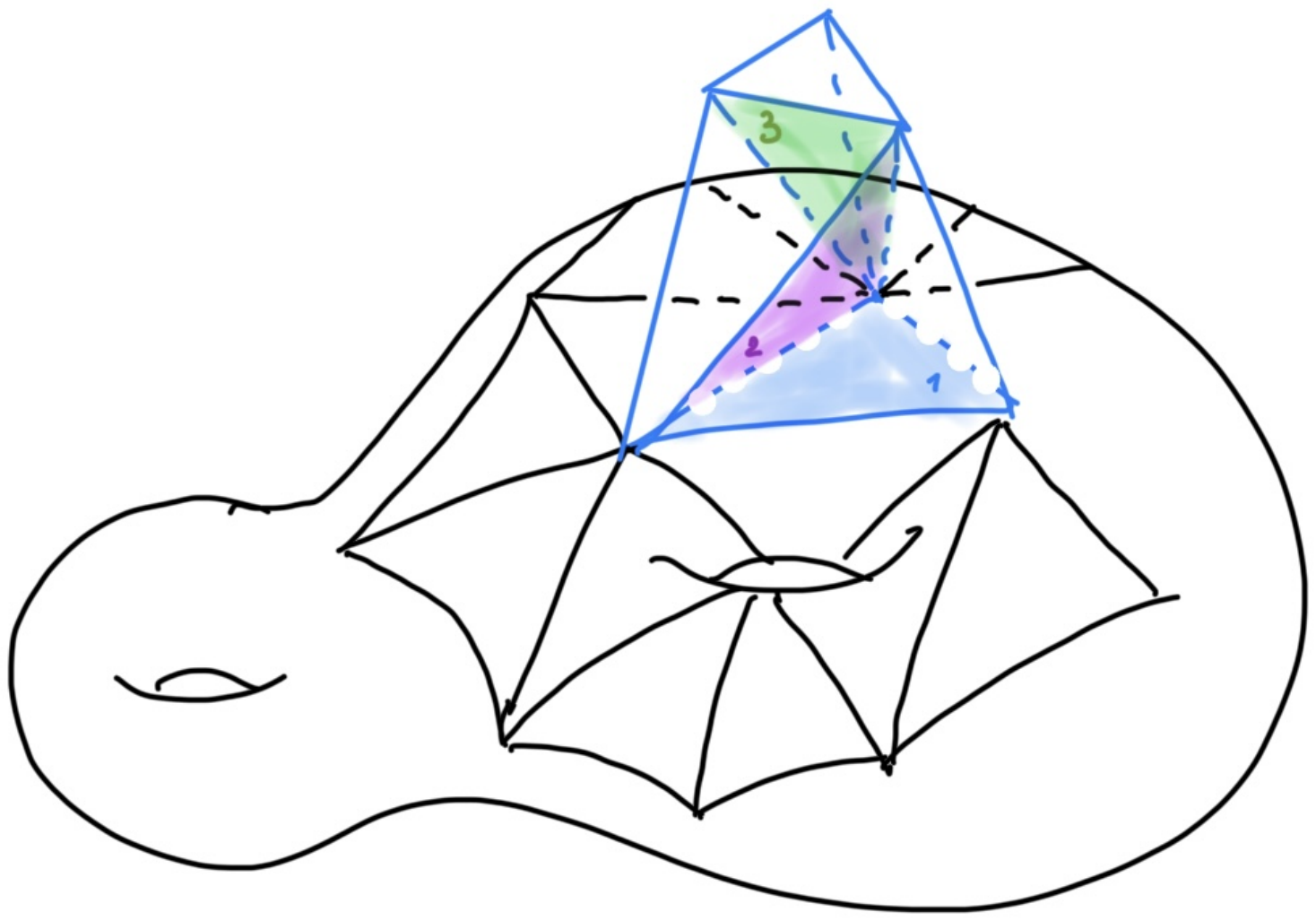}

\caption{Attaching a tetrahedron tree to an initial triangulation.}

\label{fig:T_induction_step}

\end{figure}

We will follow an approach analogous to the one used to find \(R_n\). Consider the exponential generating function (EGF) \({}_{2n_0} \mathcal{T}_3(t) = \sum\limits_{n = 0}^{\infty} {}_{2n_0} T_{n} \dfrac{t^{2n_0+2n}}{(2n_0 + 2n)!}\).

\begin{stat}
    \({}_{2n_0} \mathcal{T}_3(t)\) and \(\mathcal{R}_3\) satisfy the following differential equation
    \begin{equation} \label{eq:diff_T_R_general}
    {}_{2n_0} \mathcal{T}'_3 \mathcal{R}'_3 = \frac{1}{3} \left( \frac{t \cdot {}_{2n_0} \mathcal{T}_3'}{2} - n_0 \cdot {}_{2n_0} \mathcal{T}_3\right).
    \end{equation}
\end{stat}

\begin{proof}
Consider
\begin{align*}
{}_{2n_0} \mathcal{T}'_3(t) \mathcal{R}'_3(t) &= \left( \sum\limits_{n =0}^{\infty} \frac{{}_{2n_0} T_{n} t^{2n_0 + 2n - 1}}{(2n_0 + 2n-1)!} \right) \left( \sum\limits_{n = 1}^{\infty} \frac{R_{n} t^{2n + 1}}{(2n+1)!}\right) \\
&= \sum\limits_{n = 1}^{\infty} \sum\limits_{k = 1}^{n} \frac{{}_{2n_0} T_{n-k} t^{2n_0 + 2n - 2k - 1} R_{k} t^{2k + 1}}{(2n_0 + 2n - 2k - 1)! (2k + 1)!} \\
&= \sum\limits_{n = 1}^{\infty} \frac{n}{3} \frac{3}{n} \sum\limits_{k = 1}^{n} \frac{{}_{2n_0} T_{n-k} R_{k} t^{2n_0 + 2n} \binom{2n_0 + 2n}{2k+1} }{(2n_0 + 2n)!} \\
&=  \sum\limits_{n = 1}^{\infty} \frac{n}{3} \frac{{}_{2n_0} T_{n} t^{2n_0 + 2n}}{(2n_0 + 2n)!} \\
&= \sum\limits_{n = 0}^{\infty} \frac{n}{3} \frac{{}_{2n_0} T_{n} t^{2n_0 + 2n}}{(2n_0 + 2n)!} \\
&= \frac{1}{3} \sum\limits_{n = 0}^{\infty} \left(\frac{2(n+n_0)}{2} - n_0 \right)\frac{{}_{2n_0} T_{n} t^{2n_0 + 2n }}{(2n_0 + 2n )!} \\
&= \frac{1}{3} \left( \frac{t \cdot {}_{2n_0} \mathcal{T}'_3(t)}{2} - n_0 \cdot {}_{2n_0} \mathcal{T}_3(t)\right).
\end{align*}
\end{proof}
\begin{Th} \label{thm:T_surfaces_formula}
Let \( {}_{2n_0} T_{n} \) denote the number of non-isomorphic numberings that can be obtained from inductive triangulation consisting of \(2n_0\) triangles with \( {}_{2n_0} T_{0}\) non-isomorphic labelings. Also consider the EGF \({}_{2n_0} \mathcal{T}_3(t) = \sum\limits_{n = 0}^{\infty} {}_{2n_0} T_{n} \dfrac{t^{2n_0+2n}}{(2n_0 + 2n)!}\). Then
\begin{gather*}
{}_{2n_0} \mathcal{T}_3(t) = \frac{{}_{2n_0} T_{0}}{(12)^{n_0} (2n_0)!} \left(\sqrt[3]{4} \sqrt[3]{-3^{1.5} t + \sqrt{27t^2 - 4}} + \sqrt[3]{4} \sqrt[3]{-3^{1.5} t - \sqrt{27t^2 - 4}}\right)^{2n_0}, \\
{}_{2n_0} T_{n} = \frac{{}_{2n_0} T_{0} (3n+2n_0-1)!}{(2n_0-1)! \cdot n!} = {}_{2n_0} T_{0} \frac{(3n)!}{n!} 
\binom{3n +2n_0-1}{2n_0-1}.
\end{gather*}
\end{Th}

\begin{proof}
To solve the differential equation \eqref{eq:diff_T_R_general}, we transform its expression
\[
{}_{2n_0} \mathcal{T}'_3 \mathcal{R}'_3  = \frac{t\cdot {}_{2n_0} \mathcal{T}'_3}{6} - \frac{n_0}{3} \cdot {}_{2n_0} \mathcal{T}_3 \iff \frac{n_0}{3} {}_{2n_0} \mathcal{T}_3 = {}_{2n_0} \mathcal{T}_3'\left(-\mathcal{R}'_3+\frac{t}{6}\right).
\]
Let us separate the variables
\[
\frac{\mathrm{d} ({}_{2n_0} \mathcal{T}_3)}{{}_{2n_0} \mathcal{T}_3} = \frac{n_0}{3} \frac{\mathrm{d}t}{-\mathcal{R}'_3+\frac{t}{6}} \iff {}_{2n_0} \mathcal{T}_3 = C_1 \exp \left(\frac{n_0}{3} \int \frac{\mathrm{d} t}{-\mathcal{R}'_3+\frac{t}{6}} \right).
\]
Recall that \(\mathcal{R}'_3 = -\frac{t}{3} + \frac{\sqrt[3]{-3^{1.5}t + \sqrt{27t^2 - 4}} + \sqrt[3]{-3^{1.5}t - \sqrt{27t^2 - 4}}}{3\sqrt[3]{2} \cdot \sqrt{3}}\) (see \eqref{eq:R_3_explicit_formula}). Then
\begin{gather*}
{}_{2n_0} \mathcal{T}_3 = C_1 \exp \left(\frac{n_0}{3} \int \frac{\mathrm{d}t}{-\frac{\sqrt[3]{-3^{1.5}t + \sqrt{27t^2 - 4}} + \sqrt[3]{-3^{1.5}t - \sqrt{27t^2 - 4}}}{3\sqrt[3]{2} \cdot \sqrt{3}}+ \frac{t}{2}} \right). \\
{}_{2n_0} \mathcal{T}_3 = C_1 \exp \left( \frac{n_0}{3} \int \frac{6 \sqrt{3}\, \mathrm{d}t}{-\sqrt[3]{4} \left( \sqrt[3]{-3^{1.5} t + \sqrt{27t^2 - 4} } + \sqrt[3]{-3^{1.5} t - \sqrt{27t^2 - 4} } \right) + 3 \sqrt{3} t } \right).
\end{gather*}
Denote
\[
I(t) = \int \frac{\mathrm{d}t}{-\sqrt[3]{4} \left( \sqrt[3]{-3^{1.5} t + \sqrt{27t^2 - 4} } + \sqrt[3]{-3^{1.5} t - \sqrt{27t^2 - 4} } \right) + 3 \sqrt{3} t}.
\]
Introduce \(u = \sqrt[3]{4} \sqrt[3]{-3^{1.5} t + \sqrt{27t^2 - 4}}, v = \sqrt[3]{4} \sqrt[3]{-3^{1.5} t - \sqrt{27t^2 - 4}}, z = u + v\).
Note that
\[
uv = \sqrt[3]{16} \sqrt[3]{-3^{1.5} t + \sqrt{27t^2 - 4}} \sqrt[3]{-3^{1.5} t - \sqrt{27t^2 - 4}} = \sqrt[3]{16} \sqrt[3]{27t^2 - 27t^2 + 4} = 4;
\]
\[
z^3 = u^3+v^3 + 3uv(u+v) = 4\left(2 \cdot (-3^{1.5} t)\right) + 3 \cdot 4 z ;
\]
\[
z^3-12z = -24 \sqrt{3}t \implies (3z^2-12)\mathrm{d}z = -24 \sqrt{3}\mathrm{d}t.
\]
With this in mind,
\begin{align*}
I &= \int \frac{\mathrm{d}t}{-z + 3\sqrt{3}\,t}
   = \int \frac{\dfrac{3z^{2}-12}{-24\sqrt{3}}\,\mathrm{d}z}{-z + \dfrac{z^{3}-12z}{-8}}
   = \int \frac{(z^{2}-4)\,\mathrm{d}z}{8\sqrt{3}\!\left(z + \dfrac{z^{3}-12z}{8}\right)} \\
  &= \frac{1}{\sqrt{3}}\int \frac{(z^{2}-4)\,\mathrm{d}z}{8z + z^{3}-12z}
   = \frac{1}{\sqrt{3}}\int \frac{\mathrm{d}z}{z}
   = \frac{1}{\sqrt{3}}\ln|z| + C_{2}.
\end{align*}
Therefore
\[
{}_{2n_{0}}\mathcal{T}_{3}=C_{1}\exp\left(2n_{0}\sqrt{3}\,I(t)\right)=C|z|^{2n_{0}\sqrt{3}\cdot\frac{1}{\sqrt{3}}}=C|z|^{2n_{0}}=Cz^{2n_{0}}.
\]
Returning to the variable \(t\),
\[
{}_{2n_{0}}\mathcal{T}_{3}=C\left(\sqrt[3]{4}\sqrt[3]{-3^{3/2}t+\sqrt{27t^{2}-4}}+\sqrt[3]{4}\sqrt[3]{-3^{3/2}t-\sqrt{27t^{2}-4}}\right)^{2n_{0}}.
\]
Let us use Lagrange-Bürmann formula (Theorem \ref{thm:lagrange_bürmann}).  
We have
\begin{gather*}
t=g(z)=\frac{z^{3}-12z}{-24\sqrt{3}}, \\
{}_{2n_{0}}\mathcal{T}_{3}(z(t))={}_{2n_{0}}\mathcal{T}_{3}(z(0))+\sum_{j=1}^{\infty}\frac{t^{j}}{j!}\left[\frac{\mathrm{d}^{j-1}}{\mathrm{d}x^{j-1}}\!\left({}_{2n_{0}}\mathcal{T}_{3}'(x)\!\left(\frac{x}{g(x)}\right)^{j}\right)\right]_{x=0}.
\end{gather*}
Since \(z\) is odd, \(z(0)=z^{3}(0)=0\) and \(g'(0)=-\dfrac{1}{2\sqrt{3}}\neq 0\); hence \({}_{2n_{0}}\mathcal{T}_{3}(z(0))=0^{2n_{0}}=0\) and \({}_{2n_{0}}\mathcal{T}_{3}'(x)=2n_{0}x^{2n_{0}-1}\).  
We need the coefficient of \(x^{j-1}\) in
\begin{align*}
{}_{2n_{0}}\mathcal{T}_{3}'(x)\!\left(\frac{x}{g(x)}\right)^{j}
  &= 2n_{0}x^{2n_{0}-1}\!\left(\frac{-24\sqrt{3}}{x^{2}-12}\right)^{j} = 2n_{0}x^{2n_{0}-1}\!\left(\frac{2\sqrt{3}}{1-\frac{x^{2}}{12}}\right)^{j} \\
  &= 2n_{0}(2\sqrt{3})^{j}x^{2n_{0}-1}\sum_{s=0}^{\infty}\binom{-j}{s}(-1)^{s}\frac{x^{2s}}{12^{s}}.
\end{align*}
The power \(2s+2n_{0}-1\) must equal \(j-1\), so \(j=2s+2n_{0}\) (otherwise the coefficient is zero).  
Thus
\begin{align*}
\left[\frac{\mathrm{d}^{j-1}}{\mathrm{d}x^{j-1}}\!\left({}_{2n_{0}}\mathcal{T}_{3}'(x)\!\left(\frac{x}{g(x)}\right)^{j}\right)\right]_{x=0}
  &= (j-1)!\,2n_{0}(2\sqrt{3})^{2s+2n_{0}}\binom{-(2s+2n_{0})}{s}(-1)^{s}12^{-s} \\
  &= (j-1)!\,2n_{0}\,12^{s+n_{0}}(-1)^{s}\binom{3s+2n_{0}-1}{s}(-1)^{s}12^{-s}.
\end{align*}
Collecting everything,
\[
z^{2n_{0}}=\sum_{s=0}^{\infty}\frac{t^{2s+2n_{0}}}{2s+2n_{0}}\,2n_{0}\,12^{n_{0}}\binom{3s+2n_{0}-1}{s}.
\]
Therefore
\[
{}_{2n_{0}}\mathcal{T}_{3}(t)=C\sum_{s=0}^{\infty}\frac{t^{2s+2n_{0}}}{2s+2n_{0}}\,2n_{0}\,12^{n_{0}}\binom{3s+2n_{0}-1}{s}=\sum_{n=0}^{\infty}{}_{2n_{0}}T_{n}\frac{t^{2n_{0}+2n}}{(2n_{0}+2n)!}.
\]
Thus
\[
2n_{0}C\,12^{n_{0}}\frac{(3n+2n_{0}-1)!}{n!\,(2n+2n_{0}-1)!\,(2n+2n_{0})}={}_{2n_{0}}T_{n}\frac{1}{(2n_{0}+2n)!}.
\]
To find \(C\) note that \(z=2\sqrt{3}\,t+\dots\), so
\[
z^{2n_{0}}(t)=(2\sqrt{3})^{2n_{0}}t^{2n_{0}}+\dots,\qquad{}_{2n_{0}}G_{3}^{(2n_{0})}(0)={}_{2n_{0}}T_{0}=C(12)^{n_{0}}(2n_{0})!.
\]
Hence, \(C=\dfrac{{}_{2n_{0}}T_{0}}{(12)^{n_{0}}(2n_{0})!}.\)
Substituting back
\[
{}_{2n_{0}}T_{n}=\frac{(3n+2n_{0}-1)!}{n!}\,2n_{0}\,C\,12^{n_{0}}=\frac{{}_{2n_{0}}T_{0}(3n+2n_{0}-1)!}{(2n_{0}-1)!\,n!}={}_{2n_{0}}T_{0}\frac{(3n)!}{n!}\binom{3n+2n_{0}-1}{2n_{0}-1}.
\]
\end{proof}

\begin{rem} \label{rem:R_T_connection_surfaces}
We can establish a connection with the tetrahedron tree as follows. Take a tetrahedron composed of four triangles; it admits two non-isomorphic labelings, i.e.,
\[
n_0 = 2, \, {}_{2n_0} T_{0} = R_{1} = 2.
\]
Using the fact that the indexing is shifted by one, we
obtain
\begin{align*}
{}_{2n_0} T_{n-1} 
&= {}_{2 \cdot 2} T_{n-1} = 2 \,\frac{(3n-3)!}{(n-1)!} \binom{3(n-1) + 2 \cdot 2 - 1}{2 \cdot 2 - 1}\\
&= 2 \,\frac{(3n-3)!}{(n-1)!} \binom{3n}{3} = 2 \,\frac{(3n-3)!}{(n-1)!} \cdot \frac{3n(3n-1)(3n-2)}{6} \\
&= \frac{(3n)!}{3\,(n-1)!} = n R_{n}.
\end{align*}
This is consistent with remark \ref{rem:R_T_connection_surfaces_from_construction}.
\end{rem}

\begin{rem}
Substituting the explicit formula from theorem \ref{thm:T_surfaces_formula} into the recurrence relation \eqref{eq:T_surfaces_recurrence}, we obtain the following identity
\begin{align*}
n\, {}_{2n_0}T_{0} \frac{(3n)!}{n!} \binom{3n +2n_0-1}{2n_0-1}
&= 3 \sum\limits_{k=1}^{n} {}_{2n_0}T_{0} \frac{(3(n-k))!}{(n-k)!}
   \binom{3(n-k) +2n_0-1}{2n_0-1} \frac{(3k)!}{3 \cdot k!}
   \binom{2(n_0+n)}{2k+1} \\
&= \sum\limits_{k=1}^{n} {}_{2n_0}T_{0}
   \frac{\frac{(3n)!}{\binom{3n}{3k}}}{\frac{n!}{\binom{n}{k}}}
   \binom{3(n-k) +2n_0-1}{2n_0-1}
   \binom{2(n_0+n)}{2k+1}.
\end{align*}
Hence, for an arbitrary integer \(n_0 \geq 1\) we obtained a nice equality for binomial coefficients
\begin{equation} \label{eq:T_surfaces_combinatorial}
n \cdot \binom{3n +2n_0-1}{2n_0-1} = \sum\limits_{k=1}^{n} \frac{\binom{n}{k}}{\binom{3n}{3k}} \binom{3(n-k) +2n_0-1}{2n_0-1} \binom{2(n_0+n)}{2k+1}.
\end{equation}

\end{rem}

\section{Generalization to Other Regular Polytopes} \label{sec:general_inductive}

We extend the recurrence relations, differential equations, and formulas from the previous sections to the case where the tetrahedron is replaced with an arbitrary regular polytope of any dimension, obtaining relations analogous to those in statements \ref{stat:R_surfaces_recurrence} and \ref{stat:T_surfaces_recurrence}. This section adopts the notation \(R_n, {}_{h_0} T_n\) to avoid excessive indices, although in the general case \(R_n, {}_{h_0} T_n\) in this section differ from \(R_n, \, {}_{2n_0} T_n\) of the previous section. This is also necessary because, in odd dimensions, the number of initial polytopes \(h_0\) is not always even, which makes the surface notation \(2n_0\) inapplicable.

\subsection{Polytope Tree}

\begin{stat} \label{stat:R_general_recurrence}
Let \(h > 2\) be the number of codimension \(1\) hyperfaces of polytope, \(r\) the size of the rotation/automorphism group of the polytope's hyperfaces, and \(\tau\) the size of the rotation group of the polytope itself. Let \(R_n\) be the number of non-isomorphic ways to arrange numbers on hyperfaces in a polytope tree. Then for \(n \geq 2\)
\begin{equation} \label{eq:R_general_recurrence}
(n-1)R_n = \frac{r}{2} \sum\limits_{k=1}^{n-1} \binom{hn-2n+2}{hk-2k+1} R_k R_{n-k},    
\end{equation}
\end{stat}

\begin{proof}
The proof proceeds analogously to that of \ref{stat:R_surfaces_recurrence}; several details are omitted for conciseness. Fix a split \(n=k+(n-k)\) with \(1\le k\le n-1\). Every polytope tree  can be decomposed into two subtrees glued together along one external hyperface of codimension-\(1\), yielding the product \(R_{k} R_{n-k}\). However, the count of external hyperfaces is now \(hn - 2n + 2\), resulting in \(\binom{hn-2n+2}{hk-2k+1}\) possible ways to choose numbers for each subtree. During the gluing, we were able to rotate the codimension-\(1\) hyperface along which the gluing occurred, which gives us a multiplication by the order \(r\) of the rotation group. The number of seams, as well as the symmetry of \(k\) and \(n-k\), is preserved; thus, we obtain the recurrence relation
\[
(n-1)R_n = \frac{r}{2} \sum\limits_{k=1}^{n-1} \binom{hn-2n+2}{hk-2k+1} R_k R_{n-k}.   
\]
\end{proof}

\begin{Th} \label{thm:general_R}
Let \(h>2\), \(\tau\), and \(r\) be natural numbers and let \(R_1 = \dfrac{h!}{\tau}\). Suppose we have a recurrence relation
\begin{equation*}
(n-1)R_n = \frac{r}{2} \sum\limits_{k=1}^{n-1} \binom{hn-2n+2}{hk-2k+1} R_k R_{n-k}.    
\end{equation*}
Define the exponential generating function \(\mathcal{R}(t) = \sum\limits_{n=1}^{\infty} \dfrac{R_n t^{(h-2)n+2}}{((h-2)n+2))!}\).
Then
\begin{equation} \label{eq:mathcalR_general_differential_equation}
r(h-2) \left( \frac{\mathrm{d}\mathcal{R}}{\mathrm{d}t} \right)^2 - 2 t \frac{\mathrm{d}\mathcal{R}}{\mathrm{d}t} + 2 h \mathcal{R}(t) = 0,
\end{equation}
\begin{equation} \label{eq:R_general_explicit}
R_n = \frac{1}{r} \left( \frac{hr}{\tau} \right)^{n} \frac{(n(h-1))!}{n!}.
\end{equation}
\end{Th}

\begin{proof}
The technique will be analogous to the argument in section \ref{sec:R_surfaces}. Let us consider \((\mathcal{R}'(t))^2\)
\[
(\mathcal{R}'(t))^2 = \left( \sum_{k=1}^{\infty} R_k \frac{t^{(h-2)k + 1}}{((h-2)k + 1)!} \right) \left( \sum_{m=1}^{\infty} R_m \frac{t^{(h-2)m + 1}}{((h-2)m + 1)!} \right).
\]
\[
(\mathcal{R}'(t))^2 = \sum_{n=2}^{\infty} \sum_{k=1}^{n-1} R_k R_{n-k} \frac{t^{(h-2)n + 2}}{((h-2)k + 1)! ((h-2)(n-k) + 1)!}.
\]
Denote \(m_n = (h-2)n + 2\) and \(s_k = (h-2)k + 1\), then
\[
(\mathcal{R}'(t))^2 = \sum_{n=2}^{\infty} t^{m_n} \sum_{k=1}^{n-1} R_k R_{n-k} \frac{1}{s_k! s_{n-k}!} = \sum_{n=2}^{\infty} \frac{t^{m_n}}{m_n!} \sum_{k=1}^{n-1} R_k R_{n-k} \binom{m_n}{s_k}.
\]
From \eqref{eq:R_general_recurrence}
\[
(\mathcal{R}'(t))^2 = \sum_{n=2}^{\infty} \frac{t^{m_n}}{m_n!} \cdot \frac{2}{r} (n-1) R_n = \frac{2}{r} \sum_{n=2}^{\infty} (n-1) R_n \frac{t^{m_n}}{m_n!} = \frac{2}{r} \sum_{n=1}^{\infty} (n-1) R_n \frac{t^{m_n}}{m_n!}.
\]
Since \(m_n = (h-2)n + 2\), then:
\[
n = \frac{m_n - 2}{h-2}, \quad n-1 = \frac{m_n - 2}{h-2} - 1 = \frac{m_n - h}{h-2}.
\]
Thus
\[
(\mathcal{R}'(t))^2 = \frac{2}{r} \sum_{n=1}^{\infty} \frac{m_n - h}{h-2} R_n \frac{t^{m_n}}{m_n!} = \frac{2}{r(h-2)} \sum_{n=1}^{\infty} (m_n - h) R_n \frac{t^{m_n}}{m_n!}.
\]
Since
\[
\sum_{n=1}^{\infty} (m_n - h) R_n \frac{t^{m_n}}{m_n!} = t\mathcal{R}' - h\mathcal{R},
\]
then
\[
(\mathcal{R}'(t))^2 = \frac{2}{r(h-2)} \left( t \mathcal{R}'(t) - h \mathcal{R}(t) \right).
\]
Multiplying both sides by \(r(h-2)\) and moving all terms to one side, we obtain the differential equation
\[
r(h-2) \left( \frac{\mathrm{d}\mathcal{R}}{\mathrm{d}t} \right)^2 - 2 t \frac{\mathrm{d}\mathcal{R}}{\mathrm{d}t} + 2 h \mathcal{R}(t) = 0.
\]
Differentiating the equation and introducing the notation \( p = \dfrac{\mathrm{d}\mathcal{R}}{\mathrm{d}t} \)
\[
2r(h-2) p \frac{\mathrm{d}p}{\mathrm{d}t} - 2 \left( t \frac{\mathrm{d}p}{\mathrm{d}t} + p \right) + 2h p = 0.
\]
\[
r(h-2) p \frac{\mathrm{d}p}{\mathrm{d}t} - t \frac{\mathrm{d}p}{\mathrm{d}t} + (h - 1) p = 0.
\]
Now treating \( t \) as a function of \( p \), we obtain
\[
[r(h-2) p - t] \frac{1}{\frac{\mathrm{d}t}{\mathrm{d}p}} + (h - 1) p = 0.
\]
\[
(h - 1) p \frac{\mathrm{d}t}{\mathrm{d}p} - t = -r(h-2) p.
\]
Dividing by \( p \) (partial solutions are not suitable)
\begin{equation} \label{eq:variation_method}
(h - 1) \frac{\mathrm{d}t}{\mathrm{d}p} - \frac{t}{p} = -r(h-2).
\end{equation}
Let us solve this equation using the method of variation of constants
\[
(h - 1) \frac{\mathrm{d}t}{\mathrm{d}p} - \frac{t}{p} = 0.
\]
\[
t = C p^{\frac{1}{h-1}}.
\]
Let \( t = C(p) p^{\frac{1}{h-1}} \). Then
\[
\frac{\mathrm{d}t}{\mathrm{d}p} = C'(p) p^{\frac{1}{h-1}} + C(p) \cdot \frac{1}{h-1} p^{\frac{1}{h-1} - 1}.
\]
Substituting in \eqref{eq:variation_method}
\[
(h - 1) \left[ C'(p) p^{\frac{1}{h-1}} + \frac{1}{h-1} C(p) p^{\frac{1}{h-1} - 1} \right] - \frac{1}{p} C(p) p^{\frac{1}{h-1}} = -r(h-2).
\]
Simplify it
\[
(h - 1) C'(p) p^{\frac{1}{h-1}} + C(p) p^{\frac{1}{h-1} - 1} - C(p) p^{\frac{1}{h-1} - 1} = -r(h-2),
\]
\[
(h - 1) C'(p) p^{\frac{1}{h-1}} = -r(h-2).
\]
Hence
\[
C'(p) = -\frac{r(h-2)}{h-1} p^{-\frac{1}{h-1}}.
\]
Integrating
\[
C(p) = -\frac{r(h-2)}{h-1} \int p^{-\frac{1}{h-1}} \mathrm{d}p = -\frac{r(h-2)}{h-1} \cdot \frac{p^{1 - \frac{1}{h-1}}}{1 - \frac{1}{h-1}} + C_1 = -r p^{\frac{h-2}{h-1}} + C_1,
\]
where \( C_1 \) is the constant of integration. Hence
\[
t(p) = \left( -r p^{\frac{h-2}{h-1}} + C_1 \right) p^{\frac{1}{h-1}} = C_1 p^{\frac{1}{h-1}} - r p.
\]
Moving the terms to opposite sides and raising to the power of \(h-1\), we get
\[
(t + r p)^{h-1} = C_2 p.
\]
Let us introduce substitution \(y = t + r p, p = \mathcal{R}' = \frac{y - t}{r}, A = \frac{C_2}{r}\)
\begin{equation} \label{eq:y_general_equation}
y^{h-1} = C_2 \frac{y - t}{r} \implies y^{h-1} - A y + A t = 0.    
\end{equation}
Now, using a rescaling, we reduce the solution of equation \eqref{eq:y_general_equation} to Corollary \ref{cor:lagrange_application2}
\[
y =: A^{\frac{1}{h-2}}\, z.
\]  
Substituting into the equation gives  
\[
A^{\frac{h-1}{h-2}}(z^{h-1} - z) + A t = 0, \quad A^{\frac{1}{h-2}}(z^{h-1} - z) + t = 0.
\]  
Thus
\[
z^{h-1} - z + s = 0,\qquad \text{where } s := \frac{t}{A^{\frac{1}{h-2}}}.
\]  
By Corollary \ref{cor:lagrange_application2}, the solution \(z(s)\) has the series expansion  
\[
z(s) = \sum_{k=0}^{\infty} \frac{1}{(h-2)k+1}\binom{(h-1)k}{k}\, s^{(h-2)k+1}.
\]  
Returning to the original variables,  
\[
\begin{aligned}
y(t) &= A^{\frac{1}{h-2}} z(s) \\[2mm]
&= A^{\frac{1}{h-2}} \sum_{k=0}^{\infty}
\frac{1}{(h-2)k+1}\binom{(h-1)k}{k}
\left(\frac{t}{A^{\frac{1}{h-2}}}\right)^{(h-2)k+1} \\[2mm]
&= \sum_{k=0}^{\infty}
\frac{1}{(h-2)k+1}\binom{(h-1)k}{k}\,
A^{-k}\, t^{(h-2)k+1}.
\end{aligned}
\]
Hence
\begin{align*}
\mathcal{R}' &= \frac{1}{r} (y - t) = \frac{1}{r} \left(-t + \sum \limits_{k=0}^{\infty} A^{-k} \frac{(k(h-1))!}{k!} \frac{t^{k(h-2)+1}}{(k(h-2)+1)!}\right) \\
&= \frac{1}{r} \left(\sum \limits_{k=1}^{\infty} A^{-k} \frac{(k(h-1))!}{k!} \frac{t^{k(h-2)+1}}{(k(h-2)+1)!}\right).   
\end{align*}
\[
\mathcal{R} = \frac{1}{r} \sum \limits_{k=1}^{\infty} A^{-k} \frac{(k(h-1))!}{k!} \frac{t^{k(h-2)+2}}{(k(h-2)+2)!} + C_3.
\]
From \(\mathcal{R}(0)=0\), we deduce that \(C_3=0\). It remains to analyze the term \(A = \frac{C_2}{r}\), which encodes the initial condition. Note that
\[
\left.\frac{\mathrm{d}^{\,h-1}}{\mathrm{d}t^{\,h-1}} \left(y^{h-1}\right) \right|_{t=0} = \left.\frac{\mathrm{d}^{\,h-1}}{\mathrm{d}t^{\,h-1}} \left((t+rp)^{h-1}\right) \right|_{t=0} = \left.\frac{\mathrm{d}^{\,h-1}}{\mathrm{d}t^{\,h-1}} \left(t^{h-1}+\cdots\right) \right|_{t=0} = (h-1)!.
\]
On the other hand \(y^{h-1} = \frac{C_2}{r}\left(y-t\right)\) and
\[
\left.\frac{\mathrm{d}^{\,h-1}}{\mathrm{d}t^{\,h-1}} \left( \frac{C_2}{r} (y - t) \right) \right|_{t=0} = \left.\frac{\mathrm{d}^{\,h-1}}{\mathrm{d}t^{\,h-1}} \left( \frac{C_2}{r} r \mathcal{R}' \right) \right|_{t=0} = C_2 R_1.
\]
Therefore
\[
C_2 = \frac{(h-1)!}{R_1} = \frac{(h-1)!}{\frac{h!}{\tau}} = \frac{\tau}{h}.
\]
Hence, \(A = \frac{C_2}{r} = \frac{\tau}{hr}\), therefore,
\[
R_n = \frac{1}{r} A^{-n} \frac{(n(h-1))!}{n!} = \frac{1}{r} \left( \frac{hr}{\tau} \right)^{n} \frac{(n(h-1))!}{n!}.
\]
Note that by substituting \(h = 4\), \(r = \frac{3!}{2} = 3\), \(\tau = \frac{4!}{2} = 12\), we obtain the result of theorem \ref{thm:R_surfaces_formula}.
\end{proof}
Since we have \(nh - 2(n-1)\) hyperfaces of codimension one in the resulting figure, we can introduce a new variable to eliminate enumerations
\begin{equation} \label{eq:R_general_explicit_no_numeration}
\widehat{R}_n := \frac{R_n}{(n(h-2)+2)!} = \frac{1}{r} \left( \frac{hr}{\tau} \right)^n \frac{(n(h-1))!}{n! \, (n(h-2)+2)!}.
\end{equation}

\begin{corollary} \label{cor:widehat_R_asymptotics}
The asymptotics of \(\widehat{R}_n\) is given by

\[
\widehat{R}_n \sim \frac{\sqrt{h-1}}{r\sqrt{2\pi}(h-2)^{5/2}} n^{-5/2} \left( \frac{hr}{\tau} \frac{(h-1)^{h-1}}{(h-2)^{h-2}} \right)^n.
\] 
\end{corollary}
\begin{proof}

Let us introduce the notation \( s = n(h-2) \). Then the expression takes the form
\[
\widehat{R}_n = \frac{R_n}{(s+2)!} = \frac{1}{r} \left( \frac{hr}{\tau} \right)^n \frac{(s+n)!}{n! \, (s+2)!}.
\]
We can determine its asymptotic behavior using Stirling's formula
\[
(s+n)! \sim \sqrt{2\pi(s+n)} \left( \frac{s+n}{e} \right)^{s+n}, n! \sim \sqrt{2\pi n} \left( \frac{n}{e} \right)^n, (s+2)! \sim \sqrt{2\pi(s+2)} \left( \frac{s+2}{e} \right)^{s+2}.
\]
Consequently, we can write
\[
\frac{(s+n)!}{n! \, (s+2)!} \sim \frac{\sqrt{2\pi(s+n)}}{2\pi \sqrt{n(s+2)}} \cdot \frac{(s+n)^{s+n}}{n^n (s+2)^{s+2}} \cdot e^2.
\]
Since
\begin{align*}
\frac{(s+n)^{s+n}}{n^n (s+2)^{s+2}} 
&= \frac{(s+n)^s}{s^s} \cdot \frac{s^s}{(s+2)^{s+2}} \cdot \frac{(s+n)^n}{n^n} \\
&= \left[ (s+2)^2 \left( 1 + \frac{2}{s} \right)^{s} \right]^{-1} \left(1 + \frac{n}{s} \right)^{s} \left(1 + \frac{s}{n} \right)^{n} \\
&\sim \frac{e^{-2}}{n^2 (h-2)^2} \left(1 + \frac{1}{h-2} \right)^{n(h-2)} \left(1 + h - 2\right)^{n} \\
&\sim \frac{e^{-2}}{n^2 (h-2)^2} \cdot \frac{(h-1)^{n(h-1)}}{(h-2)^{n(h-2)}},
\end{align*}
we get
\begin{align*}
\frac{R_n}{(s+2)!} &\sim  \frac{1}{r} \left( \frac{hr}{\tau} \right)^n e^2  \frac{\sqrt{2\pi(n(h-2)+n)}}{2\pi \sqrt{n(n(h-2)+2)}} \frac{e^{-2}}{n^2(h-2)^2} \frac{(h-1)^{n(h-1)}}{(h-2)^{n(h-2)}}  \\
&\sim \frac{\sqrt{h-1}}{r\sqrt{2\pi}(h-2)^{5/2}} n^{-5/2} \left( \frac{hr}{\tau} \frac{(h-1)^{h-1}}{(h-2)^{h-2}} \right)^n.
\end{align*}
\end{proof}

\begin{rem}
As an example of a polytope, we can take the \(d\)-dimensional simplex. In this case, the number of hyperfaces is \(h = d+1\). The order of its full rotation group is \(\tau = \frac{(d+1)!}{2}\) (which follows from the fact that it is the symmetric group \(S_{d+1}\) modulo orientation), and the order of the rotation group of a hyperface of codimension one is \(r = \frac{d!}{2}\). Substituting these parameter values yields a result consistent with the estimate obtained for the \(d\)-dimensional simplex in \cite{BZ11} and for the three-dimensional case in \cite{BL22}. 
\end{rem}

\begin{rem}
Our general formulas recover
the classical Catalan enumeration of
triangulations of a convex polygon. Indeed, in this case \(h=3\), \(r=1\), and \(\tau=3\). Substituting these values into
\eqref{eq:R_general_explicit} gives
\[
R_n=\frac{(2n)!}{n!},
\qquad
\widehat{R}_n
=\frac{(2n)!}{n!(n+2)!}.
\]
Consequently,
\[
(n+2)\widehat{R}_n
=\frac{1}{n+1}\binom{2n}{n}=C_n,
\]
where \(C_n\) is the \(n\)-th Catalan number.
Combinatorially, a tree of \(n\) triangles
has \(n+2\) boundary edges, and choosing a
distinguished boundary edge gives the usual
rooted triangulations of a polygon with
\(n+2\) sides. We omit the details and leave them to the interested reader.
\end{rem}

\begin{quest}
The works \cite{BeinekePippert1971, HRS82, Beineke1974} contain closely related ideas and results, but use different terminology. It would be interesting to find the connections between our results and those earlier works.
\end{quest}

\subsection{Generalization of Inductive Triangulations}
\begin{stat} \label{stat:T_general_recurrence}
Let \(h > 2\) be the number of codimension \(1\) hyperfaces, \(h_0\) be the number of external facets in the initial complex, \(r\) the size of the rotation/automorphism group of the polytope's hyperfaces, and \(\tau\) the size of the rotation group of the polytope itself. Let \({}_{h_0} T_0\) be the number of non-isomorphic labelings of the initial triangulation. Then for \(n \geq 1\)
\begin{equation} \label{eq:T_general_recurrence}
n \cdot {}_{h_0}T_n = r \sum\limits_{k=0}^{n-1} \binom{h n - 2n + h_0}{h(n-k) - 2(n-k) + 1} {}_{h_0} T_k R_{n-k}.
\end{equation}
\end{stat}

\begin{proof}
The proof proceeds analogously to that of \ref{stat:T_surfaces_recurrence}; several details are omitted for conciseness. Fix a split \(n=k+(n-k)\) with \(0\le k\le n\), we make cuts at the \(n\) attachment faces of polytopes in the resulting boundary of complex, yielding the product \({}_{h_0} T_k R_{n-k}\) of the numbering counts of the subtree and the cut remainder. However, the count of external hyperfaces is now \(h n - 2n + h_0\), giving \(\binom{h n - 2n + h_0}{h(n-k) - 2(n-k) + 1}\) possible ways to choose numbers for each subcomplex. During the gluing, we were able to rotate the codimension \(1\) hyperface along which the gluing occurred, which gives us a multiplication by the order \(r\) of the rotation group. The number of seams is still \(n\); thus
\[
n \cdot {}_{h_0}T_n = r \sum\limits_{k=0}^{n-1} \binom{h n - 2n + h_0}{h(n-k) - 2(n-k) + 1} {}_{h_0} T_k R_{n-k}. 
\]
\end{proof}

\begin{Th} \label{thm:general_T}
Let \(R_1 = \dfrac{h!}{\tau}\), where \(h>2, h_0\), \(\tau\), and \(r\) are natural numbers. Suppose we also have recurrence relations
\begin{equation*}
(n-1)R_n = \frac{r}{2} \sum\limits_{k=1}^{n-1} \binom{hn-2n+2}{hk-2k+1} R_k R_{n-k}.   
\end{equation*}
\begin{equation}
n \cdot {}_{h_0}T_n = r \sum\limits_{k=0}^{n-1} \binom{h n - 2n + h_0}{h(n-k) - 2(n-k) + 1} {}_{h_0} T_k R_{n-k}.
\end{equation}
Define generating functions
\[
{}_{h_0} \mathcal{T}(t) = \sum_{n=0}^{\infty} \frac{{}_{h_0} T_n  t^{h n - 2n + h_0}}{(h n - 2n + h_0)!}, \quad \mathcal{R}(t) = \sum_{n=1}^{\infty} \frac{R_n  t^{h n - 2n + 2}}{(h n - 2n + 2)!}.
\]
Then
\begin{equation} \label{eq:mathcalT_general_differential_equation}
\left[ r(h-2) \mathcal{R}'(t) - t \right] \frac{\mathrm{d}({}_{h_0}\mathcal{T})}{\mathrm{d}t} = -h_0 \cdot {}_{h_0}\mathcal{T}(t),
\end{equation}
\begin{equation} \label{eq:T_general_explicit}
{}_{h_0} T_n = {}_{h_0} T_0 \left( \frac{r h}{\tau} \right)^n \frac{(n(h-1))!}{n!} \binom{n(h-1) + h_0 - 1}{h_0 - 1}.
\end{equation}
\end{Th}
\begin{proof}

Consider the product \({}_{h_0} \mathcal{T}'(t) \mathcal{R}'(t)\)
\[
{}_{h_0} \mathcal{T}'(t) \mathcal{R}'(t) = \left( \sum_{k=0}^{\infty} {}_{h_0} T_k \frac{t^{(h-2)k + h_0 - 1}}{((h-2)k + h_0 - 1)!} \right) \left( \sum_{m=1}^{\infty} R_m \frac{t^{(h-2)m + 1}}{((h-2)m + 1)!} \right).
\]
\[
{}_{h_0} \mathcal{T}'(t) \mathcal{R}'(t) = \sum_{n=1}^{\infty} \sum_{k=0}^{n-1} {}_{h_0} T_k R_{n-k} \frac{t^{(h-2)n + h_0}}{((h-2)k + h_0 - 1)!  ((h-2)(n-k) + 1)!}.
\]
Denote \(m_n = h n - 2n + h_0 = (h-2)n + h_0\).
From recurrence relation \eqref{eq:T_general_recurrence}
\[
\sum_{k=0}^{n-1} \binom{m_n}{(h-2)(n-k) + 1} {}_{h_0} T_k R_{n-k} = \frac{n}{r} \, {}_{h_0} T_n.
\]
Substituting into the expression for \({}_{h_0}\mathcal{T}'(t) \mathcal{R}'(t)\)
\[
{}_{h_0}\mathcal{T}'(t) \mathcal{R}'(t) = \sum_{n=1}^{\infty} \frac{t^{m_n}}{m_n!} \sum_{k=0}^{n-1} \binom{m_n}{(h-2)(n-k) + 1} {}_{h_0} T_k R_{n-k} = \sum_{n=1}^{\infty} \frac{t^{m_n}}{m_n!} \cdot \frac{n}{r} {}_{h_0} T_n.
\]
Since \(n = \frac{m_n - h_0}{h-2}\), then
\begin{align*}
{}_{h_0}\mathcal{T}'(t) \mathcal{R}'(t) &= \frac{1}{r} \sum_{n=1}^{\infty} \frac{m_n - h_0}{h-2}  {}_{h_0} T_n \frac{t^{m_n}}{m_n!}  \\
&= \frac{1}{r(h-2)} \sum_{n=1}^{\infty} \left( m_n - h_0 \right) {}_{h_0} T_n \frac{t^{m_n}}{m_n!} \\
&= \frac{1}{r(h-2)} \sum_{n=0}^{\infty} \left( m_n - h_0 \right) {}_{h_0} T_n \frac{t^{m_n}}{m_n!}.
\end{align*}
Thus
\[
{}_{h_0}\mathcal{T}'(t) \mathcal{R}'(t) = \frac{1}{r(h-2)} \left[ t \frac{\mathrm{d}({}_{h_0}\mathcal{T})}{\mathrm{d}t} - h_0 \cdot {}_{h_0}\mathcal{T}(t) \right].
\]
Move all terms to one side
\[
\left[ r(h-2) \mathcal{R}'(t) - t \right] \frac{\mathrm{d}({}_{h_0}\mathcal{T})}{\mathrm{d}t} = -h_0 \cdot {}_{h_0}\mathcal{T}(t).
\]
Let us rewrite this equation in the form
\[
\frac{\mathrm{d}({}_{h_0}\mathcal{T})}{{}_{h_0}\mathcal{T}} = \frac{h_0  \mathrm{d}t}{t - r(h-2) \mathcal{R}'(t)}.
\]
Integrate both sides
\[
\int \frac{\mathrm{d}({}_{h_0}\mathcal{T})}{{}_{h_0}\mathcal{T}} = h_0 \int \frac{\mathrm{d}t}{t - r(h-2) \mathcal{R}'(t)} + K_1.
\]
Then
\[
{}_{h_0}\mathcal{T}(t) = K_2 \exp\left(h_0 I(t)\right),
\]
where
\[
I(t) = \int \frac{\mathrm{d}t}{t - r(h-2) \mathcal{R}'(t)}.
\]
Let us use the equation \eqref{eq:y_general_equation}:
\( y^{h-1} - A y + A t = 0,\) where \(A = \dfrac{\tau}{r h}\).
Since \(t = y - \frac{y^{h-1}}{A}\) and \(\mathcal{R}'(t) = \frac{y-t}{r} = \frac{y^{h-1}}{A r}\), then
\[
I(t) = \int \frac{\mathrm{d}\left(y - \frac{y^{h-1}}{A}\right)}{y - \frac{y^{h-1}}{A} - r(h-2) \frac{y^{h-1}}{A r}} = \int \frac{\mathrm{d}y\left(1 - (h-1)\frac{y^{h-2}}{A}\right)}{y - \frac{y^{h-1}}{A} - (h-2) \frac{y^{h-1}}{A}} = \int \frac{\mathrm{d}y}{y} = \ln |y| + K_0.
\]
Hence
\[
{}_{h_0}\mathcal{T}(t) = K_3 y^{h_0} =: K_3 g(y),
\]
It remains to apply Theorem \ref{thm:lagrange_bürmann} (\(y(0) = 0, g'(0) = 1-\frac{(h-1)}{A}y^{h-2}(0) \neq 0\))
\begin{align*}
{}_{h_0}\mathcal{T}(y(t)) &= {}_{h_0} \mathcal{T}(0) + \sum_{j=1}^{\infty} \frac{t^j}{j!} \left[ \frac{\mathrm{d}^{j-1}}{\mathrm{d}y^{j-1}} \left( {}_{h_0}\mathcal{T}'(u) \left(\frac{u}{g(u)}\right)^j \right) \right]_{u=0} \\
&= \sum_{j=1}^{\infty} \frac{t^j}{j!} \left[ \frac{\mathrm{d}^{j-1}}{\mathrm{d}y^{j-1}} \left( h_0 u^{h_0-1} \left(1-\frac{1}{A}u^{h-2}\right)^{-j} \right) \right]_{u=0}.
\end{align*}
We are looking for the coefficient of \(u^{j-1}\) in the Maclaurin expansion of
\[
h_0 u^{h_0-1} \left(1-\frac{1}{A}u^{h-2}\right)^{-j} = h_0 u^{h_0-1} \sum_{k=0}^{\infty} \binom{j + k - 1}{k} \left(\frac{u^{h-2}}{A} \right)^{k}.
\]
That is, we need to choose \(k\) such that \(k(h-2)+h_0-1 = j-1 \implies j = k(h-2)+h_0\) (otherwise the coefficient will be \(0\)). This means that
\begin{align*}
\left[ \frac{\mathrm{d}^{j-1}}{\mathrm{d}u^{j-1}} \left( {}_{h_0} \mathcal{T}'(u) \left(\frac{u}{g(u)}\right)^j \right) \right]_{u=0} &= (j-1)! \cdot h_0 \binom{j+k-1}{k} \left(A \right)^{-k} \\
&=(k(h-2)+h_0-1)! \cdot h_0 \binom{k(h-2)+h_0 + k-1}{k} \left(A \right)^{-k}.
\end{align*}
Now it remains to compare with \({}_{h_0} \mathcal{T}(t)\)
\[
{}_{h_0} \mathcal{T}(t) = K_3  \sum_{k=0}^{\infty} \frac{t^{k(h-2)+h_0} \left(A \right)^{-k}}{k(h-2)+h_0} h_0 \binom{k(h-1)+h_0-1}{k} = \sum\limits_{n = 0}^{\infty} {}_{h_0} T_n \frac{t^{(h-2)n+h_0}}{((h-2)n+h_0)!}.
\]
This means that
\begin{align*}
\frac{{}_{h_0} T_k}{((h-2)k+h_0)!} 
&= K_3 \,\frac{h_0 \binom{k(h-1)+h_0-1}{k}}{k(h-2)+h_0} \, A^{-k} \\
&= K_3 \,\frac{h_0 \, (k(h-1)+h_0-1)! \, A^{-k}}{k! \, (k(h-2)+h_0) \, (k(h-2)+h_0-1)!}.
\end{align*}
Thus
\[
{}_{h_0} T_k = K_3 \frac{h_0(k(h-1)+h_0-1)!}{k!} \left(A \right)^{-k}.
\]
Observe that for \(k = 0\): \({}_{h_0}T_0 = h_0 K_3 (h_0-1)! \, A^{-0} \;\;\Longrightarrow\;\; K_3 = \frac{{}_{h_0}T_0}{h_0!},\) and consequently, \({}_{h_0}\mathcal{T}(t) = \frac{{}_{h_0}T_0}{h_0!}\, y^{h_0}.\) Substituting this expression
\begin{align*}
{}_{h_0} T_k &= \frac{{}_{h_0} T_0}{h_0!} \frac{h_0(k(h-1)+h_0-1)!}{k!} \left(A \right)^{-k} \\
&= {}_{h_0} T_0 \frac{(k(h-1)+h_0-1)!}{(h_0-1)! k!} \left(A \right)^{-k} \\
&= {}_{h_0} T_0  \frac{(k(h-1))!}{k! \left(A \right)^{k}} \binom{k(h-1) + h_0 - 1}{h_0 - 1}.
\end{align*}
\end{proof}
Since our constructed figure has \(h_0 + n h - 2 n\) hyperfaces of codimension one, we can introduce a new variable to eliminate enumerations
\begin{equation} \label{eq:T_general_explicit_no_numeration}
{}_{h_0}\widehat{T}_n := \frac{{}_{h_0} T_n}{(h_0 + n(h - 2))!} = {}_{h_0} T_0 \frac{1}{n!} \cdot \frac{(n(h - 1) + h_0 - 1)!}{(h_0 - 1)! (h_0 + n(h - 2))!} \left(\frac{hr}{\tau} \right)^{n}.  
\end{equation}
\begin{corollary}
The asymptotics of \({}_{h_0}\widehat{T}_n\) is given by
\[
{}_{h_0}\widehat{T}_n \sim {}_{h_0} T_0 \frac{(h - 1)^{h_0 - 1/2}}{(h_0 - 1)! \sqrt{2\pi} (h - 2)^{h_0 + 1/2}} \cdot n^{-3/2} \cdot \left( \frac{r h}{\tau} \frac{(h - 1)^{h - 1}}{(h - 2)^{h - 2}} \right)^n.
\] 
\end{corollary}
\begin{proof}
By Stirling's formula
\begin{align*}
\frac{(n(h - 1) + h_0 - 1)!}{(n(h - 2) + h_0)! n!} &\sim \frac{\sqrt{2\pi (n(h - 1) + h_0 - 1)} \left( \frac{(n(h - 1) + h_0 - 1)}{e} \right)^{(n(h - 1) + h_0 - 1)}}{\sqrt{2\pi (h_0 + n(h - 2))} \left( \frac{(h_0 + n(h - 2))}{e} \right)^{(h_0 + n(h - 2))} \sqrt{2\pi n} \left( \frac{n}{e} \right)^n}  \\
&\sim \frac{e \sqrt{(h - 1)}}{\sqrt{(h - 2)} \sqrt{2\pi n}} \cdot \frac{(n(h - 1) + h_0 - 1)^{n(h - 1) + h_0 - 1}}{(h_0 + n(h - 2))^{h_0 + n(h - 2)} n^n}.
\end{align*}
Let us analyze the following term separately
\begin{align*}
\frac{(n(h - 1) + h_0 - 1)^{n(h - 1) + h_0 - 1}}{(h_0 + n(h - 2))^{h_0 + n(h - 2)} n^n}
&\sim \frac{(n(h - 1))^{n(h - 1) + h_0 - 1} \left(1 + \frac{h_0 - 1}{n(h - 1)}\right)^{n(h - 1) + h_0 - 1}}{n^n (n(h - 2))^{n(h - 2) + h_0} \left(1 + \frac{h_0}{n(h - 2)}\right)^{n(h - 2) + h_0}} \\
&\sim  n^{-1} \frac{(h - 1)^{n(h - 1) + h_0 - 1}}{\left(h - 2\right)^{n(h - 2) + h_0}} \frac{e^{h_0 - 1}}{e^{h_0}} \frac{(1 + \frac{h_0 - 1}{n(h - 1)})^{h_0 - 1}}{\left(1 + \frac{h_0}{n(h - 2)}\right)^{h_0}} \\
&\sim (e n)^{-1} \frac{(h - 1)^{n(h - 1) + h_0 - 1}}{\left(h - 2\right)^{n(h - 2) + h_0}}.
\end{align*}
Thus
\[
\frac{{}_{h_0} T_n}{(h_0 + n(h - 2))!} \sim {}_{h_0} T_0 \frac{(rh)^{n}}{(h_0 - 1)! \tau^n} \cdot \frac{\sqrt{(h - 1)}}{\sqrt{(h - 2)} \sqrt{2\pi n}} \cdot n^{-1} \cdot \left( \frac{(h - 1)^{h - 1}}{(h - 2)^{h - 2}} \right)^n \frac{(h - 1)^{h_0 - 1}}{(h - 2)^{h_0}}.
\]
Simplify
\[
\frac{{}_{h_0} T_n}{(h_0 + n(h - 2))!} \sim {}_{h_0} T_0 \frac{(h - 1)^{h_0 - 1/2}}{(h_0 - 1)! \sqrt{2\pi} (h - 2)^{h_0 + 1/2}} \cdot n^{-3/2} \cdot \left( \frac{r h}{\tau} \frac{(h - 1)^{h - 1}}{(h - 2)^{h - 2}} \right)^n.
\]
\end{proof}
\begin{corollary}
For any triangulable manifold, the number of triangulations grows at least exponentially.
\end{corollary}

\begin{rem} 
Let us establish the relationship between \(R_n\) and \({}_{h_0} T_k\) following the analogy of remark \ref{rem:R_T_connection_surfaces}. Starting with a regular polytope, which has \(h\) hyperfaces of codimension \(1\) (i.e., \(h_0 = h\)) and \({}_{h_0} T_0 = R_1 = \frac{h!}{\tau}\), we need to glue \(n-1\) times, therefore we should calculate the ratio
\begin{align*}
\frac{{}_{h_0} T_{n-1}}{R_n} &= \left( \frac{h!}{\tau} \cdot \frac{((n - 1)(h - 1) + h_0 - 1)!}{(h_0 - 1)! (n - 1)!} \left( \frac{r h}{\tau} \right)^n \right) \div \left( \frac{h^n r^{n - 1}}{\tau^n} \cdot \frac{(n(h - 1))!}{n!} \right) \\
&= h! \cdot \frac{((n - 1)(h - 1) + h_0 - 1)!}{(h_0 - 1)! (n - 1)!} \cdot \frac{r}{\tau} \cdot \frac{n!}{(n(h - 1))!} \\
&= h! \cdot \frac{((n - 1)(h - 1) + h_0 - 1)!}{(h_0 - 1)!} \cdot \frac{n!}{(n - 1)!} \cdot \frac{r}{\tau} \cdot \frac{1}{(n(h - 1))!}.
\end{align*}
Using the fact that \(\frac{n!}{(n - 1)!} = n\) and \(h_0 = h\), we obtain
\[
\frac{{}_{h_0} T_{n-1}}{R_n} = n \cdot h! \cdot \frac{((n - 1)(h - 1) + h - 1)!}{(h - 1)!} \cdot \frac{r}{\tau} \cdot \frac{1}{(n(h - 1))!} = n \frac{r h}{\tau}.
\]
For a regular polytope, the relation \( rh = \tau \) holds. This can be verified by direct computation for all regular polytopes or derived from the orbit-stabilizer theorem applied to the action of the automorphism/rotation group on the polytope. With this identity, we obtain a result identical to that stated in remark \ref{rem:R_T_connection_surfaces}. This also simplifies the explicit formulas \eqref{eq:R_general_explicit} and \eqref{eq:T_general_explicit}. However, we have chosen to retain the notation \(\tau\), as recurrence relations of this form can arise not only for regular polytopes but also in other mathematical contexts where the initial conditions may be more complex.
\end{rem}
For cases where \( 2 \leq h-1 \leq 4 \) it is possible to obtain explicit expression for \(y\) from \eqref{eq:y_general_equation} and hence for the generating functions. See Appendix \ref{sec:explicit_expressions_generating_function} for the complete analysis.

Also, we can obtain the same formulas for \(\widehat{R}_n\) and \({}_{h_0}\widehat{T}_n\) by another method. For details, see the Appendix \ref{sec:another_way_R_T}.

\begin{rem}
From theorems \ref{thm:general_R} and \ref{thm:general_T}, we can derive generalizations of the equalities \eqref{eq:R_surfaces_combinatorial} and \eqref{eq:T_surfaces_combinatorial}. To do this, we substitute the explicit expression \eqref{eq:R_general_explicit} into recurrence \eqref{eq:R_general_recurrence}
\begin{align*}
    (n-1) \frac{1}{r} \left( \frac{hr}{\tau} \right)^{n} \frac{(n(h-1))!}{n!}
    = \frac{r}{2} \sum_{k=1}^{n-1} &\binom{hn-2n+2}{(h-2)(n-k)+1}
    \frac{1}{r} \left( \frac{hr}{\tau} \right)^{k} \frac{(k(h-1))!}{k!} \\
    &\times \frac{1}{r} \left( \frac{hr}{\tau} \right)^{n-k} \frac{((n-k)(h-1))!}{(n-k)!}.
\end{align*}
From this, for any natural number \(h>2\), we obtain an impressive relation for binomial coefficients that generalizes \eqref{eq:R_surfaces_combinatorial}
\[
(n-1)\frac{(n(h-1))!}{n!} = \frac{1}{2} \sum\limits_{k=1}^{n-1} \binom{hn-2n+2}{hk-2k+1} \frac{(k(h-1))!}{k!} \frac{((n-k)(h-1))!}{(n-k)!}.
\]
\begin{equation} \label{eq:R_general_combinatorial}
2(n-1) = \sum\limits_{k=1}^{n-1} \binom{(h-2)n+2}{(h-2)k+1} \frac{\binom{n}{k}}{\binom{n(h-1)}{k(h-1)}}.
\end{equation}
By analogy, substitute the explicit formula \eqref{eq:T_general_explicit} into the recurrence \eqref{eq:T_general_recurrence}  (immediately dividing both sides by \({}_{h_0} T_0\) and \(\left( \frac{hr}{\tau} \right)^n\))
\begin{align*}
\frac{(n(h-1))!}{n!} \binom{n(h-1)+h_0-1}{h_0-1}
    &= \frac{r}{n} \sum_{k=0}^{n-1} 
       \binom{(h-2)n+h_0}{(h-2)(n-k)+1} 
       \frac{(k(h-1))!}{k!} \\
    &\quad \times \binom{k(h-1)+h_0-1}{h_0-1} 
       \frac{1}{r} 
       \frac{((n-k)(h-1))!}{(n-k)!}.
\end{align*}
Thus, for any natural numbers \(h>2\) and \(h_0\), we derive a notable equality generalizing \eqref{eq:T_surfaces_combinatorial}
\begin{equation} \label{eq:T_general_combinatorial}
\binom{n(h-1)+h_0-1}{h_0-1} = \frac{1}{n} \sum\limits_{k=0}^{n-1} \binom{(h-2)n+h_0}{(h-2)(n-k)+1} \binom{k(h-1)+h_0-1}{h_0-1} \frac{\binom{n}{k}}{\binom{n(h-1)}{k(h-1)}}.
\end{equation}
By substituting \(h = 2\) we obtain the hockey-stick identity (after making the substitution \(j = k + h_0 - 1\) and \(r = h_0 - 1\), and using the identity \(\binom{n+h_0-1}{h_0} = \frac{n}{h_0} \binom{n+h_0-1}{h_0-1}\)); the details are left to the reader.

Also note that for \( h_0 = 1 \), the formula for an inductive complex reduces to the formula for polytope trees. This has a geometric explanation: if there is only a single hyperface to which we can glue, then we glue exactly a polytope tree at step \( n \). The same fact can also be understood via Lagrange inversion (for inductive polytopes we seek an expansion for \( y^{h_0} \), while for polytope trees we seek \( y \) --- after which we subtract \( t \), divide by a constant, and integrate). Moreover, \eqref{eq:T_general_combinatorial} implies \eqref{eq:R_general_combinatorial}; details are given in Appendix~\ref{sec:T_implies_R}.
\end{rem}

\begin{rem}

Natural generalizations of polytope trees (for the simplex case) were considered in \cite{DJ94,BZ11}, where they were termed LC-triangulations. Using the knowledge of \(R_n\), one can attempt to derive a formula for LC-triangulations, or at least sharpen the estimates, and similarly extend this approach to inductive triangulations (gluing not only trees but also LC-triangulations).
\end{rem}

\begin{quest}
There exist discrete formulations of string theory; the partition function is given by a sum over all non‑isomorphic triangulations of a \(d\)-dimensional manifold. Could one consider a simplified problem: summing only over the triangulations described in this work?
\end{quest}

\begin{quest}
From a purely geometric perspective, it might be interesting to consider polytope trees in Euclidean space that are embedded without self-intersections. Currently, our model does not account for this, and it is unclear how to impose such geometric constraints.
\end{quest}

\section{Cycles in the Dual Graph} \label{sec:cycles}
This section extends our construction from trees of polytopes to complexes that can contain cycles. We derive a recurrence relation that enumerates non‑isomorphic labelings for a connected complex of \(n\) polytopes, where the dual graph has cyclomatic number (first Betti number) of \(\mathsf{b}_1\).

\begin{stat} \label{sec:stat_U_recurrence}
Let \(U_{n,\mathsf{b}_1}\) be the number of non-isomorphic labelings of connected polytopal complex consisting of \(n\) polytopes with circuit rank \(\mathsf{b}_1\), and let \(h > 2\) be the number of codimension one hyperfaces of the polytope, and \(r\) be the size of the automorphism/rotation group of these hyperfaces. Then the following recurrence relation holds
\begin{equation} \label{eq:U_cycles_recurrence}
(n-1+\mathsf{b}_1) U_{n,\mathsf{b}_1} = \frac{r}{2} U_{n,\mathsf{b}_1-1} + \frac{r}{2} \sum_{\mathsf{s}=0}^{\mathsf{b}_1} \sum_{l=1}^{n-1} U_{n-l,\mathsf{b}_1-\mathsf{s}} \, U_{l,\mathsf{s}} \binom{(h-2)n + 2 - 2\mathsf{b}_1}{(h-2)l + 1 - 2\mathsf{s}}.    
\end{equation}
\end{stat}
\begin{proof}
The proof is analogous to the case without cycles with modifications. We also label the external hyperfaces, but the number of gluing places (seams) is \(n-1+\mathsf{b}_1\) (we added \(\mathsf{b}_1\) gluings due to the number of independent cycles, i.e., the cyclomatic number). Summing over all possible splits yields the left-hand side of the recurrence relation.

Unlike the tree case, the right-hand side of the recurrence involves two scenarios, as illustrated in Figure~\ref{fig:dual_cut}.

\begin{figure}[!htbp]
    \centering
    \includegraphics[width=0.90\textwidth]{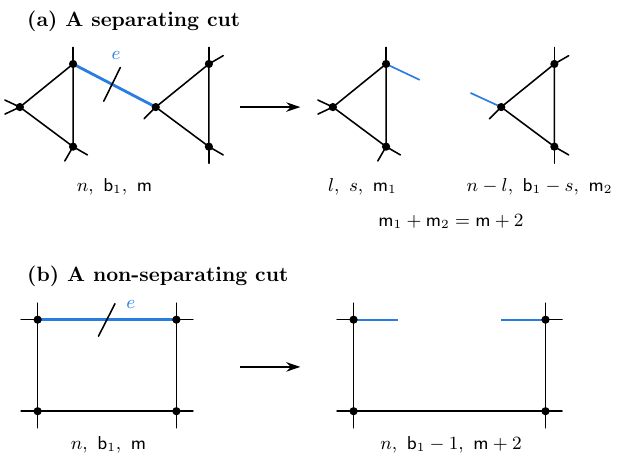}
    \caption{Separating and non-separating cuts in the dual graph.}
    \label{fig:dual_cut}
\end{figure}

\begin{enumerate}
    \item The case where the cut yields a single connected component. This happens when we cut the complex in a way that reduces the cyclomatic number by \(1\) but does not split it into two pieces. Such a cut gives the term \( \frac{r}{2} U_{n,\mathsf{b}_1-1} \). Here, \( r \) is the order of the rotation group of the hyperface (this follows from the fact that we can rotate the hyperface of polytope). The division by \( 2 \) accounts for the symmetry in swapping the labels of the last and second-to-last hyperfaces chosen for gluing.

\item The second scenario corresponds to a reduction in the cyclomatic number, which splits the complex into two connected components containing \(l\) and \(n-l\) polytopes, respectively. These components have cyclomatic numbers \(\mathsf{s}\) and \(\mathsf{b}_1 - \mathsf{s}\).  
Initially there are \(hn\) free hyperfaces. Building a connected complex of \(n\) polytopes requires exactly \(n-1\) gluings (edges of the dual graph); each gluing removes two free hyperfaces, leaving \(hn - 2(n-1)\) free hyperfaces. The cyclomatic number \(\mathsf{b}_1\) counts independent cycles in the dual graph, each consuming two additional free hyperfaces. Consequently, the total number of external faces is  
\[
hn - 2(n-1) - 2\mathsf{b}_1 = (h-2)n + 2 - 2\mathsf{b}_1.
\]  
(Observe that since no further gluings are possible (i.e., all free hyperfaces have been accounted for), we obtain the bound \(\mathsf{b}_1 \le \frac{(h-2)n+2}{2}\)). This yields the product \(U_{n-l,\mathsf{b}_1-\mathsf{s}} \, U_{l,\mathsf{s}}\) together with a binomial coefficient \(\binom{(h-2)n + 2 - 2\mathsf{b}_1}{(h-2)l + 1 - 2\mathsf{s}}\). Also we must account for rotations of the hyperfaces and the symmetry between the pairs \((l, n-l)\) and \((\mathsf{s}, \mathsf{b}_1 - \mathsf{s})\). Taking these into account gives the second term of the right-hand side and, consequently, the desired recurrence.
\end{enumerate}
\end{proof}

Dividing \eqref{eq:U_cycles_recurrence} by \(((h-2)n + 2 - 2\mathsf{b}_1)!\) and introducing \(\widehat{U}_{n,\mathsf{b}_1} = \frac{U_{n,\mathsf{b}_1}}{((h-2)n + 2 - 2\mathsf{b}_1)!}\), we obtain the equivalent recurrence relation
\begin{equation} \label{eq:normalized_U_cycles_recurrence_part1}
(n-1+\mathsf{b}_1) \widehat{U}_{n,\mathsf{b}_1} = \frac{r}{2} \left( S_1 + S_2 \right),    
\end{equation}
where
\begin{equation} \label{eq:normalized_U_cycles_recurrence_part2}
S_1 = \widehat{U}_{n,\mathsf{b}_1-1} \left( (h-2)n + 2 - 2\mathsf{b}_1 + 1 \right) \left( (h-2)n + 2 - 2\mathsf{b}_1 + 2 \right),    
\end{equation}
and
\begin{equation} \label{eq:normalized_U_cycles_recurrence_part3}
S_2 = \sum_{l=1}^{n-1} \sum_{\mathsf{s}=0}^{\mathsf{b}_1} \widehat{U}_{n-l,\mathsf{b}_1-\mathsf{s}} \widehat{U}_{l,\mathsf{s}} \left( (h-2)(n-l) + 2 - 2(\mathsf{b}_1-\mathsf{s}) \right) \left( (h-2)l + 2 - 2\mathsf{s} \right).   
\end{equation}

Let us make the substitutions
\begin{equation} \label{eq:connection_m_b1}
\mathsf{m} = (h-2)n + 2 - 2\mathsf{b}_1 \implies \mathsf{b}_1 = \frac{(h-2)n + 2 - \mathsf{m}}{2}; \quad \widehat{V}_{n,\mathsf{m}} = \widehat{U}_{n,\mathsf{b}_1}.   
\end{equation}
\begin{rem}
Note that \(m\) also has a combinatorial meaning. It denotes the number of unglued hyperfaces such that the construction is connected and its dual graph has first Betti number \(\mathsf{b}_1\).  
\end{rem}
\begin{prop}
Recurrence relation \eqref{eq:normalized_U_cycles_recurrence_part1} can be rewritten using substitution \eqref{eq:connection_m_b1} as
\begin{equation} \label{eq:normalized_V_cycles_recurrence}
\frac{h n - \mathsf{m}}{2} \widehat{V}_{n,\mathsf{m}} = \frac{r}{2} \left[ (\mathsf{m}+1)(\mathsf{m}+2) \widehat{V}_{n, \mathsf{m}+2} + \sum_{l=1}^{n-1} \sum_{\mathsf{m}_1 + \mathsf{m}_2 = \mathsf{m}+2} \mathsf{m}_1 \mathsf{m}_2 \, \widehat{V}_{n-l, \mathsf{m}_1} \widehat{V}_{l, \mathsf{m}_2} \right].
\end{equation}
\end{prop}
\begin{proof}
From \eqref{eq:connection_m_b1} we got
\[
\widehat{U}_{n,\mathsf{b}_1-1} = \widehat{V}_{n, \mathsf{m}+2}, \quad \left( (h-2)n + 3 - 2\mathsf{b}_1 \right) = \mathsf{m}+1, \quad \left( (h-2)n + 4 - 2\mathsf{b}_1 \right) = \mathsf{m}+2.
\]
We can rewrite \eqref{eq:normalized_U_cycles_recurrence_part2} in terms of \(\mathsf{m}\)
\[
S_1 = (\mathsf{m}+1)(\mathsf{m}+2) \widehat{V}_{n, \mathsf{m}+2}.
\]
For \(\widehat{U}_{n-l,\mathsf{b}_1-\mathsf{s}}\) we have
\[
\mathsf{m}_1 = (h-2)(n-l) + 2 - 2(\mathsf{b}_1-\mathsf{s}),
\]
For \(\widehat{U}_{l,\mathsf{s}}\) we have
\[
\mathsf{m}_2 = (h-2)l + 2 - 2\mathsf{s},
\]
and we have \(\mathsf{m}_1 + \mathsf{m}_2 = \mathsf{m}+2\). Therefore, we can rewrite \eqref{eq:normalized_U_cycles_recurrence_part3}
\[
S_2 = \sum_{l=1}^{n-1} \sum_{\mathsf{m}_1 + \mathsf{m}_2 = \mathsf{m}+2} \mathsf{m}_1 \mathsf{m}_2 \, \widehat{V}_{n-l, \mathsf{m}_1} \, V_{l, \mathsf{m}_2}.
\]
The left-hand side factor of \eqref{eq:normalized_U_cycles_recurrence_part1} becomes
\[
n - 1 + \mathsf{b}_1 = n - 1 + \frac{(h-2)n + 2 - \mathsf{m}}{2} = \frac{h n - \mathsf{m}}{2}.
\]
Hence
\[
\frac{h n - \mathsf{m}}{2} \widehat{V}_{n,\mathsf{m}} = \frac{r}{2} \left[ (\mathsf{m}+1)(\mathsf{m}+2) \widehat{V}_{n, \mathsf{m}+2} + \sum_{l=1}^{n-1} \sum_{\mathsf{m}_1 + \mathsf{m}_2 = \mathsf{m}+2} \mathsf{m}_1 \mathsf{m}_2 \, \widehat{V}_{n-l, \mathsf{m}_1} \widehat{V}_{l, \mathsf{m}_2} \right].
\]
\end{proof}

The combinatorial-geometric argument in the statement \ref{sec:stat_U_recurrence} yields \(\widehat{U}_{n, \mathsf{b}_1} = 0\) whenever \(\mathsf{b}_1 > \frac{(h-2)n+2}{2}\). The following lemma provides an algebraic proof of this fact.

\begin{lem} \label{lem:correctness_U}
Substituting the condition \( \mathsf{b}_1 > \frac{(h-2)n+2}{2} \) into \eqref{eq:normalized_U_cycles_recurrence_part1} forces \(\widehat{U}_{n, \mathsf{b}_1} = 0\), independent of the geometric interpretation.
\end{lem}
\begin{proof}
    It suffices to show that \(\widehat{V}_{n, \mathsf{m}} = 0\) for all \(\mathsf{m} < 0\); see Appendix \ref{sec:correctness_U}.
\end{proof}

\begin{stat} \label{stat:V_differential_equation}
Suppose we have recurrence relation \eqref{eq:normalized_V_cycles_recurrence} and generating function
\[
\widehat{\mathcal{V}}(t_0,t_1) = \sum_{n \ge 1} \sum_{\mathsf{m} \ge 0} \widehat{V}_{n,\mathsf{m}} t_0^n t_1^{\mathsf{m}}.
\]
Then
\begin{equation} \label{eq:V_differential_equation}
h t_0 \frac{\partial \widehat{\mathcal{V}}}{\partial t_0} - t_1 \frac{\partial \widehat{\mathcal{V}}}{\partial t_1}
= r \frac{\partial^2 \widehat{\mathcal{V}}}{\partial t_1^2} + r \left( \frac{\partial \widehat{\mathcal{V}}}{\partial t_1} \right)^2 .
\end{equation}
\end{stat}
\begin{proof}
The left-hand side of \eqref{eq:normalized_V_cycles_recurrence} can be written as follows
\[
\sum_{n \ge 1} \sum_{\mathsf{m} \ge 0} \left( h n - \mathsf{m} \right) \widehat{V}_{n,\mathsf{m}} \, t_0^n t_1^{\mathsf{m}}
= h\, t_0 \frac{\partial \widehat{\mathcal{V}}}{\partial t_0} - t_1 \frac{\partial \widehat{\mathcal{V}}}{\partial t_1}.
\]
First term on the right-hand side of \eqref{eq:normalized_V_cycles_recurrence} is given by
\[
\sum_{n \ge 1} \sum_{\mathsf{m} \ge 0} (\mathsf{m}+1)(\mathsf{m}+2) \, V_{n,\mathsf{m}+2} \; t_0^n t_1^{\mathsf{m}}
= \frac{\partial^2 \widehat{\mathcal{V}}}{\partial t_1^{2}} .
\]
The second term on the right-hand side of \eqref{eq:normalized_V_cycles_recurrence} is
\begin{equation} \label{eq:convolution_in_terms_m}
\sum_{n \ge 1} \sum_{\mathsf{m} \ge 0} 
\sum_{\substack{1 \le l \le n-1 \\ \mathsf{m}_1 + \mathsf{m}_2 = \mathsf{m}+2}} 
\mathsf{m}_1 \mathsf{m}_2 \; \widehat{V}_{n-l,\,\mathsf{m}_1} \, \widehat{V}_{l,\,\mathsf{m}_2} \; t_0^n t_1^{\mathsf{m}} .   
\end{equation}
Due to \(t_0^n = t_0^{\,n-l} t_0^{\,l}\) and \(t_1^{\mathsf{m}} = t_1^{\mathsf{m}_1 + \mathsf{m}_2 - 2}\), the sum \eqref{eq:convolution_in_terms_m} splits into a product
\[
\left( \sum_{n_1 \ge 1} \sum_{\mathsf{m}_1 \ge 0} \mathsf{m}_1 \, \widehat{V}_{n_1,\mathsf{m}_1} \, t_0^{\,n_1} t_1^{\mathsf{m}_1-1} \right)
\cdot
\left( \sum_{n_2 \ge 1} \sum_{\mathsf{m}_2 \ge 0} \mathsf{m}_2 \, V_{n_2,\mathsf{m}_2} \, t_0^{\,n_2} t_1^{\mathsf{m}_2-1} \right) = \left( \frac{\partial \widehat{\mathcal{V}}}{\partial t_1} \right)^2.
\]
By collecting all terms on the left‑hand and right‑hand sides of the recurrence, we obtain the PDE \eqref{eq:V_differential_equation}.
\end{proof}

\begin{stat} \label{stat:exp_V_series}
Consider the differential equation \eqref{eq:V_differential_equation}. Then
\[
\widehat{\mathcal{V}} = \ln \left( 1 + \sum_{\substack{n \ge 1 \\ \mathsf{m} \ge 0 \\ \mathsf{m} \le hn \\ hn \equiv \mathsf{m} \pmod{2}}} \left(\frac{1}{\tau}\right)^n \frac{1}{n!} \cdot \left( \frac{r}{2} \right)^{\frac{hn - \mathsf{m}}{2}} \cdot \frac{(hn)!}{\mathsf{m}! \cdot \left( \frac{hn - \mathsf{m}}{2} \right)!} \, t_0^{n} t_1^{\mathsf{m}} \right).
\]

\end{stat}

\begin{proof}
Let \(\widehat{\mathcal{V}}(t_0,t_1) = \ln \widehat{\mathcal{W}}(t_0,t_1) \iff \widehat{\mathcal{W}}(t_0,t_1) = e^{\widehat{\mathcal{V}}(t_0,t_1)}\) (combinatorial meaning of \(\widehat{\mathcal{W}}(t_0,t_1)\) is that it counts labeled objects that may be disconnected, consisting of labeled connected components). Then
\[
\frac{\partial \widehat{\mathcal{V}}}{\partial t_0} = \frac{ \frac{\partial \widehat{\mathcal{W}}}{\partial t_0} }{ \widehat{\mathcal{W}} }, \qquad
\frac{\partial \widehat{\mathcal{V}}}{\partial t_1} = \frac{ \frac{\partial \widehat{\mathcal{W}}}{\partial t_1} }{ \widehat{\mathcal{W}} }, \qquad
\frac{\partial^2 \widehat{\mathcal{V}}}{\partial t_1^2} = \frac{ \frac{\partial^2 \widehat{\mathcal{W}}}{\partial t_1^2} \, \widehat{\mathcal{W}} - \left( \frac{\partial \widehat{\mathcal{W}}}{\partial t_1} \right)^2 }{ \widehat{\mathcal{W}}^2 }.
\]
Substituting into the equation gives
\[
h t_0 \frac{ \frac{\partial \widehat{\mathcal{W}}}{\partial t_0} }{ \widehat{\mathcal{W}} }
- t_1 \frac{ \frac{\partial \widehat{\mathcal{W}}}{\partial t_1} }{ \widehat{\mathcal{W}} }
= r \frac{ \frac{\partial^2 \widehat{\mathcal{W}}}{\partial t_1^2} \,\widehat{\mathcal{W}} - \left( \frac{\partial \widehat{\mathcal{W}}}{\partial t_1} \right)^2 }{ \widehat{\mathcal{W}}^2 }
+ r \frac{ \left( \frac{\partial \widehat{\mathcal{W}}}{\partial t_1} \right)^2 }{ \widehat{\mathcal{W}}^2 }
= r \frac{ \frac{\partial^2 \widehat{\mathcal{W}}}{\partial t_1^2} }{ \widehat{\mathcal{W}} }.
\]
Multiplying both sides by \(\widehat{\mathcal{W}}\) yields
\begin{equation} \label{eq:W_PDE}
h t_0 \frac{\partial \widehat{\mathcal{W}}}{\partial t_0} - t_1 \frac{\partial \widehat{\mathcal{W}}}{\partial t_1} = r \frac{\partial^2 \widehat{\mathcal{W}}}{\partial t_1^2}.  
\end{equation}
Now we derive a recurrence relation for the coefficients \(\widehat{W}_{n,m}\) from the PDE \eqref{eq:W_PDE}. Note that
\begin{equation} \label{eq:W_derivation_recurrence_part1}
h \, t_0 \frac{\partial \widehat{\mathcal{W}}}{\partial t_0} = h \, t_0 \sum_{n \ge 0} \sum_{\mathsf{m} \ge 0} n \widehat{W}_{n,\mathsf{m}} t_0^{n-1} t_1^{\mathsf{m}} = h \sum_{n \ge 0} \sum_{\mathsf{m} \ge 0} n \widehat{W}_{n,\mathsf{m}} t_0^n t_1^{\mathsf{m}},   
\end{equation}
and
\begin{equation} \label{eq:W_derivation_recurrence_part2}
- t_1 \frac{\partial \widehat{\mathcal{W}}}{\partial t_1} = - t_1 \sum_{n \ge 0} \sum_{\mathsf{m} \ge 0} \mathsf{m} \widehat{W}_{n,\mathsf{m}} t_0^n t_1^{\mathsf{m}-1} = - \sum_{n \ge 0} \sum_{\mathsf{m} \ge 0} \mathsf{m} \widehat{W}_{n,\mathsf{m}} t_0^n t_1^{\mathsf{m}}.   
\end{equation}
Similarly,
\[
r \frac{\partial^2 \widehat{\mathcal{W}}}{\partial t_1^2} = r \sum_{n \ge 0} \sum_{\mathsf{m} \ge 0} \mathsf{m}(\mathsf{m}-1) \widehat{W}_{n,\mathsf{m}} \, t_0^n t_1^{\mathsf{m}-2}.
\]  
With the substitution \(\mathsf{m}' = \mathsf{m}-2\) (so \(\mathsf{m} = \mathsf{m}'+2\)), we get
\[
r \frac{\partial^2 \widehat{\mathcal{W}}}{\partial t_1^2} = r \sum_{n \ge 0} \sum_{\mathsf{m}' \ge -2} (\mathsf{m}' + 2)(\mathsf{m}' + 1) \widehat{W}_{n, \mathsf{m}' + 2} t_0^n t_1^{\mathsf{m}'}.
\]
Renaming \( \mathsf{m}' \) as \( \mathsf{m} \) gives us
\begin{equation} \label{eq:W_derivation_recurrence_part3}
r \frac{\partial^2 \widehat{\mathcal{W}}}{\partial t_1^2} = r \sum_{n \ge 0} \sum_{\mathsf{m} \ge -2} (\mathsf{m} + 2)(\mathsf{m} + 1) \widehat{W}_{n, \mathsf{m} + 2} t_0^n t_1^{\mathsf{m}} = r \sum_{n \ge 0} \sum_{\mathsf{m} \ge 0} (\mathsf{m} + 2)(\mathsf{m} + 1) \widehat{W}_{n, \mathsf{m} + 2} t_0^n t_1^{\mathsf{m}}.   
\end{equation}
Combining \eqref{eq:W_derivation_recurrence_part1}, \eqref{eq:W_derivation_recurrence_part2}, \eqref{eq:W_derivation_recurrence_part3} and equating coefficients of \( t_0^n t_1^{\mathsf{m}}\)
\begin{equation} \label{eq:W_recurrence_tmp1}
(hn - \mathsf{m}) \widehat{W}_{n,\mathsf{m}} = r (\mathsf{m} + 2)(\mathsf{m} + 1) \widehat{W}_{n, \mathsf{m} + 2}.   
\end{equation}
We can increase \(\mathsf{m}\) until we obtain \(n\) disconnected polytopes, because in that case we have done zero gluings (\(\mathsf{m} = hn\)), thereby maximizing the number of unglued hyperfaces. Note that \(hn\) and \(\mathsf{m}\) must have the same parity --- because at each step we glue hyperfaces in pairs. Hence \(\widehat{W}_{n,\mathsf{m}} = 0\) when \(hn \not\equiv \mathsf{m} \pmod{2}\). If \(hn \equiv \mathsf{m} \pmod{2}\), then there are \(W_{n, h n} = \frac{(hn)!}{\tau^n} \frac{1}{n!}\) such non‑isomorphic partitions, since we have \((hn)!\) possible numerations, each individual polytope has rotation group order \(\tau\), and these polytopes can be permuted among themselves (which gives division by \(n!\)). Consequently,  \(\widehat{W}_{n, h n} = \left(\frac{1}{\tau}\right)^n \frac{1}{n!}.\)

From \eqref{eq:W_recurrence_tmp1} we also deduce that \(\widehat{W}_{n,\mathsf{m}} = 0\) whenever \(\mathsf{m} > hn\). For \(\mathsf{m} < hn\) we obtain the recurrence
\begin{equation} \label{eq:W_main_recurrence}
\widehat{W}_{n, \mathsf{m}} = \frac{r (\mathsf{m}+2)(\mathsf{m}+1)}{h n - \mathsf{m}} \widehat{W}_{n, \mathsf{m}+2}.    
\end{equation}

Applying \eqref{eq:W_main_recurrence} for \(\mathsf{m}, \mathsf{m}+2, \dots, h n - 2\), we obtain
\begin{equation} \label{eq:W_main_product}
\widehat{W}_{n, \mathsf{m}} = \widehat{W}_{n, h n} \cdot \prod_{j = \mathsf{m}/2}^{h n/2 - 1} \frac{r (2j+2)(2j+1)}{h n - 2j}.   
\end{equation}
Compute the numerator in the right-hand side of \eqref{eq:W_main_product}
\[
\prod_{j = \mathsf{m}/2}^{h n/2 - 1} r (2j+2)(2j+1) = r^{(h n - \mathsf{m})/2} \cdot \prod_{j = \mathsf{m}/2}^{h n/2 - 1} (2j+2)(2j+1).
\]
Notice that for \(j\) ranging from \(\mathsf{m}/2\) to \(h n/2 - 1\), the expressions \(2j+1\) and \(2j+2\) cover all integers from \(\mathsf{m}+1\) to \(h n\). Therefore
\[
\prod_{j = \mathsf{m}/2}^{h n/2 - 1} (2j+2)(2j+1) = \frac{(h n)!}{\mathsf{m}!}.
\]
Let us compute the denominator in the right-hand side of \eqref{eq:W_main_product}
\[
\prod_{j = \mathsf{m}/2}^{h n/2 - 1} (h n - 2j) = \prod_{j = \mathsf{m}/2}^{h n/2 - 1} 2\left( \frac{h n}{2} - j \right) = 2^{(h n - \mathsf{m})/2} \prod_{j = \mathsf{m}/2}^{h n/2 - 1} \left( \frac{h n}{2} - j \right).
\]
Introduce the substitution \(i = \frac{h n}{2} - j\). When \(j = \mathsf{m}/2\), \(i = \frac{h n - \mathsf{m}}{2}\); when \(j = h n/2 - 1\), \(i = 1\). Thus
\[
\prod_{j = \mathsf{m}/2}^{h n/2 - 1} \left( \frac{h n}{2} - j \right) = \prod_{i=1}^{(h n - \mathsf{m})/2} i = \left( \frac{h n - \mathsf{m}}{2} \right)!.
\]
Then the product in the right-hand side of \eqref{eq:W_main_product} takes the form
\[
\prod_{j = \mathsf{m}/2}^{h n/2 - 1} \frac{r (2j+2)(2j+1)}{h n - 2j} = \frac{r^{(h n - \mathsf{m})/2} \cdot \frac{(h n)!}{\mathsf{m}!}}{2^{(h n - \mathsf{m})/2} \cdot \left( \frac{h n - \mathsf{m}}{2} \right)!} = \left( \frac{r}{2} \right)^{(h n - \mathsf{m})/2} \cdot \frac{(h n)!}{\mathsf{m}! \cdot \left( \frac{h n - \mathsf{m}}{2} \right)!}.
\]
Finally, the relation between \(\widehat{W}_{n, \mathsf{m}}\) and \(\widehat{W}_{n, h n}\) is expressed by the formula
\begin{equation} \label{eq:W_main_formula}
\widehat{W}_{n, \mathsf{m}} = \widehat{W}_{n, h n} \cdot \left( \frac{r}{2} \right)^{\frac{h n - \mathsf{m}}{2}} \cdot \frac{(h n)!}{\mathsf{m}! \cdot \left( \frac{h n - \mathsf{m}}{2} \right)!} = \left(\frac{1}{\tau}\right)^n \frac{1}{n!} \cdot \left( \frac{r}{2} \right)^{\frac{h n - \mathsf{m}}{2}} \cdot \frac{(h n)!}{\mathsf{m}! \cdot \left( \frac{h n - \mathsf{m}}{2} \right)!}. 
\end{equation}
For the case of regular polytopes, we can replace \(\tau\) with \(r h\), since these quantities are equal.   
\end{proof}

\begin{rem}
Dividing both sides of recurrence \eqref{eq:W_recurrence_tmp1} by \(2\) and restoring labelings via the substitution \(W_{k,\mathsf{s}} = \mathsf{s}! \widehat{W}_{k,\mathsf{s}}\), we can recover its combinatorial meaning. Recall that \(\frac{hn - \mathsf{m}}{2}\) counts the number of cut locations (seams), whereas the right‑hand side corresponds to the rotation of hyperfaces, which equals \(r\). The factor \(1/2\) arises from the symmetry in choosing the last and the second‑last index.
\end{rem}

\begin{quest}
Interestingly, we have derived an explicit formula for \(\widehat{W}_{n, \mathsf{m}}\) of disconnected objects. We conjecture that \(\widehat{U}_{n, \mathsf{b}_1} = \widehat{V}_{n, \mathsf{m}}\) admits no closed-form expression, though this remains unproven. The asymptotic behavior is also unknown.
\end{quest}

\begin{rem}
Given the explicit formula for \(\widehat{W}_{n,\mathsf{m}}\), one can devise a simple combinatorial explanation. Consider \(n\) copies of the polytope; their \(hn\) facets are then distinguishable. Suppose that \(0 \leq \mathsf{m} \leq hn\) and \(hn - \mathsf{m} = 2g\). We first choose the \(\mathsf{m}\) free facets in \(\binom{hn}{\mathsf{m}}\) ways. The remaining \(2g\) facets can be partitioned into unordered pairs in $(2g-1)!!=\frac{(2g)!}{2^g g!}$ ways, and each pair admits \(r\) gluing maps. Thus the number of gluing configurations before taking symmetries into account is $\binom{hn}{\mathsf{m}}(2g-1)!!\,r^g.$ Since disconnected polytopes have an automorphism group of order \(\tau^n n!\), it follows that
\[
\widehat W_{n,\mathsf{m}} = \frac{1}{\tau^n n!}\binom{hn}{\mathsf{m}}(hn-\mathsf{m}-1)!! \, r^{(hn-\mathsf{m})/2}
= \frac{1}{\tau^n n!}\left(\frac{r}{2}\right)^g
\frac{(hn)!}{\mathsf{m}! \, g!}.
\]
An example of such a pairing is shown in Figure~\ref{fig:facet_pairings}
for two tetrahedra with two free facets.
\begin{figure}[!htbp]
    \centering
    \includegraphics[width=0.82\textwidth]{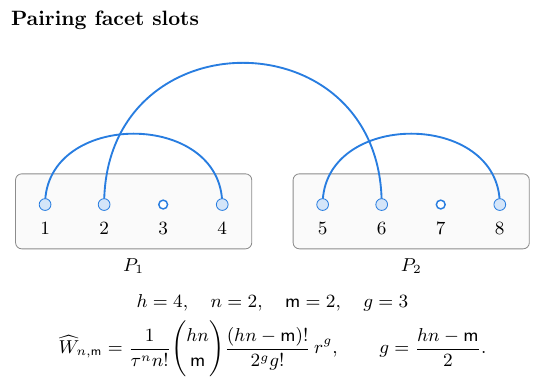}
    \caption{Pairing facet slots, with unpaired slots representing free facets.}
    \label{fig:facet_pairings}
\end{figure}
\end{rem}

It is noteworthy that for fixed \(n\), the maximum of \(\widehat{W}_{n, \mathsf{m}}\) does not occur at \(\mathsf{m}=0\) (the closed manifold). Further details are provided in appendix \ref{sec:maximum_W}.

A natural quantity to consider is \(\widehat{W}_n = \sum_{\mathsf{m}=0}^{h n} \widehat{W}_{n, \mathsf{m}} = [t_0^n] \, \widehat{\mathcal{W}}(t_0, 1)\). We reduce this quantity to Hermite polynomials and obtain its asymptotics using results from \cite{AM83}.
\begin{stat}
The asymptotic behavior of \(\widehat{W}_n\) is as follows
\[
\widehat{W}_n \sim \frac{1}{2 \sqrt{\pi n}} e^{-1/(4 r)} \left(\frac{(r h)^{h/2} e^{1 - h/2}}{\tau}\right)^n n^{n(h/2 - 1)} e^{\sqrt{\frac{h n}{r}}}.
\]  
\end{stat}
\begin{proof}
Since \(h n - \mathsf{m}\) must be even, let
\[
h n - \mathsf{m} = 2 \ell \quad \implies \quad \mathsf{m} = h n - 2 \ell,
\]
where \(\ell = 0, 1, 2, \dots, \left\lfloor \frac{h n}{2} \right\rfloor\).
Then
\[
\frac{h n - \mathsf{m}}{2} = \ell, \quad \mathsf{m}! = (h n - 2 \ell)!, \quad \left(\frac{h n - \mathsf{m}}{2}\right)! = \ell!.
\]
Substituting into \(\widehat{W}_{n, \mathsf{m}}\)
\[
\widehat{W}_{n, h n - 2 \ell} = \left(\frac{1}{\tau}\right)^n \frac{1}{n!} \left(\frac{r}{2}\right)^\ell \frac{(h n)!}{(h n - 2 \ell)! \, \ell!}.
\]
The sum is
\[
\widehat{W}_n = \sum_{\ell=0}^{\lfloor h n/2 \rfloor} \widehat{W}_{n, h n - 2 \ell}
= \left(\frac{1}{\tau}\right)^n \frac{1}{n!} \sum_{\ell=0}^{\lfloor h n/2 \rfloor} \left(\frac{r}{2}\right)^\ell \frac{(h n)!}{(h n - 2 \ell)! \, \ell!}.
\]
The Hermite polynomial \(H_s(x)\) has the representation
\[
H_s(x) = s! \sum_{\ell=0}^{\lfloor s/2 \rfloor} \frac{(-1)^\ell}{\ell! \, (s-2\ell)!} (2x)^{s-2\ell}.
\]
Therefore
\[
H_{h n}\left(\frac{i}{\sqrt{2 r}}\right) = (h n)! \left(i \sqrt{\frac{2}{r}}\right)^{h n} \sum_{\ell=0}^{\lfloor h n/2 \rfloor} \frac{1}{\ell! \, (h n - 2 \ell)!} \left(\frac{r}{2}\right)^\ell.
\]
Thus, we can rewrite the expression for \(\widehat{W}_n\)
\begin{equation} \label{eq:W_n_Hermitian}
\widehat{W}_n = \left(\frac{1}{\tau}\right)^n \frac{1}{n!} \cdot \frac{H_{h n}\left(\frac{i}{\sqrt{2 r}}\right)}{\left(i \sqrt{\frac{2}{r}}\right)^{h n}} = \left(\frac{1}{\tau}\right)^n \frac{1}{n!} \left(\frac{r}{2}\right)^{\frac{h n}{2}} \frac{H_{h n}\left(\frac{i}{\sqrt{2 r}}\right)}{i^{h n}}.  
\end{equation}
The asymptotics of Hermite polynomials (see \cite[p. 508-510, 13.6.38 and 13.5.16]{AM83}) is
\begin{equation} \label{eq:Hermitian_asymptotic}
H_s(x) \sim e^{x^2/2} \left(\frac{2s}{e}\right)^{s/2} \sqrt{2} \cos\left(x\sqrt{2s} - \frac{s\pi}{2}\right).
\end{equation}
We combine \eqref{eq:W_n_Hermitian} and \eqref{eq:Hermitian_asymptotic}
\begin{equation} \label{eq:hatW_asymptotic}
\widehat{W}_n \sim \left(\frac{1}{\tau}\right)^n \frac{1}{n!} e^{-1/(4r)} \sqrt{2} \left(\frac{rhn}{e}\right)^{hn/2} i^{-hn} \cos\left(i\sqrt{\frac{hn}{r}} - \frac{hn\pi}{2}\right)
\end{equation}
Let us separately analyze the expression
\[
\frac{1}{i^{h n}} \cos\left(i \sqrt{\frac{h n}{r}} - \frac{h n \pi}{2}\right).
\]
Since
\[
\cos(u - v) = \frac{e^{i(u - v)} + e^{-i(u - v)}}{2},
\]
we have
\[
\cos\left(i \sqrt{\frac{h n}{r}} - \frac{h n \pi}{2}\right) = \frac{e^{i\left(i \sqrt{\frac{h n}{r}} - \frac{h n \pi}{2}\right)} + e^{-i\left(i \sqrt{\frac{h n}{r}} - \frac{h n \pi}{2}\right)}}{2}.
\]
Using \(i^{-h n} = e^{-i \frac{h n \pi}{2}}\), we get
\[
e^{-i \frac{h n \pi}{2}} \cdot \cos\left(i \sqrt{\frac{h n}{r}} - \frac{h n \pi}{2}\right) = \frac{e^{-\sqrt{\frac{h n}{r}}} e^{-i h n \pi} + e^{\sqrt{\frac{h n}{r}}}}{2}.
\]
Since \(e^{-i h n \pi} = (-1)^{h n}\), we obtain
\[
\frac{1}{i^{h n}} \cos\left(i \sqrt{\frac{h n}{r}} - \frac{h n \pi}{2}\right) = \frac{1}{2}\left(e^{\sqrt{\frac{h n}{r}}} + (-1)^{h n} e^{-\sqrt{\frac{h n}{r}}}\right) \sim \frac{1}{2} e^{\sqrt{\frac{h n}{r}}}.
\]
Now, substituting the asymptotic formula into \eqref{eq:hatW_asymptotic} and also applying Stirling's formula to \(n!\), we obtain 
\[
\widehat{W}_n \sim \frac{1}{2 \sqrt{\pi n}} e^{-1/(4 r)} \left(\frac{(r h)^{h/2} e^{1 - h/2}}{\tau}\right)^n n^{n(h/2 - 1)} e^{\sqrt{\frac{h n}{r}}}.
\] 
Under the condition \( rh = \tau \), the expression can be simplified even more
\[
\widehat{W}_n \sim \frac{1}{2\sqrt{\pi n}} \,
\exp\!\left( h\sqrt{\frac{n}{\tau}} - \frac{h}{4\tau} \right) \,\left( \frac{n\tau}{e} \right)^{n(h/2-1)}.
\]

\end{proof}

\begin{quest}
The quantity \(\widehat{U}_n = \sum_{\mathsf{b}_1=0}^{\left\lfloor \frac{(h-2)n + 2}{2} \right\rfloor} \widehat{U}_{n, \mathsf{b}_1} = [t_0^n] \,\widehat{\mathcal{V}}(t_0, 1)\) presents another interesting case; for this sum, the asymptotics are also unknown, and a closed-form expression seems unlikely.
\end{quest}

\begin{quest}
Another approach to working with \(\widehat{U}_{n, \mathsf{b}_1}\), \(\widehat{V}_{n, \mathsf{m}}\), and \(\widehat{W}_{n, \mathsf{m}}\) is to proceed iteratively: starting from the base case \(\mathsf{b}_1 = 0\), each subsequent step \(\mathsf{b}_1 = k\) is expressed in terms of the solutions obtained for previous values \(\mathsf{b}_1 < k\). However, starting from \(\mathsf{b}_1 = 2\), the problem becomes significantly more difficult, and no clear pattern or closed-form expression emerges. Overall, the structure resembles a hierarchy of differential equations. Can such a hierarchy be rigorously constructed? Furthermore, might our results be relevant to quantum gravity?
\end{quest}

We can also attach complexes (but with the possibility of cycles in the dual graph) to a fixed initial complex. We call the resulting objects \textit{recursive} complexes and derive recurrence relations similar to \eqref{eq:T_surfaces_recurrence} and \eqref{eq:T_general_recurrence}.

\begin{stat} \label{stat:Xi_general_recurrence} 
Let \({}_{\mathsf{m}_0}\widehat{\Psi}_{n,\mathsf{m}}\) denote the number of labelings of connected recursive complexes obtained by attaching \(n\) polytopes to a fixed initial complex, where the initial complex has \(\mathsf{m}_0\) free hyperfaces and the resulting complex has \(\mathsf{m}\) free hyperfaces. Let \(h>2\) be the number of codimension-one hyperfaces of the polytope, and let \(r\) be the order of the automorphism (rotation) group of these hyperfaces. Then the following recurrence relation holds
\begin{equation} \label{eq:Psi_general_recurrence}
\frac{hn+\mathsf{m}_0 - \mathsf{m}}{2} \cdot {}_{\mathsf{m}_0} \Psi_{n, \mathsf{m}} = \frac{r}{2} \cdot {}_{\mathsf{m}_0} \Psi_{n, \mathsf{m}+2} + r \sum\limits_{l=0}^{n-1} \sum\limits_{\mathsf{s}=0}^{\mathsf{m}} {}_{\mathsf{m}_0} \Psi_{l, \mathsf{s}+1} V_{n-l, \mathsf{m}-\mathsf{s}+1} \binom{\mathsf{m}}{\mathsf{s}}.    
\end{equation}
\end{stat}

\begin{proof}
The proof technique combines the ideas of Statements \ref{stat:T_surfaces_recurrence}, \ref{stat:T_general_recurrence}, and \ref{sec:stat_U_recurrence}; many details are omitted for brevity. 

As in the case of Statement \ref{sec:stat_U_recurrence}, we encounter two scenarios: one that splits the complex into two connected components, and another that leaves the number of components unchanged. When the cut results in a single connected component, the contribution is given by the term $\frac{r}{2} \cdot {}_{\mathsf{m}_0} \Psi_{n, \mathsf{m}+2}$ on the right-hand side of \eqref{eq:Psi_general_recurrence}. Multiplication by \(r\) (the order of the rotation group of the hyperface) accounts for the ability to rotate the hyperface, while division by \(2\) accounts for the symmetry between the last and second-to-last hyperfaces chosen for gluing.

The second scenario --- responsible for the second term on the right-hand side of the recurrence \eqref{eq:Psi_general_recurrence} --- corresponds to the case where the complex splits into two connected components. The first component contains the initial fixed complex, i.e., it is a recursive complex with \(\mathsf{s}+1\) free facets obtained by attaching a complex of \(l\) polytopes; its contribution is \({}_{\mathsf{m}_0} \Psi_{l, \mathsf{s}+1}\). The second component is a complex of \((n-l)\) polytopes having \(\mathsf{m}-\mathsf{s}+1\) free facets (the \(+1\) appears because one face is freed after cutting), contributing \(V_{n-l, \mathsf{m}-\mathsf{s}+1}\). Thus, the total result is the sum over all possible \(l\) and \(\mathsf{s}\) of the product of these two terms multiplied by the binomial coefficient that counts the number of ways to distribute the labelings of the external faces.

In contrast to the recurrence \(\eqref{eq:normalized_V_cycles_recurrence}\) from statement \ref{sec:stat_U_recurrence}, the left‑hand side involves \(\frac{hn+\mathsf{m}_0 - \mathsf{m}}{2}\), since the number of free hyperfaces has increased by \(\mathsf{m}_0\) in the initial complex.

Equating the left‑hand side and the right‑hand side of the recurrence yields the desired result.
\end{proof}

By introducing normalized variables $\widehat{V}_{i, \mathsf{j}}$ and $ \widehat{\Psi}_{i, \mathsf{j}} := \dfrac{\Psi_{i, \mathsf{j}}}{\mathsf{j}!}$ into \eqref{eq:Psi_general_recurrence}, we obtain the equivalent equation:
\begin{equation} \label{eq:hat_Psi_general_recurrence}
\begin{split}
\frac{hn+\mathsf{m}_0 - \mathsf{m}}{2} \cdot {}_{\mathsf{m}_0} \widehat{\Psi}_{n, \mathsf{m}} 
&= \frac{r}{2} (\mathsf{m}+1) (\mathsf{m}+2) \cdot {}_{\mathsf{m}_0} \widehat{\Psi}_{n, \mathsf{m}+2} \\
&\quad + r \sum_{l=0}^{n-1} \sum_{\mathsf{s}=0}^{\mathsf{m}} {}_{\mathsf{m}_0} \widehat{\Psi}_{l, \mathsf{s}+1} \widehat{V}_{n-l, \mathsf{m}-\mathsf{s}+1} (\mathsf{s}+1) (\mathsf{m}-\mathsf{s}+1).
\end{split}
\end{equation}

\begin{rem}
Note that \({}_{\mathsf{m}_0}\widehat{\Psi}_{n,\mathsf{m}} = 0\) whenever \(\mathsf{m} < 0\). 
\end{rem}

\begin{stat} \label{stat:Xi_explicit_solution}
Define the generating function ${}_{\mathsf{m}_0}\widehat{\varPsi}(t_0,t_1) = \sum_{n\ge 0,\,\mathsf{m}\ge 0} {}_{\mathsf{m}_0}\widehat{\Psi}_{n,\mathsf{m}}\, t_0^n t_1^{\mathsf{m}}$. Then the coefficients of the transformed generating function ${}_{\mathsf{m}_0} \widehat{\Xi} = \widehat{\mathcal{W}} \cdot {}_{\mathsf{m}_0}\widehat{\varPsi}$, expanded as ${}_{\mathsf{m}_0} \widehat{\Xi}(t_0,t_1) = \sum_{n\ge 0,\;\mathsf{m}\ge 0} {}_{\mathsf{m}_0} \widehat{\xi}_{n,\mathsf{m}}\, t_0^n t_1^{\mathsf{m}}$, satisfy the recurrence relation:
\begin{equation} \label{eq:xi_recurrence_tmp1}
{}_{\mathsf{m}_0}\widehat{\xi}_{n,\mathsf{m}} (hn - \mathsf{m} + \mathsf{m}_0) = {}_{\mathsf{m}_0}\widehat{\xi}_{n,\mathsf{m}+2} \cdot r\,(\mathsf{m}+2)(\mathsf{m}+1)\,, \qquad n\ge 1,\; \mathsf{m}\ge 0.
\end{equation}
The explicit formula for $\widehat{\xi}_{n,\mathsf{m}}$ is given by:
\begin{equation} \label{eq:xi_explicit_formula}
{}_{\mathsf{m}_0}\widehat{\xi}_{n,\mathsf{m}}  = \dfrac{{}_{\mathsf{m}_0}\widehat{\Psi}_{0,\mathsf{m}_0}}{n! (\tau)^n}\;
\left(\frac{r}{2}\right)^{\frac{hn+\mathsf{m}_0-\mathsf{m}}{2}}\;
\frac{(hn+\mathsf{m}_0)!}{\mathsf{m}!\,\left(\frac{hn+\mathsf{m}_0-\mathsf{m}}{2}\right)!}.
\end{equation}
\end{stat}

\begin{proof}

By an argument analogous to the proof of Statement \ref{stat:V_differential_equation}, the recurrence \eqref{eq:hat_Psi_general_recurrence} can be rewritten as a partial differential equation for the generating function:
\begin{equation} \label{eq:widehat_Psi_differential}
h t_0\frac{\partial}{\partial t_0}\left({}_{\mathsf{m}_0}\widehat{\varPsi}\right) - t_1\frac{\partial}{\partial t_1}\left({}_{\mathsf{m}_0}\widehat{\varPsi}\right) + \mathsf{m}_0 \left({}_{\mathsf{m}_0}\widehat{\varPsi} \right)
= r\frac{\partial^2}{\partial t_1^2}\left({}_{\mathsf{m}_0}\widehat{\varPsi}\right) + 2r\frac{\partial}{\partial t_1}\left({}_{\mathsf{m}_0}\widehat{\varPsi}\right)\frac{\partial\widehat{\mathcal{V}}}{\partial t_1}.
\end{equation}
Computing the partial derivatives of $\widehat{\Xi} = \widehat{\mathcal{W}} \cdot {}_{\mathsf{m}_0}\widehat{\varPsi}$ yields:
\[
\frac{\partial\widehat{\Xi}}{\partial t_0} = \widehat{\mathcal{W}}\frac{\partial}{\partial t_0}\left({}_{\mathsf{m}_0}\widehat{\varPsi}\right) + {}_{\mathsf{m}_0}\widehat{\varPsi}\frac{\partial\widehat{\mathcal{W}}}{\partial t_0},\quad
\frac{\partial\widehat{\Xi}}{\partial t_1} = \widehat{\mathcal{W}}\frac{\partial}{\partial t_1}\left({}_{\mathsf{m}_0}\widehat{\varPsi}\right) + {}_{\mathsf{m}_0}\widehat{\varPsi}\frac{\partial\widehat{\mathcal{W}}}{\partial t_1},
\]
\[
\frac{\partial^2\widehat{\Xi}}{\partial t_1^2} = \widehat{\mathcal{W}}\frac{\partial^2}{\partial t_1^2}\left({}_{\mathsf{m}_0}\widehat{\varPsi}\right)
+ 2\frac{\partial\widehat{\mathcal{W}}}{\partial t_1}\frac{\partial}{\partial t_1}\left({}_{\mathsf{m}_0}\widehat{\varPsi}\right)
+ {}_{\mathsf{m}_0}\widehat{\varPsi}\frac{\partial^2\widehat{\mathcal{W}}}{\partial t_1^2}.
\]
Multiplying \eqref{eq:widehat_Psi_differential} by $\widehat{\mathcal{W}} \neq 0$, the left-hand side of the equation takes the form
\begin{equation} \label{eq:left_hand_side_psi}
h t_0\left(\frac{\partial}{\partial t_0}\left({}_{\mathsf{m}_0}\widehat{\Xi}\right) - \left({}_{\mathsf{m}_0}\widehat{\varPsi}\right)\frac{\partial\widehat{\mathcal{W}}}{\partial t_0}\right) - t_1\left(\frac{\partial}{\partial t_1}\left({}_{\mathsf{m}_0}\widehat{\Xi}\right) - \left({}_{\mathsf{m}_0}\widehat{\varPsi}\right)\frac{\partial\widehat{\mathcal{W}}}{\partial t_1}\right) + \mathsf{m}_0\left({}_{\mathsf{m}_0}\widehat{\Xi}\right).  
\end{equation}
Since $\frac{\partial\widehat{\mathcal{W}}}{\partial t_1} = \frac{\partial \exp(\widehat{\mathcal{V}})}{\partial t_1} = \widehat{\mathcal{W}} \frac{\partial\widehat{\mathcal{V}}}{\partial t_1}$, the right‑hand side transforms into:
\begin{equation} \label{eq:right_hand_side_psi}
r\left(\widehat{\mathcal{W}}\frac{\partial^2}{\partial t_1^2}\left({}_{\mathsf{m}_0}\widehat{\varPsi}\right) + 2\frac{\partial\widehat{\mathcal{W}}}{\partial t_1}\frac{\partial}{\partial t_1}\left({}_{\mathsf{m}_0}\widehat{\varPsi}\right)\right) = r\left(\frac{\partial^2}{\partial t_1^2}\left({}_{\mathsf{m}_0}\widehat{\Xi}\right) - \left({}_{\mathsf{m}_0}\widehat{\varPsi}\right)\frac{\partial^2\widehat{\mathcal{W}}}{\partial t_1^2}\right). 
\end{equation}
Equating \eqref{eq:left_hand_side_psi} and \eqref{eq:right_hand_side_psi}, we obtain
\begin{equation*}
\begin{aligned}
& h t_0\frac{\partial}{\partial t_0}\left({}_{\mathsf{m}_0}\widehat{\Xi}\right) 
- t_1\frac{\partial}{\partial t_1}\left({}_{\mathsf{m}_0}\widehat{\Xi}\right) 
+ \mathsf{m}_0\left({}_{\mathsf{m}_0}\widehat{\Xi}\right) \\
& \quad - h t_0 \left({}_{\mathsf{m}_0}\widehat{\varPsi}\right)\frac{\partial\widehat{\mathcal{W}}}{\partial t_0}
+ t_1 \left({}_{\mathsf{m}_0}\widehat{\varPsi}\right)\frac{\partial\widehat{\mathcal{W}}}{\partial t_1} \\
& = r\frac{\partial^2}{\partial t_1^2}\left({}_{\mathsf{m}_0}\widehat{\Xi}\right) 
- r\left({}_{\mathsf{m}_0}\widehat{\varPsi}\right)\frac{\partial^2\widehat{\mathcal{W}}}{\partial t_1^2}.
\end{aligned}
\end{equation*}

\begin{equation} \label{eq:Xi_intermediate_PDE}
\begin{aligned}
{}_{\mathsf{m}_0}\widehat{\varPsi}&\left( -h t_0 \frac{\partial\widehat{\mathcal{W}}}{\partial t_0} + t_1 \frac{\partial\widehat{\mathcal{W}}}{\partial t_1} + r \frac{\partial^2\widehat{\mathcal{W}}}{\partial t_1^2} \right) \\
&= r\frac{\partial^2}{\partial t_1^2}\left({}_{\mathsf{m}_0}\widehat{\Xi}\right) - h t_0\frac{\partial}{\partial t_0}\left({}_{\mathsf{m}_0}\widehat{\Xi}\right) + t_1\frac{\partial}{\partial t_1}\left({}_{\mathsf{m}_0}\widehat{\Xi}\right) - \mathsf{m}_0\left({}_{\mathsf{m}_0}\widehat{\Xi}\right).
\end{aligned}
\end{equation}

Note that from equation \eqref{eq:W_PDE}, the left-hand side of the equation \eqref{eq:Xi_intermediate_PDE} equals zero. This yields the simplified equation
\begin{equation} \label{eq:Xi_PDE}
h t_0\frac{\partial}{\partial t_0}\left({}_{\mathsf{m}_0}\widehat{\Xi}\right) - t_1\frac{\partial}{\partial t_1}\left({}_{\mathsf{m}_0}\widehat{\Xi}\right) + \mathsf{m}_0\left({}_{\mathsf{m}_0}\widehat{\Xi}\right) = r\frac{\partial^2}{\partial t_1^2}\left({}_{\mathsf{m}_0}\widehat{\Xi}\right).
\end{equation}
Let \({}_{\mathsf{m}_0} \widehat{\Xi}(t_0,t_1) = \sum_{n\ge1,\;m\ge0} {}_{\mathsf{m}_0} \widehat{\xi}_{n,m}\, t_0^n t_1^m\). 
Then, following a reasoning similar to the proof of Statement \ref{stat:exp_V_series}, the partial differential equation for the generating function yields the recurrence relation
\begin{equation}\label{eq:xi_explicit_formula_tmp1}
{}_{\mathsf{m}_0}\widehat{\xi}_{n,\mathsf{m}} (hn - \mathsf{m} + \mathsf{m}_0) = {}_{\mathsf{m}_0}\widehat{\xi}_{n,\mathsf{m}+2} \cdot r\,(\mathsf{m}+2)(\mathsf{m}+1)\,, \qquad n\ge 1,\; \mathsf{m}\ge 0.
\end{equation}
For each fixed \(n\), the largest non-zero index \(\mathsf{m}\) in the sequence \({}_{\mathsf{m}_0}\widehat{\xi}_{n,\mathsf{m}}\) is \(M = hn + \mathsf{m}_0\). Moreover, if \(M\not\equiv \mathsf{m} \pmod 2\), then \({}_{\mathsf{m}_0}\widehat{\xi}_{n,\mathsf{m}} = 0\). Both assertions follow from the same argument used to determine the maximal non-zero index for \(\widehat{W}_{n,\mathsf{m}}\).

By direct analogy with recurrence \(\eqref{eq:W_recurrence_tmp1}\) and its closed‑form solution \(\eqref{eq:W_main_formula}\), the explicit formula for \(\widehat{\xi}_{n,\mathsf{m}}\) following from \(\eqref{eq:xi_recurrence_tmp1}\) is obtained via the substitution \(hn \to hn + \mathsf{m}_0\) in \(\eqref{eq:W_main_formula}\)
\[
{}_{\mathsf{m}_0}\widehat{\xi}_{n,\mathsf{m}}  = {}_{\mathsf{m}_0}\widehat{\xi}_{n,hn+\mathsf{m}_0}\;
\left(\frac{r}{2}\right)^{\frac{hn+\mathsf{m}_0-\mathsf{m}}{2}}\;
\frac{(hn+\mathsf{m}_0)!}{\mathsf{m}!\,\left(\frac{hn+\mathsf{m}_0-\mathsf{m}}{2}\right)!}.
\]

\begin{figure}[!htbp]
    \centering
    \includegraphics[width=0.90\textwidth]{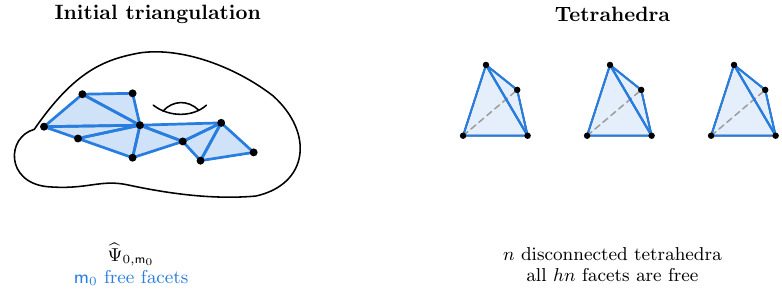}
    \caption{The initial triangulation and \(n\) disconnected
        tetrahedra before any gluings.}
    \label{fig:xi_initial_state}
\end{figure}

It remains to determine \({}_{\mathsf{m}_0}\widehat{\xi}_{n,hn+\mathsf{m}_0}\) from geometric considerations (see Figure~\ref{fig:xi_initial_state} for a visual explanation in the case of tetrahedra). Observe that \({}_{\mathsf{m}_0}\widehat{\Xi}(t_0,t_1)\) is the convolution of \(\widehat{\mathcal{W}}(t_0,t_1)\) and \({}_{\mathsf{m}_0}\widehat{\varPsi}(t_0,t_1)\); geometrically, this means that it counts pairs consisting of a recursive complex and an abstract complex (the latter possibly disconnected, i.e., their count is an element of the expansion \(\widehat{\mathcal{W}}(t_0,t_1)\)). Note that the automorphisms of the recursive complex are the product of the automorphisms of the initial complex (denoted by \( \widehat{\Psi}_{0,\mathsf{m}_0} \)) and those of the complexes built from polytopes (possibly disconnected). The maximal number of free hyperfaces is \(M = hn + \mathsf{m}_0\). This maximum is achieved when the polytopes are disjoint: \(\mathsf{m}_0\) is fixed for the initial complex, while \(hn\) is the maximal number of hyperfaces contributed by \(n\) disconnected polytopes, and gluing can only decrease this number. Since there are \(n!\) ways to permute the polytopes and \(\tau^n\) ways to rotate each polytope independently, we obtain
\[
{}_{\mathsf{m}_0}\widehat{\xi}_{n,\mathsf{m}}  = \dfrac{{}_{\mathsf{m}_0}\widehat{\Psi}_{0,\mathsf{m}_0}}{n! (\tau)^n}\;
\left(\frac{r}{2}\right)^{\frac{hn+\mathsf{m}_0-\mathsf{m}}{2}}\;
\frac{(hn+\mathsf{m}_0)!}{\mathsf{m}!\,\left(\frac{hn+\mathsf{m}_0-\mathsf{m}}{2}\right)!}.
\]
\end{proof}

Knowing \(\widehat{\xi}_{n,\mathsf{m}}\), we can recover the coefficients of \({}_{\mathsf{m}_0}\widehat{\varPsi}(t_0,t_1) = {}_{\mathsf{m}_0} \widehat{\Xi}(t_0,t_1) / \widehat{\mathcal{W}}(t_0,t_1)\) via formal power series division.

\begin{rem}
In analogy with Remark \ref{rem:R_T_connection_surfaces}, we seek to relate \(\widehat{W}_{n, \mathsf{m}}\) and \({}_{\mathsf{m}_0}\widehat{\xi}_{n,\mathsf{m}}\). Starting with one regular polytope, which has \(h\) hyperfaces of codimension \(1\) (i.e., \(\mathsf{m}_0 = h, \, {}_{\mathsf{m}_0}\widehat{\Psi}_{0,\mathsf{m}_0} = \frac{1}{\tau}\)), we need to glue \(n-1\) times, therefore we should calculate the ratio
\[
\begin{aligned}
\frac{{}_{h_0}\widehat{\xi}_{n-1,\mathsf{m}}}{\widehat{W}_{n,\mathsf{m}}}
&= 
\frac{
\frac{{}_{\mathsf{m}_0}\widehat{\Psi}_{0,\mathsf{m}_0}}{(n-1)!\,\tau^{n-1}}
\left(\frac{r}{2}\right)^{\frac{h(n-1)+\mathsf{m}_0-\mathsf{m}}{2}}
\frac{(h(n-1)+\mathsf{m}_0)!}{\mathsf{m}!\left(\frac{h(n-1)+\mathsf{m}_0-\mathsf{m}}{2}\right)!}
}{
\frac{1}{n!\,\tau^{n}}
\left(\frac{r}{2}\right)^{\frac{hn-\mathsf{m}}{2}}
\frac{(hn)!}{\mathsf{m}!\left(\frac{hn-\mathsf{m}}{2}\right)!}
} \\[1.2em]
&= {}_{\mathsf{m}_0}\widehat{\Psi}_{0,\mathsf{m}_0}
\cdot n \cdot \tau
\cdot
\left(\frac{r}{2}\right)^{
\frac{h(n-1)+\mathsf{m}_0-\mathsf{m} - (hn-\mathsf{m})}{2}
}
\cdot
\frac{(h(n-1)+\mathsf{m}_0)!}{(hn)!}
\cdot
\frac{\left(\frac{hn-\mathsf{m}}{2}\right)!}
{\left(\frac{h(n-1)+\mathsf{m}_0-\mathsf{m}}{2}\right)!} \\[1.2em]
&= {}_{\mathsf{m}_0}\widehat{\Psi}_{0,\mathsf{m}_0}
\cdot n
\cdot \tau
\cdot
\left(\frac{r}{2}\right)^{\frac{\mathsf{m}_0-h}{2}}
\cdot
\frac{(h(n-1)+\mathsf{m}_0)!}{(hn)!}
\cdot
\frac{\left(\frac{hn-\mathsf{m}}{2}\right)!}
{\left(\frac{h(n-1)+\mathsf{m}_0-\mathsf{m}}{2}\right)!} \Bigg|_{\substack{\mathsf{m}_0 = h, \\ {}_{\mathsf{m}_0}\widehat{\Psi}_{0,\mathsf{m}_0} = \frac{1}{\tau}}} \\
&= n .
\end{aligned}
\]
\end{rem}

\begin{quest}
From recurrence \eqref{eq:hat_Psi_general_recurrence}, it is clear that the growth rate of \( {}_{\mathsf{m}_0} \widehat{\Psi}_{n, \mathsf{m}} \) is superexponential, as its sum includes the terms \( \widehat{V}_{n-1, \mathsf{m}} \), which itself grows faster than exponentially. However, it is unclear whether a closed-form expression exists for \( {}_{\mathsf{m}_0} \widehat{\Psi}_{n, \mathsf{m}} \), how to determine its precise asymptotic behavior, and how to express \( {}_{\mathsf{m}_0} \widehat{\Psi}_{n, \mathsf{m}}\) explicitly in terms of \( \widehat{V}_{n, \mathsf{m}}\). Similar questions naturally arise for the quantity
\({}_{\mathsf{m}_0} \widehat{\Psi}_{n} = [t_0^n] \, {}_{\mathsf{m}_0} \widehat{\varPsi}(t_0,1)\).
\end{quest}

\begin{quest}
It would be interesting to derive recurrences that count all complexes of a fixed topology, analogous to those in \cite{GS06,GJ08,BGR08}, but this is a very challenging problem. Is it possible to bypass this difficulty by tracking topological invariants instead of the exact topology? Possible directions include: seeking asymptotic bounds instead of exact formulas, restricting attention to a subvector of Betti numbers, focusing on simple topologies such as balls or spheres, or using generating functions to track face counts \(f_i\) with a fixed Euler characteristic \(\chi\).
\end{quest}

Another natural direction is to study the specific structure of triangulations. 

\begin{quest}
Triangulated spheres and balls in dimensions \(d > 2\) remain poorly understood \cite{BZ11}, with some exceptions in three dimensions \cite{CEY12}. Can we improve the bounds of \cite{ALSY25,CEY12,R13} by incorporating our method?
\end{quest}

It is also natural to generalize this to gluings of faces of codimension \(i\) with \(i \ge 2\). In this generalized setting, two cases arise: either several gluings along the same hyperface are allowed (denoted by \({}^{\operatorname{s}}_{i}R_n\)), or they are not allowed (denoted by \({}^{\operatorname{u}}_{i}R_n\)). Under this notation, our original \(R_n\) is written as \({}^{\operatorname{u}}_{1}R_n\).

All recurrences and formulas derived in this work for codimension \(1\) (\({}^{\operatorname{u}}_{1}\widehat{R}_n\), \({}^{\operatorname{u}}_{1, h_0}\widehat{T}_n, {}^{\operatorname{u}}_{1}\widehat{V}_{n, \mathsf{m}}\), and \({}^{\operatorname{u}}_{1}\widehat{W}_{n, \mathsf{m}}\)) naturally extend to the unique-gluing case for any codimension \(i\). The proofs are completely analogous; we only need to update the parameters \(r, h, \tau\). For example, for a \(d\)-dimensional simplex and \(i=2\), the number of codimension-2 hyperfaces is \({}_{2}h = \frac{d(d+1)}{2}\), the rotation group order is \({}_{2}r = \frac{(d-1)!}{2}\), and \(\tau = \frac{(d+1)!}{2}\). We can also define \({}^{\operatorname{all}}_{u}R_n\) as the object that allows unique gluings of any codimension.

\section{Case \texorpdfstring{\(h=2\)}{h=2}} \label{sec:segment}
So far, we have discussed gluings along faces for \(h>2\).  In this section we address the remaining case \(h=2\), which was not covered. For this value, equations \eqref{eq:R_general_recurrence}, \eqref{eq:T_general_recurrence}, and \eqref{eq:U_cycles_recurrence} reduce to a much simpler form. Since the condition \(h\ne 2\) was not used for the remaining quantities, the case \(h=2\) reduces to a direct substitution into those expressions.

The case \(h=2\) represents a one-dimensional simplex (a segment). Its codimension-one facets are points, and the automorphism group of a point is trivial, i.e. \(r=1\). Thus
\begin{equation} \label{eq:app:simple_R_recurrence}
(n-1)R_n = \frac{r}{2} \sum\limits_{k=1}^{n-1} \binom{2}{1} R_k R_{n-k} = \sum\limits_{k=1}^{n-1} R_k R_{n-k}.    
\end{equation}
It is straightforward to verify that the solution is given by \(R_n = R_1^n\) with \(R_1 = \dfrac{2!}{\tau}\), where \(\tau = 2\) or \(1\) depending on whether orientation-reversing automorphisms are allowed. The proof can be carried out by induction or by using generating functions; we leave the details to the interested reader. This construction is illustrated in Figure~\ref{fig:h2_segment_gluings} for unoriented and directed segments, respectively.

\begin{figure}[!htbp]
    \centering
    \includegraphics[width=0.85\textwidth]
        {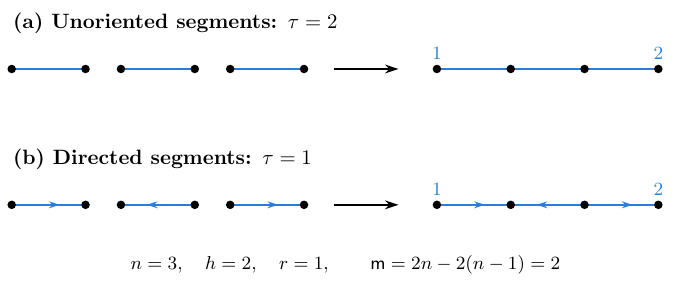}
    \caption{Gluing three initially disconnected segments into a path
in the unoriented and directed cases.}
    \label{fig:h2_segment_gluings}
\end{figure}

Using \(R_n = R_1^n\), we can simplify \eqref{eq:T_general_recurrence}
\begin{equation} \label{eq:app:simple_T_recurrence}
n \cdot {}_{h_0}T_n = r \sum\limits_{k=0}^{n-1} \binom{h_0}{1} {}_{h_0} T_k R_{n-k} = h_0 \sum\limits_{k=0}^{n-1} {}_{h_0} T_k R_1^{n-k}
\end{equation}
Define the generating function \({}_{h_0}\mathrm{T}(t)=\sum_{n\ge 0} {}_{h_0}T_n\,t^n.\) The right‑hand side is a convolution of the sequence \(\{ {}_{h_0}T_k \}\) with the geometric series \(\{R_1^m\}_{m\ge1}\):
\[
\sum_{n\ge1}\left(\sum_{k=0}^{n-1}{}_{h_0}T_k R_1^{\,n-k}\right) t^n
= \left(\sum_{k\ge0}{}_{h_0}T_k t^k\right)\left(\sum_{m\ge1}(R_1t)^m\right)
= {}_{h_0}\mathrm{T}(t)\,\frac{R_1t}{1-R_1t}.
\]
Thus
\[
t\, \cdot  {}_{h_0}\mathrm{T}'(t)=h_0\, \cdot {}_{h_0}\mathrm{T}(t)\,\frac{R_1t}{1-R_1t}.
\]
Cancelling one factor \(t\) yields
\[
\frac{{}_{h_0}\mathrm{T}'(t)}{{}_{h_0}\mathrm{T}(t)}=\frac{h_0 R_1}{1-R_1t}.
\]
Integrating with respect to \(t\) yields
\[
\ln\left({}_{h_0}\mathrm{T}(t)\right) = -h_0\ln(1-R_1t) + C,
\]
where \(C\) is an integration constant. Exponentiating, setting \(t=0\) and using \({}_{h_0}\mathrm{T}(0)={}_{h_0}T_0\) gives
\[
{}_{h_0}\mathrm{T}(t)={}_{h_0}T_0\,(1-R_1t)^{-h_0}.
\]
Expanding by the generalized binomial theorem
\[
(1-R_1t)^{-h_0}=\sum_{n\ge0}\binom{h_0+n-1}{n}(R_1t)^n,
\]
we obtain the closed form
\[
{}_{h_0}T_n = {}_{h_0}T_0\binom{h_0+n-1}{n}\left(\frac{2}{\tau}\right)^{n}.
\]
Finally, we want to deal with
\begin{equation} \label{eq:app:simple_U_recurrence}
\begin{aligned}
(n-1+\mathsf{b}_1) \widehat{U}_{n,\mathsf{b}_1}
&= \frac{r}{2} \left( \sum_{l=1}^{n-1} \sum_{\mathsf{s}=0}^{\mathsf{b}_1} \widehat{U}_{n-l,\mathsf{b}_1-\mathsf{s}} \widehat{U}_{l,\mathsf{s}} \, \left(2 - 2(\mathsf{b}_1-\mathsf{s})\right) \left(2 - 2\mathsf{s}\right) \right. \\
&\qquad \left. + \widehat{U}_{n,\mathsf{b}_1-1} \left(2 - 2\mathsf{b}_1 + 1\right) \left(2 - 2\mathsf{b}_1 + 2\right) \right).
\end{aligned}
\end{equation}
Applying Lemma \ref{lem:correctness_U} with \(h=2\), we see that \(\widehat{U}_{n,\mathsf{b}_1}\) can be non‑zero only for \(\mathsf{b}_1 \le 1\). With \(\widehat{U}_{n,0}=\widehat{R}_n\), the only remaining unknown is \(\widehat{U}_{n,1}\). Substituting \(\mathsf{b}_1 = 1\) yields
\[
n \widehat{U}_{n,1} = \frac{r}{2} \left( \sum_{l=1}^{n-1} \sum_{\mathsf{s}=0}^{1} \widehat{U}_{n-l,1-\mathsf{s}} \widehat{U}_{l,\mathsf{s}} \, \left(2 - 2(1-\mathsf{s})\right) \left(2 - 2\mathsf{s}\right) + \widehat{U}_{n, 1-1} \left(2 - 2 + 1\right) \left(2 - 2 + 2\right) \right).
\]
\[
n \widehat{U}_{n,1} = r \widehat{R}_{n} \implies \widehat{U}_{n,1} = \frac{r}{n} \widehat{R}_{n} = \frac{1}{n} \widehat{R}_{n}.
\]

The results for \(h=2\) can also be derived by an alternative route. Even in this simplest case, the alternative approach not only provides an independent verification of our method but also reveals that the case \(h=2\) already yields non-trivial results and interesting interpretations.

We have derived that \(R_n = \left(\frac{2}{\tau}\right)^n \implies \widehat{R}_n = \frac{1}{2} \left(\frac{2}{\tau}\right)^{n}\). When \(\tau = 2\) (i.e., when orientation‑reversing automorphisms are allowed), this yields a single numbering of the two terminal vertices of the path graph. In contrast, if \(\tau = 1\), we consider the edges with a distinguished direction, and the count becomes \(R_n = 2^n\), since each of the \(n\) segments can be oriented independently.

As an example for the inductive triangulation, let the initial graph be the star graph \(S_q\) with $q \ge 2$, and let the automorphism group be the cyclic shifts (see Figure~\ref{fig:star_graph_induction} for an illustration). Under these assumptions, the formula reduces to
\[
{}_{q}T_n = {}_{q}T_0\binom{q+n-1}{n}\left(\frac{2}{\tau}\right)^{n} = \frac{q!}{|C_q|} \binom{q+n-1}{n} = \frac{(n+q-1)!}{n!}.
\]
\begin{figure}[!htbp]
    \centering
    \includegraphics[width=0.78\textwidth]{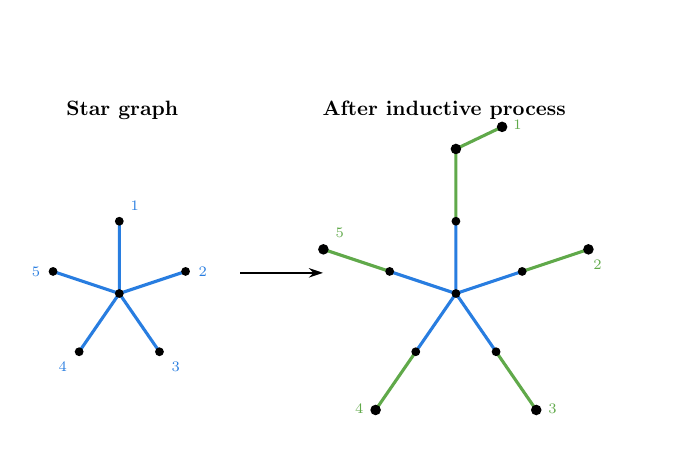}
    \caption{Attaching segments to a star graph: blue edges form the initial star, while green edges are added inductively.}
    \label{fig:star_graph_induction}
\end{figure}

We now show that this formula holds independently of our method. The resulting tree has \(q\) branches whose lengths we denote by \( a_1, a_2, \dots, a_q \ge 1 \). The total number of edges is \( n+q \), so \( \sum_{i=1}^q a_i = n+q \). This is equivalent to \( \sum_{i=1}^q a_i' = n \), where \( a_i' = a_i - 1 \) are nonnegative integers. The number of ordered nonnegative integer solutions is \(\binom{n+q-1}{q-1}\), which is obtained by placing \(q-1\) dividers among the \(n+q-1\) available positions (stars-and-bars). For each fixed ordered \(q\)-tuple \( (a_1,\dots,a_q) \), there are \( q! \) ways to label the \(q\) leaves with distinct labels \(1,\dots,q\). We identify labelings that differ by a cyclic permutation of the labels (group \( C_q \) of order \(q\)), so the number of equivalence classes per tuple is \(\frac{q!}{|C_q|} = (q-1)!\). Therefore, the total number of non‑isomorphic labelings is $(q-1)! \binom{n+q-1}{q-1} = \frac{(n+q-1)!}{n!}$.

For graphs consisting solely of a cycle (i.e., regular polygon, see Figure~\ref{fig:dihedral_group}), the automorphism group is the dihedral group of order \(2n\). This coincides with \(\widehat{U}_{n,1} = \frac{1}{n} \widehat{R}_{n} = \frac{1}{2n}.\)

\begin{figure}[!htbp]
    \centering
    \includegraphics[width=0.25\textwidth]{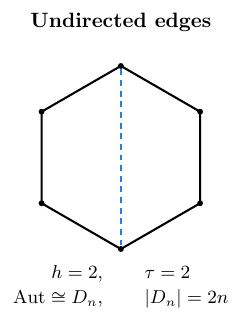}
    \caption{Symmetries of a polygon with undirected edges.}
    \label{fig:dihedral_group}
\end{figure}

\section*{Acknowledgements}

The work of Natalia Amburg was
funded by the RSF (Grant No. 24-12-00178). We are also grateful to the HSE Faculty of Mathematics for enabling this collaboration.

\section{Declaration of AI Usage}

The main ideas of this work were conceived in 2023–2024 as part of a bachelor's thesis, when AI was not yet powerful enough, and the authors did not use it for this task. Beginning in 2025, AI became more useful and assisted the authors with computations and literature search, and helped improve readability, structure, and English. The authors take full responsibility for the content.

Furthermore, an AI-assisted formalization of nearly the entire work was obtained in Lean 4. Some facts in this formalization are proved differently from the exposition in the paper. The formalization itself is available at \url{https://github.com/alialiev78/Polytope_Trees_Lean_Formalization}.

\printbibliography

\begin{appendices}
\appendixpage

\section{Closed-Form Generating Functions for Special Cases} \label{sec:explicit_expressions_generating_function}

We have already derived the series expansions for the functions \(\mathcal{R}(t)\) and \(_{h_0}\mathcal{T}(t)\). However, under the condition \(2 \leq h - 1 \leq 4\) --- where \(h\) represents the number of hyperfacets of codimension \(1\) --- these functions can also be expressed in closed form. Among the regular polytopes, only the $4$‑dimensional simplex, the pentagon, the square, and the triangle satisfy the condition. (The $3$‑dimensional simplex, i.e., the tetrahedron, has already been treated in Section~\ref{sec:R_surfaces}; the verification that all other regular polytopes violate the requirement \(2 \leq h - 1 \leq 4\) is left as an exercise for the reader.)

\begin{stat} \label{stat:generating_4_general}
Let \(h, r, \tau, h_0, y\) be quantities defined in \ref{thm:general_R}, \ref{thm:general_T}. In this statement, \(h = 5\).
We define the following quantities
\begin{align*}
    u = k^2 &= \sqrt[3]{\frac{1}{2} + \sqrt{\frac{1}{4} - \frac{64t^3}{27}}} + \sqrt[3]{\frac{1}{2} - \sqrt{\frac{1}{4} - \frac{64t^3}{27}}},\\
    \ell &= \left( k^2 - \frac{1}{k} \right) \frac{1}{2},\quad
    y = \frac{k - \sqrt{k^2 - 4\ell}}{2}.
\end{align*}
Then
\begin{equation}
\mathcal{R}'(t) = \frac{y - t}{r},\quad {}_{h_0} \mathcal{T}(t) = \frac{{}_{h_0} T_0}{h_0!} y^{h_0}.
\end{equation}
The generating function $\mathcal{R}(t)$ itself is equal to
\begin{equation*}
\begin{aligned}
\mathcal{R}(t) = \frac{1}{r} \left( \frac{y^5}{5} - \frac{(y-t)^2}{2} \right) = \frac{1}{r} \left( \frac{y^5}{5} - \frac{y^8}{2} \right).
\end{aligned}
\end{equation*}
The value of \(r\) is 12 for a four‑dimensional simplex and 1 for a pentagon.
\end{stat}

\begin{proof}
We need to solve the equation
\[
y^4 = y - t \implies y^4 - y + t = 0.
\]
We represent \(y^4 - y + t\) as a product \((y^2 + ky + m)(y^2 - ky + \ell)\). Then, using the method of undetermined coefficients, we obtain the following system
\[
\begin{cases}
m + \ell - k^2 = 0 \\
k(-m + \ell) = -1 \\
m\ell = t
\end{cases}
\iff
\qquad
\begin{cases}
m = \left( k^2 + \frac{1}{k} \right) \frac{1}{2} \\
\ell = \left( k^2 - \frac{1}{k} \right) \frac{1}{2} \\
m\ell = t
\end{cases}
;
\]
From this system we obtain \(k^4-\frac{1}{k^2} = 4t\), \(k^6-1 = 4tk^2\). Let us use Cardano's formula for \(k^2\)
\[
Q = \left( \frac{-1}{2} \right)^2 + \left( \frac{-4t}{3} \right)^3 = \frac{1}{4} - \frac{64t^3}{27}.
\]
It is sufficient for us to choose any root; let us take \(k^2 = \sqrt[3]{\frac{1}{2} + \sqrt{\frac{1}{4} - \frac{64t^3}{27}}} + \sqrt[3]{\frac{1}{2} - \sqrt{\frac{1}{4} - \frac{64t^3}{27}}}.\)
Note also that \(k^2(0) = \sqrt[3]{\frac{1}{2} + \sqrt{\frac{1}{4}}} + 0 = 1\). This means that \(k(0) = 1, m(0) = 1, \ell(0) = 0\). Taking this into account and the fact that \(0 =\mathcal{R}'(0) = \frac{2(y(0) - 0)}{r}\) we select the unique root from 4 candidates
\[
y_{1,2} = \frac{-k \pm \sqrt{k^2-4m}}{2}, y_{3,4} = \frac{k \pm \sqrt{k^2-4\ell}}{2}.
\]
Substituting \(0\): \(y_{1,2}(0) = \frac{-1 \pm \sqrt{1-4}}{2} \neq 0\), \(y_{3,4}(0) = \frac{1 \pm \sqrt{1-0}}{2}\). The one that suits us is \(y = y_4 = \frac{k - \sqrt{k^2-4\ell}}{2}\). From here we find \(\mathcal{R}'_{4}(t) = \frac{y-t}{r}\).

Now let us find $\mathcal{R}_{4}$. Instead of expressing the integrand in terms of $t$ (or $u$), we can easily compute the integral by changing the variable of integration to $y$. Since $t = y - y^4$, the differential is:
\[
\mathrm{d}t = (1 - 4y^3) \, \mathrm{d}y.
\]
Integrating with respect to $t$:
\begin{align*}
\mathcal{R}(t) &= \int \frac{y^4}{r} \, \mathrm{d}t = \int \frac{y^4}{r} (1 - 4y^3) \, \mathrm{d}y \\
&= \frac{1}{r} \int (y^4 - 4y^7) \, \mathrm{d}y \\
&= \frac{1}{r} \left( \frac{y^5}{5} - \frac{y^8}{2} \right) + C.
\end{align*}
To find the integration constant $C$, we evaluate the function at $t=0$. As shown earlier, $y(0) = 0$, so:
\[
0 = \mathcal{R}(0) = \frac{1}{r} \left( \frac{0^5}{5} - \frac{0^8}{2} \right) + C \implies C = 0.
\]
Finally, since $y^4 = y - t$, we can rewrite $y^8$ as $(y-t)^2$, yielding the final expression:
\[
\mathcal{R}(t) = \frac{1}{r} \left( \frac{y^5}{5} - \frac{(y-t)^2}{2} \right).
\]
It remains to substitute the values of \(r\). In the case of a four-dimensional simplex \(r = \frac{4!}{2} = 12\). In the case of a pentagon \(r = 1\).
\end{proof}

\begin{stat}
For the square we have
\[
y^3 - y + t = 0, \quad \mathcal{R}^{\mathrm{square}} = 3\left(\frac{y^4}{12} - \frac{y^6}{6}\right).
\]
\end{stat}

\begin{proof}
For the square we need to solve the equation \(y^3 - y + t = 0\). 
We already solved such an equation in Theorem~\ref{thm:R_surfaces_solution} and obtained the answer:
\[
y = \sqrt[3]{ -\frac{t}{2} + \sqrt{ \frac{t^2}{4} - \frac{1}{27} } } \;+\; \sqrt[3]{ -\frac{t}{2} - \sqrt{ \frac{t^2}{4} - \frac{1}{27}} }.
\]
In the case of the square only the following changes occur:
\[
(\mathcal{R}')^{\text{square}}(t) = \frac{y-t}{r} = \frac{3}{r}\cdot\frac{y-t}{3} = \frac{3}{r}\,\mathcal{R}'_3 
\;\Longrightarrow\; \mathcal{R}^{\text{square}} = \frac{3}{r}\left(\frac{y^4}{12} - \frac{y^6}{6}\right).
\]
For the square \(r = 1\). This completes the proof.
\end{proof}

The last polytope that remains to be considered is the triangle.
\begin{stat}
    \begin{equation}
    \mathcal{R}_{2}(t) = \frac{1}{2} \left(t - t^2 + \frac{(1 - 4t)^{3/2}}{6} \right) - \frac{1}{12}.
    \end{equation}
\end{stat}

\begin{proof}
We need to solve the equation for $p = \mathcal{R}'_2$
\[ 
p = (t+ p)^2 \iff p^2+2tp+t^2 = p \iff p^2+p(2t-1)+t^2 = 0.
\]
\[
\mathcal{R}'_2 = \frac{1-2t \pm \sqrt{4t^2-4t+1-4t^2}}{2} = \frac{1-2t \pm \sqrt{1-4t}}{2}.
\]
Since \(\mathcal{R}'_2(0) = 0\), we have \(\mathcal{R}'_2= \frac{1-2t - \sqrt{1-4t}}{2}\). Integrating, we get \[\mathcal{R}_2(t) = \frac{1}{2} \left(t - t^2 + \frac{(1 - 4t)^{3/2}}{6} \right) + C_1.\] Given that \(\mathcal{R}_2(0) = 0,\) it follows that \(C_1 = \frac{-1}{12}.\) Moreover, \(y = t+ \frac{2! \mathcal{R}'_2}{2} = t+\mathcal{R}'_2 = \frac{1 - \sqrt{1-4t}}{2}\).
\end{proof}

\section{Alternative Proofs of the Explicit Formulas for Regular Polytopes} \label{sec:another_way_R_T}

Substituting \(\widehat{R}_n = \frac{R_n}{\left(n(h-2)+2\right)!}\) into the recurrence relation  \eqref{eq:R_general_recurrence}, we obtain
\begin{equation} \label{eq:app:R_general_recurrence}
(n-1)\widehat{R}_n = \frac{r}{2} \sum\limits_{k=1}^{n-1} ((h-2)k+2) \widehat{R}_k ((h-2)(n-k)+2) \widehat{R}_{n-k}.
\end{equation}
For an arbitrary initial complex, a corresponding recurrence relation can be derived similarly
\begin{equation} \label{eq:app:eq:T_general_recurrence}
n \cdot {}_{h_0}\widehat{T}_n = r \sum\limits_{k=0}^{n-1} ((h-2)k+h_0) \cdot {}_{h_0}\widehat{T}_k ((h-2)(n-k)+2)  \widehat{R}_{n-k}.
\end{equation}
The explicit formulas can be derived using two approaches. The first approach was introduced and developed above. Another approach is to consider
\(\widetilde{\mathcal{R}}(t) = \sum\limits_{n=1}^{\infty} \widehat{R}_n t^{n}\) and
\( {}_{h_0} \widetilde{\mathcal{T}}(t) = \sum_{n=0}^{\infty} {}_{h_0} \widehat{T}_n t^{n} \). Let us also introduce an auxiliary sequence
\[
S_n = \left((h-2)n + 2\right) \widehat{R}_n, \quad S(t) = \sum_{n=1}^{\infty} S_n t^n.
\]
Then the relation \eqref{eq:app:R_general_recurrence} can be rewritten as
\[
(n-1) \frac{S_n}{(h-2)n + 2} = \frac{r}{2} \sum_{k=1}^{n-1} S_k S_{n-k}.
\]
Note that
\[
\sum_{n=2}^{\infty} (n-1) \widehat{R}_n t^n = t \widetilde{\mathcal{R}}'(t) - \widetilde{\mathcal{R}}(t)
\]
and
\[
\sum_{n=2}^{\infty} \left( \frac{r}{2} \sum_{k=1}^{n-1} S_k S_{n-k} \right) t^n = \frac{r}{2} (S(t))^2.
\]
Consequently, we obtain
\[
t \widetilde{\mathcal{R}}'(t) - \widetilde{\mathcal{R}}(t) = \frac{r}{2} (S(t))^2 = \frac{r}{2} \left((h-2)t \widetilde{\mathcal{R}}'(t) + 2 \widetilde{\mathcal{R}}(t)\right)^2.
\]
Thus, we have a system of two equations
\[
\begin{cases}
(h-2)t \widetilde{\mathcal{R}}'(t) + 2 \widetilde{\mathcal{R}}(t) = S(t), \\
t \widetilde{\mathcal{R}}'(t) - \widetilde{\mathcal{R}}(t) = \dfrac{r}{2} (S(t))^2.
\end{cases}
\]
Express \( t \widetilde{\mathcal{R}}'(t) \) from the second equation and substitute it into the first
\[
(h-2)\left(\widetilde{\mathcal{R}}(t) + \frac{r}{2} S^2(t)\right) + 2 \widetilde{\mathcal{R}}(t) = S(t).
\]
\begin{equation} \label{eq:app:y_R_S_connection}
h \widetilde{\mathcal{R}}(t) = S(t) - \frac{r(h-2)}{2} S^2(t) \implies
\widetilde{\mathcal{R}}(t) = \frac{S(t)}{h} - \frac{r(h-2)}{2h} S^2(t).
\end{equation}
Differentiate both sides of the last equality
\[
\widetilde{\mathcal{R}}'(t) = \left( \frac{1}{h} - \frac{r(h-2)}{h} S(t) \right) S'(t) \implies t \widetilde{\mathcal{R}}'(t) = t \left( \frac{1 - r(h-2)S(t)}{h} \right) S'(t).
\]
On the other hand, from the previous step
\begin{equation*}
t \widetilde{\mathcal{R}}'(t) = \widetilde{\mathcal{R}}(t) + \frac{r}{2} S(t)^2 = \left( \frac{S}{h} - \frac{r(h-2)}{2h} S^2 \right) + \frac{r}{2} S^2 = \frac{S + r S^2}{h}.    
\end{equation*}
Therefore
\[
t \cdot \frac{1 - r(h-2)S}{h} \cdot S' = \frac{S(1 + rS)}{h}.
\]
Multiply by \( h \), separate variables, and integrate
\[
\int \frac{1 - r(h-2)S}{S(1 + rS)} \, \mathrm{d}S = \int \frac{\mathrm{d}t}{t}.
\]
Since
\[
\frac{1 - r(h-2)S}{S(1 + rS)} = \frac{1}{S} - \frac{r(h-1)}{1 + rS},
\]
Then
\[
\ln|S| - (h-1)\ln|1 + rS| = \ln|t| + K_1.
\]
\[
K_2 t = \frac{S(t)}{(1 + r S(t))^{h-1}}.
\]
Since
\[
\widetilde{\mathcal{R}}(t) = \widehat{R}_1 t + O(t^2), \quad S(t) = S_1 t + O(t^2),
\]
we have
\[
S_1 t + O(t^2) = (h-2)t(\widehat{R}_1 + O(t)) + 2(\widehat{R}_1 t + O(t^2)) = h\widehat{R}_1 t + O(t^2) \implies S_1 = h\widehat{R}_1.
\]
By comparing the coefficients of \( t^1 \), we obtain
\[
K_2 t = (S_1 t + O(t^2)) \cdot (1 + r(S_1 t + O(t^2)))^{1-h}.
\]
Expand the second factor using the binomial theorem, keeping only terms up to \( O(t) \)
\[
K_2 t = (S_1 t + O(t^2)) \cdot (1 + O(t)) \implies K_2 = S_1 = h\widehat{R}_1.
\]
Thus
\begin{equation} \label{eq:app:another_way_equation}
t = \frac{S(t)}{h\widehat{R}_1(1 + r S(t))^{h-1}}. 
\end{equation}
Set \(\widetilde{y} = 1 + rS\). Then
\[
r h\widehat{R}_1 \, t = \frac{\widetilde{y} - 1}{\widetilde{y}^{\,h-1}} \implies \widetilde{y} = 1 + (r h\widehat{R}_1 t)\,\widetilde{y}^{\,h-1}.
\]  
This is exactly the equation in Corollary \ref{cor:lagrange_application1} with \(d = h-1\) and the independent variable replaced by \(r h\widehat{R}_1 t\). Hence  
\[
\widetilde{y}(t) = \sum_{n=0}^{\infty} \frac{1}{(h-2)n+1}\binom{(h-1)n}{n} (r h\widehat{R}_1 t)^n.
\]
Returning to the function $S$, we obtain
\[
S(t) = \sum_{n=1}^{\infty} \frac{t^n}{n!} (h\widehat{R}_1)^n \cdot r^{n-1} \cdot \frac{[n(h-1)]!}{[n(h-1) - n + 1]!}.
\]
Hence
\[
\widehat{R}_n = \frac{S_n}{(h-2)n + 2} = \frac{1}{r} (h\widehat{R}_1 r)^n \cdot \frac{[n(h-1)]!}{n! [n(h-2) + 2]!}.
\]
Since \( \frac{(h-1)!}{R_1} = \frac{\tau}{h} \), then \( \widehat{R}_1 = \frac{R_1}{((h-2) \cdot 1 + 2)!} = \frac{h!}{\tau \cdot h!} = \frac{1}{\tau} \). This result matches \eqref{eq:R_general_explicit_no_numeration}.

To compute an explicit formula for inductive complexes using the second method, we introduce
\[
{}_{h_0}J_n = ((h-2)n + h_0) \cdot {}_{h_0}\widehat{T}_n; \quad {}_{h_0}J(t) = \sum_{n=0}^{\infty} {}_{h_0}J_n t^n.
\]
Multiply both sides by \( t^n \) in \eqref{eq:app:eq:T_general_recurrence} and sum over \( n \geq 1 \)
\[
\sum_{n=1}^{\infty} n \cdot {}_{h_0}\widehat{T}_n t^n = r \sum_{n=1}^{\infty} \sum_{k=0}^{n-1} {}_{h_0}J_k S_{n-k} t^n.
\]
The left-hand side
\[
\sum_{n=0}^{\infty} n \cdot {}_{h_0}\widehat{T}_n t^n = \sum_{n=1}^{\infty} n \cdot {}_{h_0}\widehat{T}_n t^n = t \frac{\mathrm{d}}{\mathrm{d}t} \left( \sum_{n=0}^{\infty} {}_{h_0}\widehat{T}_n t^n \right) = t \frac{\mathrm{d}}{\mathrm{d}t} \left({}_{h_0}\widetilde{\mathcal{T}}(t)\right).
\]
The right-hand side equals
\[
r \cdot {}_{h_0} J(t) \cdot S(t) = r \left[ (h-2)t \cdot {}_{h_0}\widetilde{\mathcal{T}}'(t) + h_0 \cdot {}_{h_0}\widetilde{\mathcal{T}}(t) \right] S(t).
\]
Thus, we obtain
\[
t \cdot {}_{h_0}\widetilde{\mathcal{T}}'(t) = r \left[ (h-2)t \cdot {}_{h_0}\widetilde{\mathcal{T}}'(t) + h_0 \cdot {}_{h_0}\widetilde{\mathcal{T}}(t) \right] S(t).
\]
Transform
\[
\left[1 - r (h-2) S(t)\right] t \cdot {}_{h_0}\widetilde{\mathcal{T}}'(t) - r h_0 S(t) \cdot {}_{h_0} \widetilde{\mathcal{T}}(t) = 0.
\]
Separation of variables
\[
\frac{\mathrm{d} ({}_{h_0}\widetilde{\mathcal{T}}(t))}{{}_{h_0}\widetilde{\mathcal{T}}(t)} = \frac{r h_0 S(t)}{\left[1 - r (h-2) S(t)\right] t} \mathrm{d}t.
\]
Earlier, we obtained
\[
\frac{\mathrm{d}t}{t} = \frac{[1 - r (h-2) S]}{S (1 + r S)} \mathrm{d}S.
\]
Substituting into the integral
\[
\int \frac{S(t)}{\left[1 - r (h-2) S(t)\right] t} \mathrm{d}t = \int \frac{S}{ [1 - r (h-2) S]} \cdot \frac{[1 - r (h-2) S]}{S (1 + r S)} \mathrm{d}S = \int \frac{\mathrm{d}S}{1 + r S}.
\]
From this, \({}_{h_0}\widetilde{\mathcal{T}}(t)\)
\[
{}_{h_0}\widetilde{\mathcal{T}}(t) = \widetilde{K}_1 \exp\left( r h_0 \cdot \frac{1}{r} \ln|1 + r S(t)| \right) = \widetilde{K}_2 (1 + r S(t))^{h_0}.
\]
From the initial condition \({}_{h_0}\widetilde{\mathcal{T}}(0) = {}_{h_0}\widehat{T}_0\) we deduce $\widetilde{K}_2 = {}_{h_0}\widehat{T}_0$, i.e.
\begin{equation} \label{eq:app:y_T_connection}
{}_{h_0}\widetilde{\mathcal{T}}(t) = {}_{h_0}\widehat{T}_0 (1 + r S(t))^{h_0} = {}_{h_0}\widehat{T}_0 \widetilde{y}^{\,h_0},
\end{equation}
where \( \widetilde{y} = 1 + r S(t) \). Given \( S = \frac{\widetilde{y} - 1}{r} \), substitute into the equation \eqref{eq:app:another_way_equation}
\begin{equation} \label{eq:app:another_y_equation}
t = \frac{\frac{\widetilde{y} - 1}{r}}{h\widehat{R}_1 \widetilde{y}^{\,h-1}} = \frac{\widetilde{y} - 1}{h\widehat{R}_1 r \widetilde{y}^{\,h-1}}.  
\end{equation}
Let \( \phi(\widetilde{y}) = h\widehat{R}_1 r \widetilde{y}^{\,h-1} \), then
\[
t = \frac{\widetilde{y} - 1}{\phi(\widetilde{y})} \implies \widetilde{y} = 1 + t \phi(\widetilde{y}).
\]
Apply the Lagrange-Bürmann formula with shift \(1\) for \( Q(y) = \widetilde{y}^{\, h_0} \)
\[
Q(y) = Q(1) + \sum_{n=1}^{\infty} \frac{t^n}{n!} \left[ \frac{\mathrm{d}^{n-1}}{\mathrm{d} u^{n-1}} \left( Q'(u) \phi(u)^n \right) \right]_{u=1}.
\]
In our case
\[
Q(\widetilde{y}) = 1 + \sum_{n=1}^{\infty} \frac{t^n}{n!} \left[ \frac{\mathrm{d}^{n-1}}{\mathrm{d} u^{n-1}} \left( h_0 u^{h_0-1} \cdot (h\widehat{R}_1 r u^{h-1})^n \right) \right]_{u=1}.
\]
Compute the derivative
\begin{align*}
\frac{\mathrm{d}^{n-1}}{\mathrm{d} u^{n-1}} u^{h_0 - 1 + n(h-1)} = \left( \prod_{k=0}^{n-2} \left( h_0 - 1 + n(h-1) - k \right) \right) \cdot u^{\,h_0 - 1 + n(h-1) - (n-1)}.
\end{align*}
Given that we substitute \( u=1 \) and \({}_{h_0}\widetilde{\mathcal{T}}(t) = {}_{h_0}\widehat{T}_0 \widetilde{y}^{\, h_0} = {}_{h_0}\widehat{T}_0 Q(\widetilde{y}(t)) \), we obtain
\[
{}_{h_0}\widetilde{\mathcal{T}}(t) = {}_{h_0}\widehat{T}_0 + \sum_{n=1}^{\infty} {}_{h_0}\widehat{T}_0 \frac{t^n}{n!} h_0 (h\widehat{R}_1 r)^n \frac{(h_0 + n(h-1)-1)!}{(h_0 + n(h-1) - n)!}.
\]
This means that
\[
{}_{h_0}\widehat{T}_n = {}_{h_0}\widehat{T}_0 \frac{h_0}{n!} \left( \frac{h r}{\tau} \right)^{n} \cdot \frac{(h_0 + n(h-1) - 1)!}{(h_0 + n(h-2))!}.
\]
Given that \({}_{h_0}\widehat{T}_0 = \frac{{}_{h_0}T_0}{h_0!}\), our result coincides exactly with \eqref{eq:T_general_explicit_no_numeration}.

Note that also for the second method, equation \eqref{eq:app:another_y_equation} allows us to obtain a closed-form expression for the generating functions, analogous to the one derived for the first method in Appendix \ref{sec:explicit_expressions_generating_function}. Since \(\widehat{R}_1 = \frac{1}{\tau}\), for a regular polytope we have \( h \widehat{R}_1  r = \frac{hr}{\tau} = 1 \). Consequently, the equation \eqref{eq:app:another_y_equation} simplifies to the following form
\begin{equation} \label{eq:y_tilde_general}
t  = \frac{\widetilde{y} - 1}{\widetilde{y}^{\,h-1}}.
\end{equation}
Instead of solving this equation, we can reduce it to the known equation $\widehat{y}^{\, h-1} - \widehat{y} + \omega = 0$, as we did in Corollary \ref{cor:lagrange_application2}. Multiplying equation \eqref{eq:y_tilde_general} by $t^{\frac{1}{h-2}}$ yields
\[
t^{\frac{h-1}{h-2}}\widetilde{y}^{\, h-1} - t^{\frac{1}{h-2}}\widetilde{y} + t^{\frac{1}{h-2}} = 0.
\]
By introducing the substitutions \(\omega = t^{\frac{1}{h-2}}\) and \(\widehat{y} = \widetilde{y} \omega = \widetilde{y} t^{\frac{1}{h-2}}\), the terms group together to form:
\begin{equation} \label{eq:y_standard_reduced}
\widehat{y}^{\, h-1} - \widehat{y} + \omega = 0.
\end{equation}
We have already solved this equation under the condition \(\widehat{y}(0)=0\); now we verify its uniqueness. Consider the function \(F(\widehat{y},\omega) = \widehat{y}^{\, h-1} - \widehat{y} + \omega\). At the point \((0,0)\)
\[
F(0,0) = 0,\qquad \frac{\partial F}{\partial y}\bigg|_{(0,0)} = \left((h-1)\widehat{y}^{\, h-2} - 1\right)\big|_{\widehat{y}=0} = -1 \neq 0.
\]
Hence, by the implicit function theorem, there exists a unique analytic function \(\widehat{y}(\omega)\) in a neighbourhood of \(\omega=0\) satisfying \(\widehat{y}(0)=0\) and \(F(\widehat{y}(\omega),\omega)=0\).

Notice that since \(\widetilde{y}(0) = 1 + r S(0) = 1\) and by continuity, the initial condition \(\widetilde{y}(t) \to 1\) implies
\[
\lim_{\omega \to 0} \frac{\widehat{y}(\omega)}{\omega} = 1 \quad\Longrightarrow\quad \widehat{y}(\omega) \sim \omega.
\]
Hence \(\widehat{y}(0) = 0\), and we can apply the previous results.

In summary, we take the expressions for \(\widehat{y}\) from Section~\ref{sec:explicit_expressions_generating_function} in terms of \(\omega\), apply the substitutions $\widetilde{y} = \widehat{y}/\omega$ and \(\omega = t^{1/(h-2)}\), insert the result into \(\widetilde{y} = 1 + r S(t)\), and finally substitute the resulting expression into \eqref{eq:app:y_R_S_connection} and \eqref{eq:app:y_T_connection} together with the polytope-specific parameters \(h\) and \(r\). This yields closed-form generating functions for the degree indexing via the second method; the remaining details are omitted for brevity and left to the reader.

\section{Proof of Lemma \ref{lem:correctness_U}} \label{sec:correctness_U}

\begin{proof}
We  prove the lemma by double induction: first on $n$, and for each fixed $n$, by induction on $\mathsf{m}$.
Recall that
\[
S_1 = \frac{r}{2} (\mathsf{m}+1)(\mathsf{m}+2) \widehat{V}_{n, \mathsf{m}+2}, \quad S_2 = \frac{r}{2} \sum_{l=1}^{n-1} \sum_{\mathsf{m}_1 + \mathsf{m}_2 = \mathsf{m}+2} \mathsf{m}_1 \mathsf{m}_2 \, \widehat{V}_{n-l, \mathsf{m}_1} \, V_{l, \mathsf{m}_2}.
\]

\paragraph{Base case: $n = 1$.}
We prove by induction on $\mathsf{m} \le -1$ that $\widehat{V}_{1,\mathsf{m}} = 0$. Note that the sum $S_2$ is empty (since $n = 1$ and the upper limit for \(l\) is \(0\))

For $\mathsf{m} = -1$ or $\mathsf{m} = -2$, the term $S_1$ vanishes because $(\mathsf{m}+1)(\mathsf{m}+2) = 0$. Hence the recurrence gives
\[
\frac{h - \mathsf{m}}{2}\,\widehat{V}_{1,\mathsf{m}} = 0,
\]
and, as \(\frac{h - \mathsf{m}}{2} > \frac{h}{2} > 0\), it follows that \(\widehat{V}_{1,\mathsf{m}} = 0\) for $\mathsf{m} \in \{-1, -2\}$.

Now assume that the assertion has already been proved for all indices 
\(t\) with \(\mathsf{m}+1 \le t \le -1\). For the term \(S_1\), 
observe that \(\mathsf{m}+2\) lies in the range \(\mathsf{m}+1 \le \mathsf{m}+2 \le -1\). 
Thus the induction hypothesis applies and gives \(\widehat{V}_{1,\mathsf{m}+2} = 0\), 
so \(S_1 = 0\). Both terms vanish and since \((h - \mathsf{m})/2 > 0\), it follows that \(\widehat{V}_{1,\mathsf{m}} = 0\) for \(\mathsf{m} \le -1\).

\paragraph{Inductive step for $n > 1$.}
Assume that for all $k < n$ and all $y \le -1$, we have $\widehat{V}_{k,y} = 0$. We now prove, by induction on $\mathsf{m} \le -1$, that $\widehat{V}_{n,\mathsf{m}} = 0$.

For $\mathsf{m} = -1$ or $\mathsf{m} = -2$, we have $(\mathsf{m}+1)(\mathsf{m}+2) = 0$, so $S_1 = 0$. We claim $S_2 = 0$: indeed, for any term in the double sum, $\mathsf{m}_1 + \mathsf{m}_2 = \mathsf{m}+2 \le 1$. Thus at least one of \(\mathsf{m}_1, \mathsf{m}_2\) is \(\le 0\), since otherwise 
their sum would be at least \(2\). If $\mathsf{m}_i = 0$, the term vanishes due to the factor $\mathsf{m}_i$. If $\mathsf{m}_i \le -1$, then the corresponding factor $\widehat{V}_{k, \mathsf{m}_i}$ has $k < n$, so by the outer induction hypothesis it is zero. Hence $S_2 = 0$, and substituting $S_1 = S_2 = 0$ into the recurrence gives
\[
\frac{h n - \mathsf{m}}{2}\,\widehat{V}_{n,\mathsf{m}} = 0.
\]
The coefficient $\frac{h n - \mathsf{m}}{2} > 0$ (since $h \ge 2$, $n \ge 1$, and $\mathsf{m} \le -1$). Therefore $\widehat{V}_{n,\mathsf{m}} = 0$.

Now assume that the assertion has already been proved for  all indices \(t\) with \(\mathsf{m}+1 \le t \le -1\). For the current \(\mathsf{m} \le -3\), we again have \(S_2 = 0\): indeed, for any term in the double sum, \(\mathsf{m}_1 + \mathsf{m}_2 = \mathsf{m}+2 \le -1\). Thus at least one of \(\mathsf{m}_1, \mathsf{m}_2\) must be \(\le -1\) (since if both were \(\ge 0\), their sum would be \(\ge 0\)). The corresponding factor \(\widehat{V}_{k, \mathsf{m}_i}\) has \(k < n\), so by the outer induction hypothesis it is zero. Hence \(S_2 = 0\).

It remains to handle \(S_1 = \frac{r}{2}(\mathsf{m}+1)(\mathsf{m}+2)\,\widehat{V}_{n,\mathsf{m}+2}\). Since \(\mathsf{m}+2\) satisfies \(\mathsf{m}+1 \le \mathsf{m}+2 \le -1\), the inner induction hypothesis applies and gives \(\widehat{V}_{n,\mathsf{m}+2} = 0\). Consequently, \(S_1 = 0\) and therefore \(\widehat{V}_{n,\mathsf{m}} = 0\). This completes the induction step and the proof.
\end{proof}

\section{Maximum of \texorpdfstring{\(\widehat{W}_{n,\mathsf{m}}\)}{W(n,m)} for Fixed \texorpdfstring{\(n\)}{n}}
\label{sec:maximum_W}
To maximize \(\widehat{W}_{n,\mathsf{m}}\) for a fixed sufficiently large \(n\) we need to check when the value \(\frac{r (\mathsf{m}+2)(\mathsf{m}+1)}{h n - \mathsf{m}}\) in the recurrence \eqref{eq:W_main_recurrence} is greater than 1. The point of monotonicity change \(\mathsf{m}^*\) is the root of the equation
\[
\frac{h n - \mathsf{m}}{r (\mathsf{m}+1)(\mathsf{m}+2)} = 1 \quad \Longleftrightarrow \quad r \mathsf{m}^2 + (3r + 1)\mathsf{m} + 2r - h n = 0.
\]
The positive solution is
\[
\mathsf{m}^* = \frac{-(3r + 1) + \sqrt{r^2 + 6r + 1 + 4 h r n}}{2r}.
\]
Now we have several options for the maximum value:

1. \(\mathsf{m} = 0\) (lower bound),

2. \(\mathsf{m} = h n\) (upper bound),

3. If \(\mathsf{m}^*\) is a natural number, then $\widehat{W}_{n,\mathsf{m}^*} = \widehat{W}_{n,\mathsf{m}^*+2}$, so the maximum is attained at the two consecutive admissible values $\mathsf{m}^*$ and $\mathsf{m}^*+2$. Otherwise, we define \(\mathsf{m}_1\) as the largest admissible \(\mathsf{m} \le \mathsf{m}^*\)  and set \(\mathsf{m}_2 = \mathsf{m}_1 + 2\) as the next admissible value. Consider
\[
\frac{\widehat{W}_{n, h n}}{\widehat{W}_{n, h n - 2}} = \frac{2}{r \cdot h n (h n - 1)}.
\]
This ratio is less than 1. Therefore, this point cannot be a maximum. Next, we analyze
\[
\frac{\widehat{W}_{n, 2}}{\widehat{W}_{n, 0}} = \frac{h n}{2 r}.
\]
If \(h n < 2 r\), then \(\widehat{W}_{n, 2} < \widehat{W}_{n, 0}\). However, for large \(n\), this is not true. Thus, the maximizers are \(\mathsf{m}^*\) and $\mathsf{m}^*+2$ if $\mathsf{m}^*$ is a natural number; otherwise, the maximum is attained either at \(\mathsf{m}_1\) or at \(\mathsf{m}_2\). The choice between \(\mathsf{m}_1\) and \(\mathsf{m}_2\) as the maximizer depends on the parameters.

\begin{rem}
The case of odd $\mathsf{m}$ is treated analogously, and we leave it to the reader.
\end{rem}

\section{Proof of the Implication \texorpdfstring{\(\eqref{eq:T_general_combinatorial} \Rightarrow \eqref{eq:R_general_combinatorial}\)}{T implies R}} \label{sec:T_implies_R}
\begin{proof}
Substituting $h_0 = 1$ into \eqref{eq:T_general_combinatorial} and simplifying, we obtain
$$1 = \frac{1}{n} \sum\limits_{k=0}^{n-1} \binom{(h-2)n+1}{(h-2)(n-k)+1} \frac{\binom{n}{k}}{\binom{n(h-1)}{k(h-1)}}$$

Multiplying both sides by $n$ and evaluating the term for $k=0$ (which exactly equals $1$), we can subtract it from both sides to get
\begin{equation} \label{eq:n-1_tmp1}
n-1 = \sum\limits_{k=1}^{n-1} \binom{(h-2)n+1}{(h-2)(n-k)+1} \frac{\binom{n}{k}}{\binom{n(h-1)}{k(h-1)}}   
\end{equation}
Replacing \(k\) by \(n-k\) in the summation.
\begin{equation} \label{eq:n-1_tmp2}
n-1 = \sum\limits_{k=1}^{n-1} \binom{(h-2)n+1}{(h-2)k+1} \frac{\binom{n}{k}}{\binom{n(h-1)}{k(h-1)}}.
\end{equation}
By the symmetry of binomial coefficients $\binom{A}{B} = \binom{A}{A-B}$, we can also rewrite \eqref{eq:n-1_tmp1}
\begin{equation} \label{eq:n-1_tmp3}
n-1 = \sum\limits_{k=1}^{n-1} \binom{(h-2)n+1}{(h-2)k} \frac{\binom{n}{k}}{\binom{n(h-1)}{k(h-1)}}.
\end{equation}
Adding \eqref{eq:n-1_tmp2} and \eqref{eq:n-1_tmp3} and applying Pascal's identity $\binom{N}{K} + \binom{N}{K-1} = \binom{N+1}{K}$, we finally deduce
$$2(n-1) = \sum\limits_{k=1}^{n-1} \binom{(h-2)n+2}{(h-2)k+1} \frac{\binom{n}{k}}{\binom{n(h-1)}{k(h-1)}}.$$   
\end{proof}

\end{appendices}

\end{document}